\pdfoutput=1
\documentclass[12pt]{amsart}
\usepackage{amsmath,amssymb,amsthm,tikz,tikz-cd,color,xcolor,ytableau,mathdots,amscd}
\usepackage[all]{xy}
\usepackage{fullpage}
\usepackage[protrusion=true,expansion=true]{microtype}
\usepackage{hyperref}
\usepackage{placeins}
\usepackage{enumitem}
\usepackage{booktabs}
\usepackage[T1]{fontenc}  
\usetikzlibrary{shapes,arrows,svg.path}
\newtheorem{theorem}{\textbf{Theorem}}[section]
\newtheorem{proposition}[theorem]{\textbf{Proposition}}
\newtheorem{corollary}[theorem]{\textbf{Corollary}}
\newtheorem{lemma}[theorem]{\textbf{Lemma}}

\newtheorem{circling}[theorem]{\textbf{Circling Rule}}
\theoremstyle{definition}

\newtheorem{remark}[theorem]{\textbf{Remark}}

\numberwithin{equation}{section}

\newcommand{\red}{\color{red}}
\newcommand{\blue}{\color{blue}}

\definecolor{darkgreen}{RGB}{0,180,0}
\colorlet[named]{green}{darkgreen}

\newcommand{\wt}{\operatorname{wt}}
\newcommand{\GL}{\operatorname{GL}}

\newcommand{\GTP}{\operatorname{GTP}}
\newcommand{\bzl}{\operatorname{string}}
\newcommand{\Pind}{\Pi_{\operatorname{nd}}}
\newcommand{\key}{\operatorname{\bold{key}}}

\hypersetup{
    pdftitle={ Colored five-vertex models and Demazure atoms },
    pdfauthor={Ben Brubaker, Valentin Buciumas, Daniel Bump, Henrik P. A. Gustafsson},
    linktocpage=true,           
    colorlinks=true,            
    linkcolor=blue!75!black,    
    citecolor=green!50!black,   
    filecolor=magenta,          
    urlcolor=blue!75!black      
}

\begin{document}
\title{Colored five-vertex models and Demazure Atoms}
\author{Ben Brubaker}
\address{School of Mathematics, University of Minnesota, Minneapolis, MN 55455}
\email{brubaker@math.umn.edu}
\author{Valentin Buciumas}
\address{School of Mathematics and Physics, 
The University of Queensland, 
St. Lucia, QLD, 4072, 
Australia}
\email{valentin.buciumas@gmail.com}
\author{Daniel Bump}
\address{Department of Mathematics, Stanford University, Stanford, CA 94305-2125}
\email{bump@math.stanford.edu}
\author{Henrik P. A. Gustafsson}
\address{\hspace{-\parindent}\textnormal{Until September 15, 2019:} \newline
  \indent Department of Mathematics, Stanford University, Stanford, CA 94305-2125. \newline 
  \textnormal{Since September 16, 2019:} \newline
  \indent School of Mathematics, Institute for Advanced Study, Princeton, NJ~08540. \newline 
  \indent Department of Mathematics, Rutgers University, Piscataway, NJ~08854. \newline 
  \indent Department of Mathematical Sciences, University of Gothenburg and Chalmers University of Technology, SE-412~96 Gothenburg, Sweden.}
\email{gustafsson@ias.edu}
\maketitle

\newcommand{\gammaice}[4]{\begin{tikzpicture}
\coordinate (a) at (-.75, 0);
\coordinate (b) at (0, .75);
\coordinate (c) at (.75, 0);
\coordinate (d) at (0, -.75);
\coordinate (aa) at (-.75,.5);
\coordinate (cc) at (.75,.5);
\draw (a)--(c);
\draw (b)--(d);
\draw[fill=white] (a) circle (.25);
\draw[fill=white] (b) circle (.25);
\draw[fill=white] (c) circle (.25);
\draw[fill=white] (d) circle (.25);
\node at (0,1) { };
\node at (a) {$#1$};
\node at (b) {$#2$};
\node at (c) {$#3$};
\node at (d) {$#4$};
\path[fill=white] (0,0) circle (.2);
\node at (0,0) {$z_i$};
\end{tikzpicture}}

\newcommand{\gammaa}{\begin{tikzpicture}
\coordinate (a) at (-.75, 0);
\coordinate (b) at (0, .75);
\coordinate (c) at (.75, 0);
\coordinate (d) at (0, -.75);
\coordinate (aa) at (-.75,.5);
\coordinate (cc) at (.75,.5);
\draw (a)--(c);
\draw (b)--(d);
\draw[fill=white] (a) circle (.25);
\draw[fill=white] (b) circle (.25);
\draw[fill=white] (c) circle (.25);
\draw[fill=white] (d) circle (.25);
\node at (0,1) { };
\node at (a) {$+$};
\node at (b) {$+$};
\node at (c) {$+$};
\node at (d) {$+$};
\path[fill=white] (0,0) circle (.2);
\node at (0,0) {$z_i$};
\end{tikzpicture}}

\newcommand{\gammaA}[8]{\begin{tikzpicture}
\coordinate (a) at (-.75, 0);
\coordinate (b) at (0, .75);
\coordinate (c) at (.75, 0);
\coordinate (d) at (0, -.75);
\coordinate (aa) at (-.75,.5);
\coordinate (cc) at (.75,.5);
\draw[line width=0.5mm, #1] (a)--(0,0);
\draw[line width=0.6mm, #2] (b)--(0,0);
\draw[line width=0.5mm, #3] (c)--(0,0);
\draw[line width=0.6mm, #4] (d)--(0,0);
\draw[line width=0.5mm, #1,fill=white] (a) circle (.25);
\draw[line width=0.5mm, #2,fill=white] (b) circle (.25);
\draw[line width=0.5mm, #3,fill=white] (c) circle (.25);
\draw[line width=0.5mm, #4,fill=white] (d) circle (.25);
\node at (0,1) { };
\node at (a) {$#5$};
\node at (b) {$#6$};
\node at (c) {$#7$};
\node at (d) {$#8$};
\path[fill=white] (0,0) circle (.2);
\node at (0,0) {$z_i$};
\end{tikzpicture}}

\newcommand{\gammaAA}[4]{\begin{tikzpicture}
\coordinate (a) at (-.75, 0);
\coordinate (b) at (0, .75);
\coordinate (c) at (.75, 0);
\coordinate (d) at (0, -.75);
\coordinate (aa) at (-.75,.5);
\coordinate (cc) at (.75,.5);
\draw (a)--(c);
\draw (b)--(d);
\draw[line width=0.5mm, #1,fill=white] (a) circle (.25);
\draw[line width=0.5mm, #2,fill=white] (b) circle (.25);
\draw[line width=0.5mm, #2,fill=white] (c) circle (.25);
\draw[line width=0.5mm, #1,fill=white] (d) circle (.25);
\node at (0,1) { };
\node at (a) {$#3$};
\node at (b) {$#4$};
\node at (c) {$#4$};
\node at (d) {$#3$};
\path[fill=white] (0,0) circle (.2);
\node at (0,0) {$z_i$};
\draw [line width=0.5mm,#1] (-.5,0) to [out=0,in=90] (0,-.5);
\draw [line width=0.5mm,#2] (0,.5) to [out=-90,in=180] (.5,0);
\end{tikzpicture}}

\newcommand{\gammab}[2]{\begin{tikzpicture}
\coordinate (a) at (-.75, 0);
\coordinate (b) at (0, .75);
\coordinate (c) at (.75, 0);
\coordinate (d) at (0, -.75);
\coordinate (aa) at (-.75,.5);
\coordinate (cc) at (.75,.5);
\draw (a)--(c);
\draw [line width=0.6mm, #1] (b)--(d);
\draw[fill=white] (a) circle (.25);
\draw[line width=0.5mm, #1, fill=white] (b) circle (.25);
\draw[fill=white] (c) circle (.25);
\draw[line width=0.5mm, #1, fill=white] (d) circle (.25);
\node at (0,1) { };
\node at (a) {$+$};
\node at (b) {$#2$};
\node at (c) {$+$};
\node at (d) {$#2$};
\path[fill=white] (0,0) circle (.2);
\node at (0,0) {$z_i$};
\end{tikzpicture}}

\newcommand{\gammaB}[2]{\begin{tikzpicture}
\coordinate (a) at (-.75, 0);
\coordinate (b) at (0, .75);
\coordinate (c) at (.75, 0);
\coordinate (d) at (0, -.75);
\coordinate (aa) at (-.75,.5);
\coordinate (cc) at (.75,.5);
\draw [line width=0.5mm,#1] (a)--(c);
\draw (b)--(d);
\draw[line width=0.5mm, #1, fill=white] (a) circle (.25);
\draw[fill=white] (b) circle (.25);
\draw[line width=0.5mm, #1, fill=white] (c) circle (.25);
\draw[fill=white] (d) circle (.25);
\node at (0,1) { };
\node at (a) {$#2$};
\node at (b) {$+$};
\node at (c) {$#2$};
\node at (d) {$+$};
\path[fill=white] (0,0) circle (.2);
\node at (0,0) {$z_i$};
\end{tikzpicture}}

\newcommand{\gammac}[2]{\begin{tikzpicture}
\coordinate (a) at (-.75, 0);
\coordinate (b) at (0, .75);
\coordinate (c) at (.75, 0);
\coordinate (d) at (0, -.75);
\coordinate (aa) at (-.75,.5);
\coordinate (cc) at (.75,.5);
\draw (a)--(c);
\draw (b)--(d);
\draw[line width=0.5mm, #1, fill=white] (a) circle (.25);
\draw[fill=white] (b) circle (.25);
\draw[fill=white] (c) circle (.25);
\draw[line width=0.5mm, #1, fill=white] (d) circle (.25);
\node at (0,1) { };
\node at (a) {$#2$};
\node at (b) {$+$};
\node at (c) {$+$};
\node at (d) {$#2$};
\path[fill=white] (0,0) circle (.2);
\node at (0,0) {$z_i$};
\draw [line width=0.5mm,#1] (-.5,0) to [out=0,in=90] (0,-.5);
\end{tikzpicture}}

\newcommand{\gammaC}[2]{\begin{tikzpicture}
\coordinate (a) at (-.75, 0);
\coordinate (b) at (0, .75);
\coordinate (c) at (.75, 0);
\coordinate (d) at (0, -.75);
\coordinate (aa) at (-.75,.5);
\coordinate (cc) at (.75,.5);
\draw (a)--(c);
\draw (b)--(d);
\draw[fill=white] (a) circle (.25);
\draw[line width=0.5mm, #1, fill=white] (b) circle (.25);
\draw[line width=0.5mm, #1, fill=white] (c) circle (.25);
\draw[fill=white] (d) circle (.25);
\node at (0,1) { };
\node at (a) {$+$};
\node at (b) {$#2$};
\node at (c) {$#2$};
\node at (d) {$+$};
\path[fill=white] (0,0) circle (.2);
\node at (0,0) {$z_i$};
\draw [line width=0.5mm,#1] (0,.5) to [out=-90,in=180] (.5,0);
\end{tikzpicture}}

\begin{abstract}
  Type A Demazure atoms are pieces of Schur functions, or sets of tableaux
  whose weights sum to such functions. Inspired by colored vertex models of
  Borodin and Wheeler, we will construct solvable lattice models whose
  partition functions are Demazure atoms; the proof of this makes use of a
  Yang-Baxter equation for a colored five-vertex model.  As a biproduct, we
  will construct Demazure atoms on Kashiwara's $\mathcal{B}_\infty$ crystal
  and give new algorithms for computing Lascoux-Sch\"utzenberger keys.
\end{abstract}

\section{Introduction}
\label{sec:intro}

Exactly solvable lattice models have found numerous applications in the study of special functions. (See \cite{KorepinBogoliebovIzergin,
KuperbergASM, LascouxSix, LLTFlag, GorbounovKorff, GorbounovKorffStroppel, WheelerZJGrothendieck, KnutsonZGPuzzles} to name but a few.) Here we use 
the Gelfand school interpretation of ``special function,'' meaning one that arises as a matrix coefficient of a group representation. 
If the group is a complex Lie group or a $p$-adic reductive group, these matrix coefficients include highest weight characters and in 
particular, Schur polynomials, as well as Demazure characters and various specializations and limits of Macdonald polynomials. 
Many of these special functions may be studied by methods originating in statistical mechanics, by expressing them as a multivariate 
generating function (the ``partition function'') over the admissible states of a solvable lattice model. The term ``solvable'' means that 
the model possesses a solution of the Yang-Baxter equation that often permits one to express the partition function of the model 
in closed form. Knowing that a special function is expressible as a partition function of a solvable lattice model then leads to a host 
of interesting combinatorial properties, including branching rules, exchange relations under Hecke operators, Pieri- and Cauchy-type 
identities, and functional equations. 

We will concentrate on the five- and six-vertex models on a square lattice, two-dimensional lattice models with five (respectively, six) 
admissible configurations on the edges adjacent to any vertex in the lattice. The latter models are sometimes referred to as ``square ice'' 
models, as the six configurations index the ways in which hydrogen atoms may be placed on two of the four edges adjacent to an 
oxygen atom at each vertex. Then weights for each configuration may be chosen so that the partition function records the probability 
that water molecules are arranged in any given way on the lattice (see for example \cite{Baxter}). More recently, lattice models with different weighting schemes 
have been studied in relation with certain stochastic models like the Asymmetric Simple Exclusion Process (ASEP) or the 
Kardar-Parisi-Zhang (KPZ) stochastic partial differential equation. These were shown to be part of a large family of solvable 
lattice models, called stochastic higher spin six-vertex models in \cite{BorodinAIM, CorwinPetrov}. Solutions to the Yang-Baxter 
equation also arise naturally from R-matrices of quantum groups; these higher spin models were associated to 
R-matrices for $U_q(\widehat{\mathfrak{sl}}_2)$. In this paper, we only make use of the associated quantum groups to differentiate among
the various lattice model weighting schemes and the resulting solutions to the Yang-Baxter equations.

Subsequently, Borodin and Wheeler \cite{BorodinWheelerColored} introduced
generalizations of the above
models, which they call colored lattice models.
Antecedents to these colored models appeared earlier in
\cite{BorodinBufetovWheeler, FodaWheeler}.
(A different notion of ``colored'' models appears in many other
works such as~\cite{AkutsuDeguchiOhtsuki}.)
In~\cite{BorodinWheelerColored}, ``colors'' are additional attributes
introduced to the boundary data and internal edges of a given model,
corresponding to replacing the governing quantum group
$U_q(\widehat{\mathfrak{sl}}_2)$ in the setting mentioned above by
$U_q(\widehat{\mathfrak{sl}}_{r+1})$. The partition functions of their colored
lattice models are non-symmetric \emph{spin} Hall-Littlewood
polynomials. These are functions depending on a parameter $s$, which recover
non-symmetric Hall-Littlewood polynomials when one sets $s=0$.

The idea of introducing ``color'' in this way may be applied to a wide variety of lattice models. If one chooses the Boltzmann 
weights for the colored models appropriately, then one obtains a refinement of the (uncolored) partition function as
a sum of partition functions indexed by all permutations of colors. Moreover, if the resulting colored model is solvable, 
then similar applications to those described above will follow. For example in~\cite{BorodinWheelerColored}, properties 
for these generalizations of Hall-Littlewood polynomials are proved including branching rules, exchange relations under 
Hecke divided-difference operators and Cauchy type identities motivated by the study of multi-species versions of the ASEP.

Inspired by these ideas of Borodin and Wheeler, this paper studies colored
versions of an (uncolored) five-vertex model whose partition function is (up
to a constant) a Schur polynomial $s_\lambda$ indexed by a partition
$\lambda$. The states of the uncolored system are in bijection with the set of semi-standard
Young tableaux of shape $\lambda$, so the above closed form of the partition function is a reformulation of the classical
combinatorial definition of the Schur function.
This uncolored five-vertex model is a degeneration (crystal limit)
of a six-vertex model described in Hamel and King~\cite{HamelKing},
that is similarly equivalent to the generalization of the combinatorial
definition of the Schur function by Tokuyama~\cite{Tokuyama}. These
models were shown to be solvable by Brubaker, Bump and Friedberg~\cite{hkice}.
See Section~\ref{sec:kansas} for the full definition of the uncolored five-vertex model used in this paper.

In Section \ref{sec:oz} we introduce our colored five-vertex model. A color is assigned to each of the $r$ rows of its
rectangular lattice and permuting these colors gives a system for
each element of the symmetric group $S_r$. We find Boltzmann weights for the
colored models that simultaneously refine the uncolored model and produce a
(colored) Yang-Baxter equation associated to a quantum superalgebra (see Theorem~\ref{coloredybethm}). This allows us to evaluate the partition
functions for the colored models for each $w \in S_r$ and prove in
Theorem~\ref{zdematoms} that they are \textit{Demazure atoms}.

Demazure atoms, introduced by Lascoux and
Sch\"{u}tzenberger~\cite{LascouxSchutzenbergerKeys} and referred to as
``standard bases'' there, decompose Demazure characters into their smallest
non-intersecting pieces. So in particular, summing Demazure atoms over a
Bruhat interval produces Demazure characters. Mason~\cite{MasonAtoms} coined
the term ``atoms'' and showed that they are specializations of non-symmetric
Macdonald polynomials of Cartan type~A with $q=t=0$. Basic properties of Demazure
atoms and characters are reviewed in Section~\ref{sec:do}.

Demazure characters and Schur polynomials may be viewed as
polynomial functions in formal variables or as functions on an algebraic torus
associated to a given reductive group. But they may also be lifted to subsets
of the Kashiwara-Nakashima~\cite{KashiwaraNakashima} crystal
$\mathcal{B}_\lambda$ whose elements are semistandard Young tableaux of a
given shape $\lambda$, called \textit{Demazure crystals}.  The existence of
such a lift of Demazure modules to crystals was shown by
Littelmann~\cite{LittelmannYT} and Kashiwara~\cite{KashiwaraDemazure}.
Summing the weights of the Demazure crystal recovers the Demazure
character.

Just as Littelmann and Kashiwara lifted Demazure characters
to the crystal, polynomial Demazure atoms may also be lifted to subsets
of the crystal. We will call these sets \textit{crystal Demazure
atoms}. Summing the weights of the crystal Demazure atom, one obtains the
usual polynomial Demazure atom. Crystals and the refined Demazure
character formula are briefly reviewed in Section~\ref{dcandci}.

Although the theory of Demazure characters and crystals is in place for all
Cartan types, most of the literature concerning Demazure atoms and the related
topic of Lascoux-Sch\"utzenberger keys is for Cartan Type~A. However the
$\overline{B}_{w\lambda}(\lambda)$ in Section~9.1 of
in~\cite{KashiwaraCristallines} are Demazure atoms for an arbitrary Kac-Moody
Cartan Type. Moreover recently~\cite{JaconLecouvey} (using the results
in~\cite{KashiwaraCristallines}) defined keys for all Kac-Moody Cartan types,
with a special emphasis on affine Type~A. There are also Type~C results
in~\cite{Santos}. See \cite{HershLenart,AssafSchilling} for other recent
work on Demazure atoms.

Based on Theorem~\ref{zdematoms}, which shows that the partition functions of our colored models are Demazure atoms, it is natural to ask for a
more refined version of the connection between colored ice and the crystal
Demazure atoms.
In Section~\ref{bijsec}, we accomplish this by exhibiting a bijection between
the admissible states of colored ice and crystal Demazure atoms as a subset of an
associated crystal $\mathcal{B}_\lambda$.  Showing this refined bijection is
much more difficult than the initial evaluation of the partition function. Its
proof forms a major part of this paper and builds on Theorem~\ref{daform},
which gives an algorithmic description of Demazure atoms. This result
is proved in Section~\ref{sec:dafpro} after introducing Kashiwara's
$\mathcal{B}_\infty$ crystal in Section~\ref{sec:binf}. As a biproduct of our
arguments, we will also obtain a theory of Demazure atoms on
$\mathcal{B}_\infty$.  The proofs take input from both the colored ice model
and the Yang-Baxter equation, and from crystal base theory, particularly
Kashiwara's $\star$-involution of~$\mathcal{B}_\infty$.

Another biproduct of the results in Section \ref{bijsec} is a new formula for \textit{Lascoux-Sch\"utzenberger keys}.
These are tableaux with the defining property that each column (except the first) is a subset
of the column before it. What is most important is that each crystal Demazure atom
contains a unique key. Thus if $T\in\mathcal{B}_\lambda$ there is a unique key
$\key(T)$ that is in the same crystal Demazure atom as $T$; this is called
the \textit{right key} of $T$. We will review this theory in
Subsection~\ref{subsec:alg}. Algorithms for computing $\key(T)$ may be found
in
{\cite{LascouxSchutzenbergerKeys,ReinerShimozono,LenartUnified,MasonAtoms,MasonRSK,WillisThesis,ProctorWillis,WillisKey,WillisFilling,AvalKeys,Seaborn}}.
In this paper we give a new algorithm for computing the
Lascoux-Sch\"utzenberger right key of a tableau in a highest weight
crystal. Since this algorithm may be of independent interest we will describe
it (and the topic of Lascoux-Sch\"utzenberger keys) in this introduction, in
Subsection~\ref{subsec:alg} below. We prove the algorithm in
Section \ref{sec:algo}.

This paper also serves as a stepping stone to colored versions of the
six-vertex (or ``ice'' type) models of~\cite{hkice} and of~\cite{BBB}. Indeed, since the
results of this paper, we have shown that analogous colored partition functions recover special values
of Iwahori fixed vectors in Whittaker models for general linear groups over a
$p$-adic field \cite{BBBGIwahori} and their metaplectic covers (in progress), respectively.
The colored five-vertex model in this paper is a degeneration of these models.

\subsection{\label{subsec:alg}Lascoux-Sch\"utzenberger keys}
Type~A Demazure atoms are pieces of Schur functions: if $\lambda$ is a
partition of length $\leqslant r$, the Schur function $s_{\lambda} (z_1,
\cdots, z_r)$ can be decomposed into a sum, over the Weyl group $W = S_r$, of
such atoms. This is an outgrowth of the Demazure character formula: if
$\partial_w$ is the Demazure operator defined later in Section~\ref{sec:do} then
$\partial_w \mathbf{z}^{\lambda}$ is called a \textit{Demazure character}. Originally
these were introduced by Demazure~{\cite{Demazure}} to study the cohomology of
line bundles on flag and Schubert varieties. A variant represents the Demazure
character as $\sum_{y \leqslant w} \partial^\circ_y \mathbf{z}^{\lambda}$
where $\partial^\circ_y$ are modified operators, and $y \leqslant w$ is
the Bruhat order. The components $\partial^\circ_y \mathbf{z}^{\lambda}$
are called (polynomial) {\textit{Demazure atoms}}.

As we will explain in Section~\ref{sec:oz}, a state of the colored lattice
model features $r$ colored lines running through a grid moving downward and rightward. These can cross, but
they are allowed to cross at most once. Each line intersects the boundary of
the grid in two places, and the colors are permuted depending on which lines
cross. Hence they determine a permutation $w$ from this braiding, which can be
encoded into the boundary conditions.  This allows us to construct a system
$\mathfrak{S}_{\mathbf{z}, \lambda, w}$ whose partition function satisfies the identity
\begin{equation}
  \label{partdc} Z (\mathfrak{S}_{\mathbf{z}, \lambda, w})
  =\mathbf{z}^{\rho} \partial^\circ_w \mathbf{z}^{\lambda},
\end{equation}
where $\rho$ is the Weyl vector. Here the polynomial
$\partial^\circ_w \mathbf{z}^{\lambda}$ is the Demazure atom.

The Schur function $s_{\lambda}$ is the character of the
Kashiwara-Nakashima~{\cite{KashiwaraNakashima}} crystal $\mathcal{B}_{\lambda}$ of tableaux. The Demazure character formula was lifted by
Littelmann~{\cite{LittelmannYT}} and Kashiwara~{\cite{KashiwaraDemazure}} to
define subsets $\mathcal{B}_{\lambda}(w) \subseteq \mathcal{B}_\lambda$ whose characters are Demazure
characters $\partial_w \mathbf{z}^{\lambda}$. If $w = 1_W$ then
$\mathcal{B}_{\lambda} (w) = \{ v_{\lambda} \}$ where $v_{\lambda}$ is the
highest weight element. If $w_0$ is the long element then
$\mathcal{B}_{\lambda} (w_0) =\mathcal{B}_{\lambda}$. If $w \leqslant w'$ in
the Bruhat order then $\mathcal{B}_{\lambda} (w) \subseteq
\mathcal{B}_{\lambda} (w')$.

In type~A, the results of Lascoux and
Sch{\"u}tzenberger~{\cite{LascouxSchutzenbergerKeys}} give
an alternative decomposition of $\mathcal{B}_{\lambda}$ into
disjoint subsets that we will here denote
$\mathcal{B}_{\lambda}^{\circ} (w)$. Then
\[ \mathcal{B}_{\lambda} (w) = \bigcup_{y \leqslant w}
   \mathcal{B}^{\circ}_{\lambda} (y) . \]
The term \textit{Demazure atom} is used in the literature to
mean two closely related but different things: the sets that
we are denoting $\mathcal{B}_{\lambda}^{\circ} (w)$ or their
characters, which are the functions
$\partial^\circ_w \mathbf{z}^{\lambda}$. When we need to
distinguish them, we will use the term {\textit{crystal
Demazure atoms}} to refer to the subsets
$\mathcal{B}_{\lambda}^{\circ} (w)$ while their characters
will be referred to as \textit{polynomial Demazure atoms}.

Since (up to the factor $\mathbf{z}^{\rho}$) the character of the colored
system indexed by $w$ is the polynomial Demazure atom
$\mathcal{B}^{\circ}_{\lambda} (w)$, we may hope that, when we identify the set
of states of our model with a subset of $\mathcal{B}_{\lambda}$, the the set
of states indexed by $w$ is $\mathcal{B}^{\circ}_{\lambda} (w)$. This is true and we will give a proof of this fact using 
techniques developed by Kashiwara, particularly the
$\star$-involution of the $\mathcal{B}_{\infty}$ crystal, as well as
(\ref{partdc}), which is proved using the Yang-Baxter equation.

As a biproduct of this proof we obtain apparently new algorithms for computing
Lascoux-Sch\"utzenberger right keys, which we now explain.

First, we will explain a theorem of Lascoux-Sch\"utzenberger that concerns the
following question: given a tableau $T \in \mathcal{B}_{\lambda}$, determine
$w \in W$ such that $T \in \mathcal{B}^{\circ}_{\lambda} (w)$. The set of
Demazure atoms is in bijection with the orbit $W
\lambda$ in the weight lattice, and this bijection may be made explicit as follows. The weights $W
\lambda$ are extremal in the sense that they are the vertices of the convex
hull of the set of weights of $\mathcal{B}_{\lambda}$. Each extremal weight 
$w \lambda$ has multiplicity one, in that there exists a unique element 
$u_{w \lambda}$ of $\mathcal{B}_{\lambda}$ with weight $w \lambda$. These
extremal elements are called \textit{key tableaux}, and they may be
characterized by the following property: if $C_1, \ldots, C_k$ are the columns
of a tableau $T$, then $T$ is a key if and only if each column $C_i$ contains
$C_{i + 1}$ elementwise.

Lascoux and Sch\"utzenberger proved that every crystal Demazure atom contains
a unique key tableau, and every key tableau is contained in a unique crystal
Demazure atom.  The weight of the key tableau in
$\mathcal{B}_\lambda^\circ(w)$ is $w\lambda$.  If
$T \in \mathcal{B}_{\lambda}$ let $\key (T)$ be the unique key that is in the
same atom as $T$. This is called the {\textit{right key}} by Lascoux
and Sch{\"u}tzenberger; its origin is in the work of
Ehresmann~{\cite{Ehresmann}} on the topology of flag varieties. (There is also
a \textit{left key}, which is $\key(T')'$, where $T \mapsto T'$ is the
Sch{\"u}tzenberger (Lusztig) involution of $\mathcal{B}_{\lambda}$.)

We will describe two apparently new algorithms that compute $\key (T')$ and
$\key (T)$, respectively. The algorithms depend on a map
$\omega:\mathcal{B}_\lambda\rightarrow W$ such that if $w=w_0\omega(T)$
then $T\in\mathcal{B}_\lambda^\circ(w)$. Thus $\key(T)$ is determined
by the condition that $\wt\bigl(\key(T)\bigr)=w\lambda = w_0\omega(T)\lambda$.
The extremal weight $w\lambda$ has multiplicity one in the crystal
$\mathcal{B}_\lambda$, so the unique key tableau $\key(T)$ with that
weight is determined by $w\lambda$. To compute it, the most frequently
occurring entry (as specified by the weight) must appear in every column of
$\key(T)$, the next most frequently occurring entry must then appear in every
remaining, non-filled column, and so on. The entries of the columns are
thus determined, and arranging each column in ascending order we get~$\key(T)$.

Given a tableau $T$, the first algorithm computes
$\omega(T')$, and the second algorithm computes $\omega(T)$.
The two algorithms depend on the notion of a \textit{nondescending
product} of a sequence of simple
reflections $s_i$ in the Weyl group $W$.
Let $i_1, \cdots, i_k$ be a sequence of indices and define the \textit{nondescending
product} $\Pi_{\operatorname{nd}}(s_{i_1},\cdots,s_{i_k})$
to be $s_{i_1}$ if $k=1$ and then recursively
\begin{equation}
\label{pinddef}
\Pind (s_{i_1}, \cdots, s_{i_k}) = \left\{ \begin{array}{ll}
     s_{i_1} \pi & \text{if $s_{i_1} \pi > \pi$}\\
     \pi & \text{otherwise},
   \end{array} \right.
\end{equation}
where $\pi = \Pind (s_{i_2}, \cdots, s_{i_k})$.

\begin{remark}
\label{heckecompute}
There is another way of calculating the nondescending product.
There is a degenerate Hecke algebra $\mathcal{H}$
with generators $S_i$ subject to the braid relations and the quadratic
relation $S_i^2 = S_i$.\footnote{It may be worth remarking that these are the same relations
satisfied by the Demazure operators $\partial_i$.}
Given $w \in W$, set $S_w = S_{j_1} \cdots S_{j_{\ell}}$
where $w = s_{j_1} \cdots s_{j_{\ell}}$ is a reduced expression. Then the
$S_w$ ($w\in W$) form a basis of $\mathcal{H}$, and we will denote
by $\{ \cdot \}$ the map from this basis to $W$ that sends $S_w$ to $w$. Then
\[\Pi_{\operatorname{nd}}(s_{i_1}, \cdots, s_{i_k}) = \{S_{i_1}\cdots S_{i_k}\}\;.\]
\end{remark}

An element $T$ of $\mathcal{B}_\lambda$ is a semistandard
Young tableau with entries in $\{1,2,\ldots,r\}$
and shape $\lambda$. There is associated with $T$ a Gelfand-Tsetlin pattern
$\Gamma(T)$ as follows. The top row is the shape $\lambda$;
the second row is the shape of the tableau obtained from
$T$ by erasing all entries equal to $r$. The third
row is the shape of the tableau obtained by further
erasing all $r-1$ entries, and so forth. For example
suppose that $r=4$, $\lambda=(5,3,1)$. Here is a
tableau and its Gelfand-Tsetlin pattern:
\begin{equation}
\label{examplepats}
T=\begin{ytableau}1&1&2&4&4\\2&3&4\\3\end{ytableau}\;,\qquad
\Gamma(T)=\left\{\begin{array}{cccccccccc}5&&3&&1&&0\\&3&&2&&1\\&&3&&1\\&&&2\end{array}\right\}\;.
\end{equation}

\subsection*{First algorithm}
To compute $\omega(T')$, we decorate the Gelfand-Tsetlin pattern
as follows. For each subtriangle
\[\begin{array}{ccc}x&&y\\&z\end{array}\]
if $z=y$ then we circle the $z$. We then transfer the circles in the Gelfand-Tsetlin pattern
to the following array:
\begin{equation}
\label{gamcirc}
\left[\begin{array}{ccccccc}
s_1&&s_2&&\cdots&&s_{r-1}\\
&\ddots&&\vdots&&\iddots\\
&&s_1&&s_2\\
&&&s_1 \end{array}\right]\;.
\end{equation}
Note that the array of reflections has one fewer row than the
first, but that circling cannot happen in the top row of the Gelfand-Tsetlin
pattern. Now we traverse this array in the order bottom to top,
right to left. We take the subsequence of circled
entries in the indicated order, and their nondescending
product is $\omega(T')$.

\subsection*{Second algorithm}
To compute $\omega(T)$, we decorate the Gelfand-Tsetlin pattern
as follows. For each subtriangle
\[\begin{array}{ccc}x&&y\\&z\end{array}\]
if $z=x$ then we circle the $z$. We then transfer the circles in the Gelfand-Tsetlin pattern
to the following array:
\begin{equation}
  \label{delcirc}
\left[\begin{array}{ccccccc}
s_1&&s_2&&\cdots&&s_{r-1}\\
&\ddots&&\vdots&&\iddots\\
&&s_{r-2}&&s_{r-1}\\
&&&s_{r-1} \end{array}\right]\;.
\end{equation}
Now we traverse this array in the order bottom to top,
left to right. We take the subsequence of circled
entries in the indicated order, and their nondescending product is $\omega(T)$.

\bigbreak
Let us illustrate these algorithms with the example (\ref{examplepats}).

For the first algorithm, we obtain the following circled Gelfand-Tsetlin pattern and array of simple reflections 
\[\left\{\vcenter{\xymatrix@-2pc{5&&3&&1&&0\\&*+[o][F-]{3}&&2&&{1}\\&&{3}&&*+[o][F-]{1}\\&&&2}}\right\}\;,\qquad
\left[\vcenter{\xymatrix@-2pc{*+[o][F-]{s_1}&&s_2&&{s_3}\\&{s_1}&&*+[o][F-]{s_2}\\&&s_1}}\right]
\]
The first algorithm predicts that if $T'$ is the Sch\"utzenberger involute
of $T$ then $\omega(T')=s_2s_1$, which is the nondescending product of the circle entries
in the order bottom to top, right to left. Thus $w_0\omega(T')=w_0s_2s_1=s_1s_2s_3s_2$.
We claim that $\key(T')$ is the unique key tableau with shape $(5,3,1,0)$ having weight
$w_0\omega(T')\lambda=(0,5,1,3)$. Let us check this. The tableau
$T'$ and its key (computed by Sage using the algorithm in Willis~\cite{WillisKey})
are:
\[T'=\begin{ytableau}1&1&1&2&2\\3&3&4\\4\end{ytableau}\;,\qquad
\key(T')=\begin{ytableau}2&2&2&2&2\\3&4&4\\4\end{ytableau}\;.\]
As claimed $\wt(\key(T'))=w_0\omega(T')\lambda$.

For the second algorithm, there are two circled entries, and we transfer
the circles to the array of reflections as follows:
\[\left\{\vcenter{\xymatrix@-2pc{5&&3&&1&&0\\&{3}&&2&&*+[o][F-]{1}\\&&*+[o][F-]{3}&&1\\&&&2}}\right\}\;,\qquad
\left[\vcenter{\xymatrix@-2pc{{s_1}&&s_2&&*+[o][F-]{s_3}\\&*+[o][F-]{s_2}&&s_3\\&&s_3}}\right]
\]
Thus $\omega(T)=s_2s_3$ is the (nondescending) product in the order
bottom to top, left to right. Then if $w=w_0s_2s_3=s_3s_1s_2s_1$,
the right key of $T$ is determined by the condition that its
weight is $w\lambda=(1,3,0,5)$. Indeed, the right key of $T$ is
\[\key(T)=\begin{ytableau}1&2&2&4&4\\2&4&4\\4\end{ytableau}\;.\]
This is the unique key tableau with shape $(5,3,1,0)$ and
weight $(1,3,0,5)$.

The two algorithms hinge on Theorem~\ref{daform}, which refines results on
keys due to Lascoux and Sch\"{u}tzenberger \cite{LascouxSchutzenbergerKeys}.
The proof of Theorem~\ref{daform} is detailed in the subsequent three sections
of the paper, and the resulting algorithms are proved in
Section~\ref{sec:algo}.

\subsection{A sketch of the proofs}
In Section~\ref{sec:kansas} we review the \textit{Tokuyama model}
(in its crystal limit), a statistical-mechanical system
$\mathfrak{S}_{\mathbf{z},\lambda}$ whose partition function is 
$\mathbf{z}^\rho s_\lambda(\mathbf{z})$ in terms of the Schur
function $s_\lambda$ (Proposition~\ref{zschur}). The states of this 5-vertex
model system are in bijection with
$\mathcal{B}_\lambda$. For $w\in W$ we will describe a refinement
$\mathfrak{S}_{\mathbf{z},\lambda,w}$ of this system in Section~\ref{sec:oz}
whose states are a subset of those of $\mathfrak{S}_{\mathbf{z},\lambda}$.
The Weyl group element $w$ is encoded in the boundary conditions.
Thus the set of states of $\mathfrak{S}_{\mathbf{z},\lambda,w}$ may be
identified with a subset of $\mathcal{B}_\lambda$. If $S$
is a subset of a crystal, the \textit{character} of $S$
is $\sum_{v\in S}\mathbf{z}^{\wt(v)}$. Using a Yang-Baxter equation, in
Theorem~\ref{zdematoms}, are able to prove a recursion formula for the
character of $\mathfrak{S}_{\mathbf{z},\lambda,w}$, regarded as
a subset of $\mathcal{B}_\lambda$, and this is the same as
the character of the crystal Demazure atom $\mathcal{B}_\lambda^\circ(w)$.
This suggests but does not prove that the states of
$\mathfrak{S}_{\mathbf{z},\lambda,w}$ comprise $\mathcal{B}_\lambda^\circ(w)$.
The equality of $\mathfrak{S}_{\mathbf{z},\lambda,w}$ and
$\mathcal{B}_\lambda^\circ(w)$ is Theorem~\ref{daform}.
Leveraging the information in Theorem~\ref{zdematoms} into
a proof of Theorem~\ref{daform} is accomplished in Sections~\ref{sec:binf}
and~\ref{sec:dafpro} using methods of Kashiwara~\cite{KashiwaraDemazure},
namely transferring the problem to the infinite $\mathcal{B}_\infty$ crystal,
then using Kashiwara's $\star$-involution of that crystal to transform and
solve the problem. The information that we obtained from the Yang-Baxter
equation in Theorem~\ref{zdematoms} is used at a key step (\ref{eq:keystep})
in the proof. A more detailed outline of these proofs will
be given near the beginning of Section~\ref{sec:binf}.

The two algorithms are treated in Section~\ref{sec:algo}, but the
key insight is earlier in Theorem~\ref{crystalmt}, where
the first algorithm is proved for $\mathfrak{S}_{\mathbf{z},\lambda,w}$.
The idea is that the unique permutation $w$ such that a given state of
$\mathfrak{S}_{\mathbf{z},\lambda}$ lies in of
$\mathfrak{S}_{\mathbf{z},\lambda,w}$ is determined by the pattern of
crossings of colored lines; these crossings correspond to the circled entries
in (\ref{gamcirc}). Then with Theorem~\ref{daform} in hand, the result applies
to $\mathcal{B}^\circ_\lambda(w)$. The second algorithm is deduced from the first using
properties of crystal involutions.

\bigbreak
\textbf{Acknowledgements:} This work was supported by NSF grants
DMS-1801527 (Brubaker) and DMS-1601026 (Bump). Buciumas was supported by ARC
grant DP180103150. Gustafsson was at Stanford University (his affiliation at the date of submission) supported by the Knut and Alice Wallenberg
Foundation. We thank Amol Aggarwal, Alexei Borodin, Vic Reiner, Anne
Schilling, Michael Wheeler and Matthew Willis for helpful conversations and communications. We thank the referees for useful comments which improved the exposition of the paper. 

\section{\label{sec:do}Demazure operators}
Let us review the theory of Demazure operators.
Let $\Phi$ be a root system with weight lattice $\Lambda$, which may be
regarded as the weight lattice of a complex reductive Lie group $G$. Thus if
$T$ is a maximal torus of $G$, then we may identify $\Lambda$ with the group
$X^{\ast} (T)$ of rational characters of $T$. If $\mathbf{z} \in T$ and
$\lambda \in \Lambda$ we will denote by $\mathbf{z}^{\lambda}$ the
application of $\lambda$ to $\mathbf{z}$. Let $\mathcal{O} (T)$ be the set
of polynomial functions on $T$, that is, finite linear combinations of the
functions $\mathbf{z}^{\lambda}$.

We decompose $\Phi$ into positive and negative roots, and let $\alpha_i$
($i\in I$) be the simple positive roots, where $I$ is an index set.
Let $\alpha_i^{\vee} \in X_{\ast} (T)$ denote the corresponding simple coroots and $s_i$ the
corresponding simple reflections generating the Weyl group $W$. To each simple reflection $s_i$ with
$i\in I$, we define the isobaric Demazure operator acting on
$f \in \mathcal{O} (T)$ by
\begin{equation} \label{isobaricddefined}
\partial_i f (\mathbf{z}) = 
\frac{f (\mathbf{z})-\mathbf{z}^{-\alpha_i} f (s_i \mathbf{z})}
{1 -\mathbf{z}^{-\alpha_i}}. \end{equation}
The numerator is divisible by the denominator, so the resulting function is
again in $\mathcal{O}(T)$.

It is straightforward to check that $\partial^2_i = \partial_i = s_i \partial_i$. Given any $\mu \in \Lambda$, set 
$k = \langle \mu, \alpha_i^{\vee} \rangle$ so $s_i (\mu) = \mu - k\alpha_i$. Then the action on the monomial $\mathbf{z}^{\mu}$ is given by 
\begin{equation}
  \partial_i \mathbf{z}^{\mu} = \left\{ \begin{array}{ll}
    \mathbf{z}^{\mu} +\mathbf{z}^{\mu - \alpha_i} + \ldots
    +\mathbf{z}^{s_i (\mu)} & \text{if $k \geqslant 0$,}\\
    0 & \text{if $k = - 1$,}\\
    - (\mathbf{z}^{\mu + \alpha_i} +\mathbf{z}^{\mu + 2 \alpha_i} + \ldots
    +\mathbf{z}^{s_i (\mu + \alpha_i)}) & \text{if $k < - 1$} .
  \end{array} \right.
\end{equation}
We will also make use of $\partial^\circ_i:=\partial_i-1$, that is
\[\partial^\circ_i f(\mathbf{z}):=
\frac{f(\mathbf{z})-f(s_i\mathbf{z})}{{\mathbf{z}^{\alpha}_i-1}}.\]

Both $\partial_i$ and $\partial^\circ_i$ satisfy the braid relations. Thus
\[\partial_{i}\partial_j\partial_i\cdots = \partial_j\partial_i\partial_j\cdots,\]
where the number of terms on both sides is the order of $s_is_j$ in $W$, and
similarly for the $\partial^\circ_i$. These are proved in~\cite{BumpLie},
Proposition~25.1 and Proposition~25.3. (There is a typo in the second
Proposition where the wrong font is used for $\partial_i$.) Consequently
to each $w \in W$, and any reduced decomposition $w = s_{i_1} \cdots s_{i_k}$, we may define
$\partial_w = \partial_{i_1} \cdots \partial_{i_k}$ and
$\partial^\circ_w = \partial^\circ_{i_1} \cdots
\partial^\circ_{i_k}$. For $w = 1$ we let $\partial_1 = \partial^\circ_1 = 1$.

Let $w_0$ be the long Weyl group element. If $\lambda$ is a dominant weight
let $\chi_{\lambda}$ denote the character of the irreducible representation
$\pi_{\lambda}$ with highest weight $\lambda$. The {\textit{Demazure character
formula}} is the identity, for $\mathbf{z}\in T$:
\[ \chi_{\lambda} (\mathbf{z}) = \partial_{w_0} \mathbf{z}^{\lambda} .
\]
For a proof, see~\cite{BumpLie}, Theorem~25.3. More generally for any Weyl
group element $w$, we may consider $\partial_w\mathbf{z}^\lambda$. These
polynomials are called \textit{Demazure characters}.

Next we review the theory of (polynomial) \textit{Demazure atoms}. 
These are polynomials of the form 
$\partial^\circ_w\mathbf{z}^\lambda$. They were introduced in
  type~A by Lascoux and Sch\"utzenberger~\cite{LascouxSchutzenbergerKeys}, who
  called them ``standard bases.'' The modern term ``Demazure atom''
  was introduced by Mason in~\cite{MasonAtoms}, who showed that they are
  specializations of nonsymmetric Macdonald polynomials, among
  other things. The following theorem, done for type~A in~\cite{LascouxSchutzenbergerKeys},
  relates Demazure characters and Demazure atoms and
  is valid for any finite Cartan type.

\begin{theorem}
  \label{thm:lskeys}
  Let $f\in\mathcal{O}(T)$. Then
  \begin{equation}
    \label{phidemaz}
    \partial_{w}f(\mathbf{z})=\sum_{y\leqslant w}\partial^\circ_{y}f(\mathbf{z}).
  \end{equation}
\end{theorem}

\begin{proof}
  We prove this by induction with respect to the Bruhat order. Setting $\phi (w) := \partial^\circ_w f(\mathbf{z})$ and
  assuming the theorem for $w$, we must show that for any $s_i$ with $s_i w > w$ in the Bruhat order,
   \begin{equation}
   \label{inductionstep} 
   \sum_{y\leqslant s_iw}\phi(y)=\partial_{s_iw} f(\mathbf{z}). 
   \end{equation}

  We recall ``Property~Z'' of
  Deodhar~{\cite{DeodharCharacterizations}}, which asserts that if $s_i w > w$
  and $s_i y > y$ then the following inequalities are equivalent:
\[ y \leqslant w\quad \iff \quad y \leqslant s_i w \quad \iff \quad
s_i y \leqslant s_i w\;.\]
  Using this fact we may split the sum on the left-hand side as follows
  \begin{equation*}
    \sum_{y \leqslant s_i w} \phi (y)
    = \sum_{\substack{y \leqslant s_i w \\ y < s_i y }} \phi (y) + \sum_{\substack{y \leqslant s_i w \\ s_i y < y}} \phi (y)
    = \sum_{\substack{y \leqslant s_i w \\ y < s_i y}} \phi (y) + \sum_{\substack{s_i y \leqslant s_i w \\
    y < s_i y}} \phi (s_i y) 
    = \sum_{\substack{
       y \leqslant w\\
       y < s_i y}} \bigl(\phi (y) + \phi (s_i y) \bigr) \, .
  \end{equation*}

  If $s_iw>w$ then
  \begin{equation}
    \label{partialrecurse}
    \phi(w)+\phi(s_iw)=\partial_i\phi(w).
  \end{equation}
  Indeed, since $\partial_i=\partial^\circ_i+1$, this is another way of writing
  \[\partial^\circ_{s_iw}f(\mathbf{z})=\partial^\circ_i\partial^\circ_w\,f(\mathbf{z})\,,\]
  which follows from the definitions.
  
  Using (\ref{partialrecurse}), we obtain
  \begin{equation}
    \label{partialcomp}
     \sum_{y \leqslant s_i w} \phi (y) = \partial_i \Bigl(
     \sum_{\substack{y \leqslant w\\ y < s_i y}} \phi (y) \Bigr) .
  \end{equation}
Still assuming $s_iw>w$ we will prove that
\begin{equation}
  \label{partialsimp}
 \partial_i \sum_{\substack{   y \leqslant w\\
       y < s_i y}} \phi (y) = \partial_i \sum_{y \leqslant w} \phi (y) .
\end{equation}
  We split the terms on the right-hand side into three groups and write
  \[ \partial_i \sum_{y \leqslant w} \phi (y) = \partial_i 
     \sum_{\substack{
       y \leqslant w\\
       y < s_i y\\
       s_i y \leqslant w}} \bigl( \phi (y) + \phi (s_i y) \bigr) + \partial_i
     \sum_{\substack{
       y \leqslant w\\
       y < s_i y\\
       s_i y \nleqslant w}} \phi (y) . \]
  Now using (\ref{partialrecurse}) again this equals
  \[ \partial_i \sum_{\substack{
       y \leqslant w\\
       y < s_i y\\
       s_i y \leqslant w}} \partial_i \phi (y) + \partial_i
     \sum_{\substack{  y \leqslant w\\
       y < s_i y\\
       s_i y \nleqslant w}}\phi (y), \]
  and remembering that $\partial_i^2 = \partial_i$ this equals
  \[ \partial_i \Bigl( \sum_{\substack{
       y \leqslant w\\
       y < s_i y\\
       s_i y \leqslant w}} \phi (y) + \sum_{\substack{
       y \leqslant w\\
       y < s_i y\\
       s_i y \nleqslant w}}\phi (y) \Bigr) = \partial_i
     \sum_{\substack{
       y \leqslant w\\
       y < s_i y}} \phi (y), \]
     proving (\ref{partialsimp}).

   Now (\ref{inductionstep}) follows using (\ref{partialcomp}),
   (\ref{partialsimp}) and our induction hypothesis.
\end{proof}

\section{\label{sec:kansas}Ice Models for $\GL(r)$}

In statistical mechanics, an {\it ensemble} is a probability distribution
over every possible admissible state (i.e., microscopic arrangement) of particles in a given physical system.
The probability of any given state is measured by its {\it Boltzmann weight},
which is calculated by computing the energy associated to all local interactions
between particles. If there are only finitely many admissible states in the ensemble (as in all of the examples in this paper), 
then the {\it partition function} is defined to be a sum of the Boltzmann weights of each state. While computing
the partition function explicitly is often intractable, there is a nice class of so-called 
{\it solvable models}~\cite{Baxter,JimboMiwa} for
which the partition function may be computed using a microscopic symmetry of the partition function
known as the {\it Yang-Baxter equation}. With few exceptions,
solvable models are based on two-dimensional physical
systems.

The \textit{six-vertex} or \textit{ice-type models} are a class of
two-dimensional solvable models based on a square, planar grid in which admissible states are
determined by associating one of two spins $\{ +, - \}$ to 
each edge. See Figure~\ref{icestate} for an example. The term \textit{six-vertex} refers to the fact that only six admissible configurations of spins are
allowed on the four edges adjacent to any vertex in the grid. Similarly, \textit{five-vertex models} 
are systems, typically degenerations of six-vertex models,
in which only five local configurations are allowed. An example of such a set of configurations can be found in Figure~\ref{uncoloredbw} where the configuration labeled $\texttt{b}_1$ is removed. In the next two sections, we will revisit all of the above terms and give precise definitions
for an ensemble of admissible states and associated weights that result in a solvable model first for a five-vertex model based on the configurations in Figure~\ref{uncoloredbw}, and then generalizations thereof. 
Our Boltzmann weights for states will depend on several complex variables and while they will not try to model the probability distribution of a physical system, they will nonetheless
result in solvable variants of the above five-vertex model whose partition functions are explicitly evaluable as Demazure atoms.

More precisely, inspired by colored lattice models in Borodin and Wheeler~\cite{BorodinWheelerColored},
we will show that Demazure atoms and characters for $\GL(r)$ can be represented as
partition functions of certain ``colored five-vertex models.''
Strictly speaking, it is no longer true that there are only five
allowed configurations at a vertex. Still, the allowed
configurations can be classified into five different
groups, which we will denote $\tt{a}_1$, $\tt{a}_2$,
$\tt{b}_2$, $\tt{c}_1$ and $\tt{c}_2$ in keeping with notational conventions of \cite{Baxter}.
Before introducing the colored models, we begin with a model that is not
new, but rather a special case of models due to Hamel and
King~\cite{HamelKing} and Brubaker, Bump and Friedberg~\cite{hkice}.

Our five-vertex models will occur on square grids inside a finite rectangle of fixed size. Then to describe the ensemble of admissible states of the model, it suffices to specify
the size of the rectangle and the spins associated to edges along the boundary of this rectangle. Indeed, then the admissible states will consist of all possible assignments 
of spins to the remaining edges of the grid so that every vertex has adjacent edges in one of the five allowable configurations of Figure~\ref{uncoloredbw} (those not of form $\tt{b}_1$). 

Given an integer partition $\lambda = (\lambda_1, \ldots, \lambda_r)$ with $r$ parts, our grid will have $r$ rows and $N+1$ columns, where $N$ is a fixed integer at least $\lambda_1 + r - 1$.
In order to enumerate the vertices, the columns are labeled $0$ to $N$ from
right to left, and the rows are labeled $1$ to $r$ from top to bottom.
Vertices occur at every crossing of rows and columns and boundary edges are those edges in the grid connected to only one vertex. The spins $\{+,-\}$ of the edges on
the boundary are fixed according to the choice of $\lambda$ by the following rules.
For the top boundary edges, we put~$-$ in the columns labeled $\lambda_i+r-i$ for $i \in \{1, \ldots, r\}$ and~$+$ in the remaining columns. Then, we put~$+$ on all the left and bottom
boundary edges and~$-$ on the right boundary edges. As noted above, an (admissible) \textit{state} $\mathfrak{s}$ of the resulting system
assigns spins to the interior edges so that each vertex is one of the five configurations in Figure~\ref{uncoloredbw}
\textit{excluding} patterns of type $\tt{b}_1$, which are not allowed (or equivalently, are assigned weight $0$). An example of an admissible state for $\lambda = (2,1,0)$ and $N=4$ is given in Figure~\ref{icestate}. 

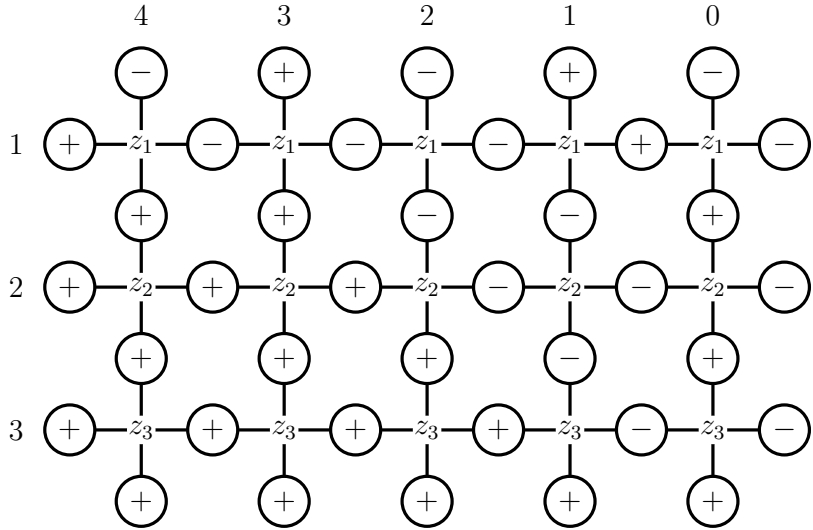
\begin{figure}[h]
\[
  \scalebox{.95}{\begin{tikzpicture}
    \draw [line width=0.45mm] (1,0)--(1,6);
    \draw [line width=0.45mm] (3,0)--(3,6);
    \draw [line width=0.45mm] (5,0)--(5,6);
    \draw [line width=0.45mm] (7,0)--(7,6);
    \draw [line width=0.45mm] (9,0)--(9,6);
    \draw [line width=0.45mm] (0,1)--(10,1);
    \draw [line width=0.45mm] (0,3)--(10,3);
    \draw [line width=0.45mm] (0,5)--(10,5);
    \draw[line width=0.45mm, fill=white] (1,6) circle (.35);
    \draw[line width=0.45mm, fill=white] (3,6) circle (.35);
    \draw[line width=0.45mm, fill=white] (5,6) circle (.35);
    \draw[line width=0.45mm, fill=white] (7,6) circle (.35);
    \draw[line width=0.45mm, fill=white] (9,6) circle (.35);
    \draw[line width=0.45mm, fill=white] (1,4) circle (.35);
    \draw[line width=0.45mm, fill=white] (3,4) circle (.35);
    \draw[line width=0.45mm, fill=white] (5,4) circle (.35);
    \draw[line width=0.45mm, fill=white] (7,4) circle (.35);
    \draw[line width=0.45mm, fill=white] (9,4) circle (.35);
    \draw[line width=0.45mm, fill=white] (1,2) circle (.35);
    \draw[line width=0.45mm, fill=white] (3,2) circle (.35);
    \draw[line width=0.45mm, fill=white] (5,2) circle (.35);
    \draw[line width=0.45mm, fill=white] (7,2) circle (.35);
    \draw[line width=0.45mm, fill=white] (9,2) circle (.35);
    \draw[line width=0.45mm, fill=white] (1,0) circle (.35);
    \draw[line width=0.45mm, fill=white] (3,0) circle (.35);
    \draw[line width=0.45mm, fill=white] (5,0) circle (.35);
    \draw[line width=0.45mm, fill=white] (7,0) circle (.35);
    \draw[line width=0.45mm, fill=white] (9,0) circle (.35);
    \draw[line width=0.45mm, fill=white] (0,5) circle (.35);
    \draw[line width=0.45mm, fill=white] (2,5) circle (.35);
    \draw[line width=0.45mm, fill=white] (4,5) circle (.35);
    \draw[line width=0.45mm, fill=white] (6,5) circle (.35);
    \draw[line width=0.45mm, fill=white] (8,5) circle (.35);
    \draw[line width=0.45mm, fill=white] (10,5) circle (.35);
    \draw[line width=0.45mm, fill=white] (0,3) circle (.35);
    \draw[line width=0.45mm, fill=white] (2,3) circle (.35);
    \draw[line width=0.45mm, fill=white] (4,3) circle (.35);
    \draw[line width=0.45mm, fill=white] (6,3) circle (.35);
    \draw[line width=0.45mm, fill=white] (8,3) circle (.35);
    \draw[line width=0.45mm, fill=white] (10,3) circle (.35);
    \draw[line width=0.45mm, fill=white] (0,1) circle (.35);
    \draw[line width=0.45mm, fill=white] (2,1) circle (.35);
    \draw[line width=0.45mm, fill=white] (4,1) circle (.35);
    \draw[line width=0.45mm, fill=white] (6,1) circle (.35);
    \draw[line width=0.45mm, fill=white] (8,1) circle (.35);
    \draw[line width=0.45mm, fill=white] (10,1) circle (.35);

    \path[fill=white] (1,1) circle (.2);
    \node at (1,1) {$z_3$};
    \path[fill=white] (3,1) circle (.2);
    \node at (3,1) {$z_3$};
    \path[fill=white] (5,1) circle (.2);
    \node at (5,1) {$z_3$};
    \path[fill=white] (7,1) circle (.2);
    \node at (7,1) {$z_3$};
    \path[fill=white] (9,1) circle (.2);
    \node at (9,1) {$z_3$};

    \path[fill=white] (1,3) circle (.2);
    \node at (1,3) {$z_2$};
    \path[fill=white] (3,3) circle (.2);
    \node at (3,3) {$z_2$};
    \path[fill=white] (5,3) circle (.2);
    \node at (5,3) {$z_2$};
    \path[fill=white] (7,3) circle (.2);
    \node at (7,3) {$z_2$};
    \path[fill=white] (9,3) circle (.2);
    \node at (9,3) {$z_2$};

    \path[fill=white] (1,5) circle (.2);
    \node at (1,5) {$z_1$};
    \path[fill=white] (3,5) circle (.2);
    \node at (3,5) {$z_1$};
    \path[fill=white] (5,5) circle (.2);
    \node at (5,5) {$z_1$};
    \path[fill=white] (7,5) circle (.2);
    \node at (7,5) {$z_1$};
    \path[fill=white] (9,5) circle (.2);
    \node at (9,5) {$z_1$};

    \node at (1,6) {$-$};
    \node at (3,6) {$+$};
    \node at (5,6) {$-$};
    \node at (7,6) {$+$};
    \node at (9,6) {$-$};
    \node at (1,4) {$+$};
    \node at (3,4) {$+$};
    \node at (5,4) {$-$};
    \node at (7,4) {$-$};
    \node at (9,4) {$+$};
    \node at (1,2) {$+$};
    \node at (3,2) {$+$};
    \node at (5,2) {$+$};
    \node at (7,2) {$-$};
    \node at (9,2) {$+$};
    \node at (1,0) {$+$};
    \node at (3,0) {$+$};
    \node at (5,0) {$+$};
    \node at (7,0) {$+$};
    \node at (9,0) {$+$};
    \node at (0,5) {$+$};
    \node at (2,5) {$-$};
    \node at (4,5) {$-$};
    \node at (6,5) {$-$};
    \node at (8,5) {$+$};
    \node at (10,5) {$-$};
    \node at (0,3) {$+$};
    \node at (2,3) {$+$};
    \node at (4,3) {$+$};
    \node at (6,3) {$-$};
    \node at (8,3) {$-$};
    \node at (10,3) {$-$};
    \node at (0,1) {$+$};
    \node at (2,1) {$+$};
    \node at (4,1) {$+$};
    \node at (6,1) {$+$};
    \node at (8,1) {$-$};
    \node at (10,1) {$-$};
    \node at (1.00,6.8) {$ 4$};
    \node at (3.00,6.8) {$ 3$};
    \node at (5.00,6.8) {$ 2$};
    \node at (7.00,6.8) {$ 1$};
    \node at (9.00,6.8) {$ 0$};

    \node at (-.75,1) {$ 3$};
    \node at (-.75,3) {$ 2$};
    \node at (-.75,5) {$ 1$};
  \end{tikzpicture}}
\]        
\caption{A state of a five-vertex model system with $N=4$, $r=3$ and $\lambda=(2,1,0)$.}
\label{icestate}
\end{figure}

Next we describe the Boltzmann weight $\beta(\mathfrak{s})$ of a state $\mathfrak{s}$. It will depend on a choice of $r$ complex numbers $\mathbf{z} = (z_1, \ldots, z_r)$ in $(\mathbb{C}^\times)^r$. 
We set
$$  \beta(\mathfrak{s}) := \prod_{v: \; \textrm{vertex in } \mathfrak{s}} \textrm{wt}(v), $$ 
where the function $\textrm{wt}(v)$ is defined in Figure~\ref{uncoloredbw} and depends on the row~$i$ in which the vertex $v$ appears.
For example, one may
quickly check that the state in Figure~\ref{icestate} has Boltzmann weight $z_1^3 z_2^2 z_3$. 

Let $\mathfrak{S}_{\mathbf{z},\lambda}$ denote the ensemble of all admissible states with boundary
conditions dictated by $\lambda$ and weights depending on parameters $\mathbf{z} = (z_1, \ldots, z_r)$.
Further define the \textit{partition function} $Z(\mathfrak{S}_{\mathbf{z},\lambda})$
to be the sum of the Boltzmann weights over all states in the ensemble.
Our notation suppresses the choice of number of columns $N$; indeed, the partition function is independent of any such (large enough) choice, since adding columns to the left
of the $\lambda_1+r-1$ column adds only $\tt{a}_1$ patterns, which have weight~1.

\begin{figure}
\[
\begin{array}{|c|c|c|c|c|c|}
\hline
  \tt{a}_1&\tt{a}_2&\tt{b}_1&\tt{b}_2&\tt{c}_1&\tt{c}_2\\
\hline
  \gammaice{+}{+}{+}{+} &
  \gammaice{-}{-}{-}{-} &
  \gammaice{+}{-}{+}{-} &
  \gammaice{-}{+}{-}{+} &
  \gammaice{-}{+}{+}{-} &
  \gammaice{+}{-}{-}{+}\\
\hline
  1&z_i&0&z_i&z_i&1\\
\hline\end{array}\]
\caption{Boltzmann weights $\textrm{wt}(v)$ for a vertex $v$ in the $i$-th row of the uncolored system.}
\label{uncoloredbw}
\end{figure}

We will next describe bijections between states of this system and
two other sets of combinatorial objects:
Gelfand-Tsetlin patterns with top row $\lambda$
and semistandard Young tableaux of shape
$\lambda$
with entries in $\{1,2,\ldots,r\}$. These will allow us to conclude that $Z(\mathfrak{S}_{\mathbf{z},\lambda})$
is, up to a simple factor, the Schur polynomial $s_\lambda(\mathbf{z})$.

Our boundary conditions imply via a
combinatorial argument (\cite{Baxter} Section~8.3 or Proposition~19.1
in~\cite{wmd5book}) that in any given state $\mathfrak{s}$ of
the system, the number of $-$ spins in the row of $N$ vertical edges above
the $i$-th row will be exactly $r+1-i$.
Let $(i,j)$ with $r \geqslant j\geqslant i$ enumerate these spins and let $A_{i,j}$ be their corresponding column numbers,
in descending order. Then
\[\GTP(\mathfrak{s}):=\left\{\begin{array}{ccccccc}
A_{1,1}&&A_{1,2}&&\cdots&& A_{1,r}\\
&A_{2,2}&&\cdots&& A_{2,r}\\
&&\ddots&\vdots&\iddots\\
&&&A_{r,r}\end{array}\right\}\]
is a \textit{left-strict} Gelfand-Tsetlin pattern, meaning that
$A_{i,j}>A_{i+1,j+1}\geqslant A_{i,j+1}$. This follows from
Proposition~19.1 of~\cite{wmd5book}, taking into account the
omission of $\tt{b}_1$ patterns in Figure~\ref{uncoloredbw}, which implies
that the inequality $A_{i,j}>A_{i+1,j+1}$ is strict.

\begin{remark}
If we allowed patterns of type $\tt{b}_1$ we would
have $A_{i,j}\geqslant A_{i+1,j+1}\geqslant A_{i,j+1}$
and $A_{i,j}>A_{i,j+1}$.
\end{remark}

Since $\GTP(\mathfrak{s})$ is left-strict, 
we may subtract $\rho_{r+1-i}:=(r-i,r-i-1,\cdots,0)$ from the $i$-th
row of $\GTP(\mathfrak{s})$ to obtain another Gelfand-Tsetlin
pattern. We denote this \textit{reduced} pattern by
\begin{equation}
  \label{reducedgtp}
 \GTP^\circ(\mathfrak{s}):=\left\{\begin{array}{ccccccc}
a_{1,1}&&a_{1,2}&&\cdots&& a_{1,r}\\
&a_{2,2}&&\cdots&& a_{2,r}\\
&&\ddots&\vdots&\iddots\\
&&&a_{r,r}\end{array}\right\}\;,
\end{equation}
whose entries are $a_{i,j}=A_{i,j}-r+j$.
The top row of $\GTP^\circ(\mathfrak{s})$ is $\lambda$. The map
$\mathfrak{s}\mapsto \GTP^\circ(\mathfrak{s})$ is easily seen to be
a bijection between the
states of $\mathfrak{S}_{\mathbf{z},\lambda}$ and the set of Gelfand-Tsetlin
patterns with top row $\lambda$.

There is also associated with a state $\mathfrak{s}$ a
semistandard Young tableau, which may be described as follows.
Let $\mathcal{B}_\lambda$ be the set of semi-standard Young tableaux
of shape $\lambda$ with entries in $\{1,2,3,\ldots,r\}$.
We first associate a tableau $\mathfrak{T}\in\mathcal{B}_\lambda$ with any
Gelfand-Tsetlin pattern. 
The top row of the pattern is
the shape $\lambda$ of $\mathfrak{T}$. Removing the
cells labeled $r$ from the tableau results in the shape
that is the second row of the Gelfand-Tsetlin pattern, etc. This
procedure is reversible and so there is another bijection between
$\mathcal{B}_\lambda$ and Gelfand-Tsetlin patterns with top row $\lambda$.
We may compose this with our previous bijection between
$\mathfrak{S}_{\mathbf{z},\lambda}$ and Gelfand-Tsetlin patterns. Given an
admissible state $\mathfrak{s}$, we will denote the associated tableau by
$\mathfrak{T}(\mathfrak{s})$.

For example with the state $\mathfrak{s}$ in Figure~\ref{icestate}, we have
\[\GTP(\mathfrak{s})=\left\{\begin{array}{ccccc}4&&2&&0\\&2&&1\\&&1\end{array}\right\},\qquad
\GTP^\circ(\mathfrak{s})=\left\{\begin{array}{ccccc}2&&1&&0\\&1&&1\\&&1\end{array}\right\},\qquad
\mathfrak{T}(\mathfrak{s})=\begin{ytableau}1&3\\2\end{ytableau}\;.\]

The set $\mathcal{B}_\lambda$ has the structure of a
Kashiwara-Nakashima crystal of tableaux (see \cite{KashiwaraNakashima,BumpSchilling}).
As such it comes with a weight map
$\wt:\mathcal{B}_\lambda\longrightarrow\Lambda,$ where $\Lambda \simeq \mathbb{Z}^r$ denotes
the weight lattice for $G=\GL(r)$.
If $\mathfrak{T}\in\mathcal{B}_\lambda$, then identifying $\Lambda$ with $\mathbb{Z}^r$,
we define
$\wt(\mathfrak{T})=(\mu_1,\cdots,\mu_r)$ where $\mu_i$ is the number
of entries in $\mathfrak{T}$ equal to~$i$. 

\begin{proposition}
  \label{zschur}
  Let $\lambda \in \Lambda$ be a dominant weight and $\mathfrak{s} \in \mathfrak{S}_{\mathbf{z},\lambda}$ be an admissible state of the uncolored five-vertex model defined above.
  \begin{enumerate}[label={(\roman*)}, leftmargin=*]
    \item \label{itm:zweightid} The Boltzmann weight $\beta(\mathfrak{s})$ and the weight map of the associated tableau $\mathfrak{T}(\mathfrak{s})$ are related by
    $$\beta(\mathfrak{s})=\mathbf{z}^{\rho+w_0 \wt(\mathfrak{T}(\mathfrak{s}))}.$$
  \item \label{itm:schur} The partition function of an ensemble $\mathfrak{S}_{\mathbf{z},\lambda}$ is related to Schur functions by $${Z(\mathfrak{S}_{\mathbf{z},\lambda})=\mathbf{z}^\rho\,s_\lambda(\mathbf{z})}.$$
\end{enumerate}
\end{proposition}

To illustrate \ref{itm:zweightid}, in the example of Figure~\ref{icestate}, we have
\[\beta(\mathfrak{s})=z_1^3z_2^2z_3,\quad\mathbf{z}^\rho=z_1^2z_2,\quad \textrm{and} \quad
\mathbf{z}^{w_0\wt(\mathfrak{T}(\mathfrak{s}))}=z_1z_2z_3.\]

\begin{proof}[Proof of Proposition~\ref{zschur}]
  To prove \ref{itm:zweightid}, note that from the weights in Figure~\ref{uncoloredbw} a
  vertex in the $i$-th row contributes a factor of $z_i$ if and only
  if the spin to the left of the vertex is $-$. Hence the power of
  $z_i$ equals the number of $-$ spins on the horizontal edges in
  the $i$-th row, not counting the $-$ on the right boundary edge.
  Now such $-$ occur on the horizontal edges between the $A_{i,j}$ and
  $A_{i+1,j+1}$ columns, or to the right of the $A_{i,r}$ column. Hence
  the power of $z_i$ in $\mathfrak{S}_{\mathbf{z},\lambda}$ is
  \[\sum_{j=i}^rA_{i,j}-\sum_{j=i}^{r-1}A_{i+1,j+1}=
  \left(\sum_{j=i}^ra_{i,j}-\sum_{j=i+1}^{r}a_{i+1,j}\right)+r-i.\]
  The term in parentheses is the number of $r+1-i$
  entries in the tableau $\mathfrak{T}(\mathfrak{s})$.
  Taking the product over all $i$ gives~\ref{itm:zweightid}.

  Using \ref{itm:zweightid} and the combinatorial formula
  \[s_\lambda(\mathbf{z})=\sum_{\mathfrak{T}}\mathbf{z}^{\wt{\mathfrak{T}}}\]
  for the Schur function we have $Z(\mathfrak{S}_{\mathbf{z},\lambda})=\mathbf{z}^\rho\,s_\lambda(w_0\mathbf{z})$.
  Part \ref{itm:schur} now follows from the symmetry of the Schur function.
\end{proof}

Alternatively, we can evaluate the partition function using a local symmetry
known as the Yang-Baxter equation, which is Theorem~\ref{thm:ybe} below.
To state this we need to introduce a new type of vertices that we will call
rotated vertices. These vertices are rotated by 45 degrees counterclockwise
and there are two parameters $z_i,z_j$ associated to each vertex. We denote
such rotated vertices by $R_{z_i,z_j}$ (here we use $R$ as their Boltzmann
weights may be alternately viewed as entries of an $R$-matrix that ``solves''
a lattice model). These vertices can be attached to the grid systems we defined
before, like the one in Figure~\ref{icestate} to obtain new systems. It is by
working with these new systems that we can use the Yang-Baxter equation and
derive functional equations for the partition function of our initial system
(the one without any rotated vertices).

The Boltzmann weights of the rotated vertices are different from the Boltzmann weights of the regular vertices and are given in Figure~\ref{crystalrmatrix}.

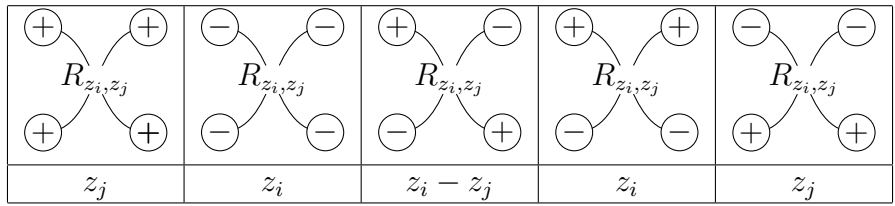
\begin{figure}[b]
\[\begin{array}{|c|c|c|c|c|}
\hline
\begin{tikzpicture}[scale=0.7]
\draw (0,0) to [out = 0, in = 180] (2,2);
\draw (0,2) to [out = 0, in = 180] (2,0);
\draw[fill=white] (0,0) circle (.35);
\draw[fill=white] (0,2) circle (.35);
\draw[fill=white] (2,0) circle (.35);
\draw[fill=white] (2,2) circle (.35);
\node at (0,0) {$+$};
\node at (0,2) {$+$};
\node at (2,2) {$+$};
\node at (2,0) {$+$};
\node at (2,0) {$+$};
\path[fill=white] (1,1) circle (.3);
\node at (1,1) {$R_{z_i,z_j}$};
\end{tikzpicture}&
\begin{tikzpicture}[scale=0.7]
\draw (0,0) to [out = 0, in = 180] (2,2);
\draw (0,2) to [out = 0, in = 180] (2,0);
\draw[fill=white] (0,0) circle (.35);
\draw[fill=white] (0,2) circle (.35);
\draw[fill=white] (2,0) circle (.35);
\draw[fill=white] (2,2) circle (.35);
\node at (0,0) {$-$};
\node at (0,2) {$-$};
\node at (2,2) {$-$};
\node at (2,0) {$-$};
\path[fill=white] (1,1) circle (.3);
\node at (1,1) {$R_{z_i,z_j}$};
\end{tikzpicture}&
\begin{tikzpicture}[scale=0.7]
\draw (0,0) to [out = 0, in = 180] (2,2);
\draw (0,2) to [out = 0, in = 180] (2,0);
\draw[fill=white] (0,0) circle (.35);
\draw[fill=white] (0,2) circle (.35);
\draw[fill=white] (2,0) circle (.35);
\draw[fill=white] (2,2) circle (.35);
\node at (0,0) {$-$};
\node at (0,2) {$+$};
\node at (2,2) {$-$};
\node at (2,0) {$+$};
\path[fill=white] (1,1) circle (.3);
\node at (1,1) {$R_{z_i,z_j}$};
\end{tikzpicture}&
\begin{tikzpicture}[scale=0.7]
\draw (0,0) to [out = 0, in = 180] (2,2);
\draw (0,2) to [out = 0, in = 180] (2,0);
\draw[fill=white] (0,0) circle (.35);
\draw[fill=white] (0,2) circle (.35);
\draw[fill=white] (2,0) circle (.35);
\draw[fill=white] (2,2) circle (.35);
\node at (0,0) {$-$};
\node at (0,2) {$+$};
\node at (2,2) {$+$};
\node at (2,0) {$-$};
\path[fill=white] (1,1) circle (.3);
\node at (1,1) {$R_{z_i,z_j}$};
\end{tikzpicture}&
\begin{tikzpicture}[scale=0.7]
\draw (0,0) to [out = 0, in = 180] (2,2);
\draw (0,2) to [out = 0, in = 180] (2,0);
\draw[fill=white] (0,0) circle (.35);
\draw[fill=white] (0,2) circle (.35);
\draw[fill=white] (2,0) circle (.35);
\draw[fill=white] (2,2) circle (.35);
\node at (0,0) {$+$};
\node at (0,2) {$-$};
\node at (2,2) {$-$};
\node at (2,0) {$+$};
\path[fill=white] (1,1) circle (.3);
\node at (1,1) {$R_{z_i,z_j}$};
\end{tikzpicture}\\
\hline
z_j&z_i&z_i-z_j&z_i&z_j\\
\hline
\end{array}\]
\caption{The R-matrix for the uncolored system. From~\cite{BBB} we know
that we may regard this combinatorial R-matrix as the ``crystal limit'' of the
$U_q(\widehat{\mathfrak{gl}}(1|1))$ R-matrix when $q\rightarrow 0$.}
\label{crystalrmatrix}
\end{figure}

Now consider the following two miniature systems that contain both regular and rotated vertices:
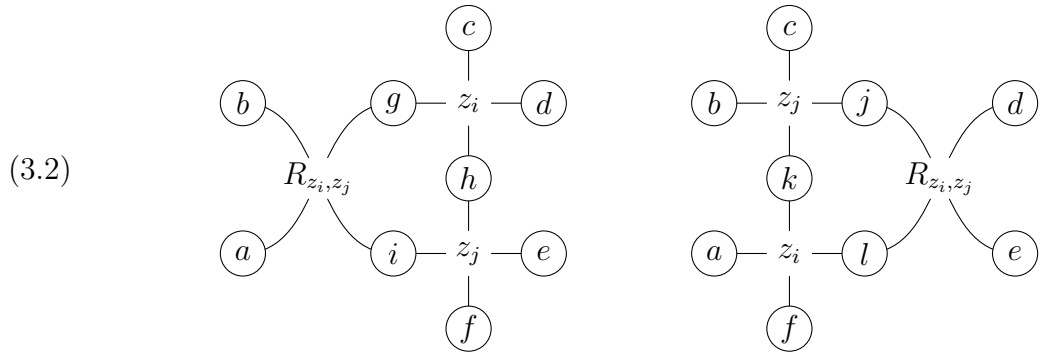
\begin{equation}
\label{eqn:ybe}
\hfill
\begin{tikzpicture}[baseline=(current bounding box.center)]
  \draw (0,1) to [out = 0, in = 180] (2,3) to (4,3);
  \draw (0,3) to [out = 0, in = 180] (2,1) to (4,1);
  \draw (3,0) to (3,4);
  \draw[fill=white] (0,1) circle (.3);
  \draw[fill=white] (0,3) circle (.3);
  \draw[fill=white] (3,4) circle (.3);
  \draw[fill=white] (4,3) circle (.3);
  \draw[fill=white] (4,1) circle (.3);
  \draw[fill=white] (3,0) circle (.3);
  \draw[fill=white] (2,3) circle (.3);
  \draw[fill=white] (2,1) circle (.3);
  \draw[fill=white] (3,2) circle (.3);
  \node at (0,1) {$a$};
  \node at (0,3) {$b$};
  \node at (3,4) {$c$};
  \node at (4,3) {$d$};
  \node at (4,1) {$e$};
  \node at (3,0) {$f$};
  \node at (2,3) {$g$};
  \node at (3,2) {$h$};
  \node at (2,1) {$i$};
\path[fill=white] (3,3) circle (.3);
\node at (3,3) {$z_i$};
\path[fill=white] (3,1) circle (.3);
\node at (3,1) {$z_j$};
\path[fill=white] (1,2) circle (.3);
\node at (1,2) {$R_{z_i,z_j}$};
\end{tikzpicture}\qquad\qquad
\begin{tikzpicture}[baseline=(current bounding box.center)]
  \draw (0,1) to (2,1) to [out = 0, in = 180] (4,3);
  \draw (0,3) to (2,3) to [out = 0, in = 180] (4,1);
  \draw (1,0) to (1,4);
  \draw[fill=white] (0,1) circle (.3);
  \draw[fill=white] (0,3) circle (.3);
  \draw[fill=white] (1,4) circle (.3);
  \draw[fill=white] (4,3) circle (.3);
  \draw[fill=white] (4,1) circle (.3);
  \draw[fill=white] (1,0) circle (.3);
  \draw[fill=white] (2,3) circle (.3);
  \draw[fill=white] (1,2) circle (.3);
  \draw[fill=white] (2,1) circle (.3);
  \node at (0,1) {$a$};
  \node at (0,3) {$b$};
  \node at (1,4) {$c$};
  \node at (4,3) {$d$};
  \node at (4,1) {$e$};
  \node at (1,0) {$f$};
  \node at (2,3) {$j$};
  \node at (1,2) {$k$};
  \node at (2,1) {$l$};
\path[fill=white] (1,3) circle (.3);
\node at (1,3) {$z_j$};
\path[fill=white] (1,1) circle (.3);
\node at (1,1) {$z_i$};
\path[fill=white] (3,2) circle (.3);
\node at (3,2) {$R_{z_i,z_j}$};
\end{tikzpicture}
\end{equation}

Here, as with the system defined before, we fix the spins of the exterior edges ($a,b,c,d,e,f$). 
An assignment of spins to the interior edges is again called a state. 
Both systems have
a partition function defined by summing the weights of the admissible states made from all
possible assignments of spins to the interior edges ($g$, $h$, $i$
in the left system, or $j$, $k$, $l$ on the right). The weight of the entire state is computed just as above: we take a product of the weights of each vertex using the weights of the regular vertices that are given in Figure~\ref{uncoloredbw} and the weights of the rotated vertices that are given in Figure~\ref{crystalrmatrix}.

For example if $(a,b,c,d,e,f)=(+,-,+,-,+,+)$ there is
only one choice $(g,h,i)=(-,+,+)$ that gives a nonzero
contribution to the first system, and the partition
function is the Boltzmann weight $z_iz_j$ of this state.
For the second system, there are two states with
nonzero contribution, namely $(j,k,l)=(-,+,+)$, with
weight $z_j^2$ and $(+,-,-)$ with weight $z_j(z_i-z_j)$.
The partition function again equals $z_iz_j$.

%
%
%
%

\begin{theorem}
\label{thm:ybe}
Let $a,b,c,d,e,f\in\{+,-\}$. Then the partition functions of
the two systems in \eqref{eqn:ybe} are equal.
\end{theorem}

\begin{proof}
This is a special case of a Yang-Baxter equation found in~\cite{hkice}.
Referring to the arXiv version of the paper, the Boltzmann weights
are in Table~1 of that paper with $t_i=0$.
\end{proof}

The symmetry of the Schur function may be easily deduced from this via a
procedure called the ``train argument'' that amounts to repeated use of
Theorem~\ref{thm:ybe} on a larger grid system with an attached rotated vertex as later illustrated in Figure~\ref{modpf} for the colored five-vertex model. See also~\cite[Lemma 4]{hkice}, leading to an alternate proof of
the evaluation of the partition function. 

The models of this section may be described as the ``uncolored'' (or
equivalently ``one-colored'') version of our five-vertex models. They were
known before the writing of this paper. In the next section, we present a
generalization known as colored models, which are new. We will prove a
Yang-Baxter equation in the colored setting (Theorem~\ref{coloredybethm})
that will then be used to relate the partition function of the lattice models
to the Demazure atoms.

\section{\label{sec:oz}Colored Ice Models for $\GL(r)$}

There are multiple ways to depict admissible states of the six-vertex
model. Many of these are described in Chapter~8 of Baxter's inspiring book~\cite{Baxter}.  In
particular, rather than using spins or arrows to decorate edges, one can
instead use the presence or absence of a line (or ``path'') along an edge.
These are the ``line configurations'' in~\cite{Baxter}, Figure~8.2.
Our convention will be that the presence of a line corresponds to a $-$ spin,
so that admissible states may be viewed as a collection of paths moving
downward and rightward through the lattice. Inspired by ideas of Borodin and
Wheeler~\cite{BorodinWheelerColored} in the context of certain other solvable
lattice models, we may assign colors to each such path to refine the partition
function of the prior section to produce polynomial Demazure atoms.

First we describe the relevant solvable colored lattice model. Just as before, upon fixing a dominant weight $\lambda = (\lambda_1, \ldots, \lambda_r)$, we begin with a rectangular lattice of $N+1$ columns ($N \geqslant \lambda_1 + r - 1$) and $r$ rows whose
edges are to be assigned spins $\pm$ according to a five-vertex model.  Moreover, to each edge with $-$ spin, we assign a ``color,'' an additional attribute from a finite set $\{ c_1,\cdots,c_r \}$ of size equal to the number of rows in the model. We will order these colors by 
$c_1>c_2>\cdots>c_r$. By a \textit{colored spin} we mean either $+$, or a color $c_i$.
For the purpose of comparing with the uncolored system, we regard
a colored spin $c_i$ as a spin $-$ with an extra piece
of data, namely a color.

To each dominant weight $\lambda$, we now define $r!$ distinct partition functions. 
Given $w\in W=S_r$ and a vector of colors $\mathbf{c}=(c_1,\cdots,c_r)$, let $w\mathbf{c}$ be the permuted vector of colors, that is
$(w\mathbf{c})_i=c_{w^{-1}i}$. We will call such vectors of colors \emph{flags}. Now assign boundary conditions to the colored lattice model as follows.
To the vertical top boundary edges, we assign spins $-$ in the
columns labeled $\lambda_i+r-i$ as before ($1\leqslant i\leqslant r$).
Now however we also need to assign colors to these edges,
and we assign the color $c_i$ to the $\lambda_i+r-i$
column. Each edge along the right boundary is also assigned a $-$ spin,
but here we assign the colors $w\mathbf{c}$ in order from
top to bottom. Just as before, all remaining boundary spins along the bottom, left, and top are $+$.

Admissible states are then assignments of colored spins to the interior edges such that every vertex has
adjacent spins as in Figure~\ref{coloredweights} with the understanding that the colors {\red{red}}~$>$~{\blue{blue}} may be replaced by any colors $c_i$ and $c_j$ with $c_i > c_j$. Boltzmann weights for each vertex are listed in the figure as well. 
We denote the resulting system of admissible states as $\mathfrak{S}_{\mathbf{z},\lambda,w}$. In short, the choice of $w \in W$ specifies
the row where each colored path, moving downward and rightward through the lattice, exits the right-hand boundary. As before, we denote by $Z(\mathfrak{S}_{\mathbf{z},\lambda,w})$ the partition function of the colored lattice model. 

\begin{figure}[t]
\[
\begin{array}{|c|c|c|c|c|c|}
\hline
\tt{a}_1&\multicolumn{4}{|c|}{\tt{a}_2}&\tt{b}_1\\
\hline
\gammaa & \gammaA{red}{blue}{red}{blue}{R}{B}{R}{B} & \gammaAA{blue}{red}{B}{R}
& \gammaA{red}{red}{red}{red}{R}{R}{R}{R} & \gammaA{blue}{blue}{blue}{blue}{B}{B}{B}{B}
& \gammab{red}{R}
\\ 
\hline
1&\multicolumn{4}{|c|}{z_i}&0\\
\hline
\hline
\multicolumn{2}{|c|}{\tt{b_2}} & \multicolumn{2}{|c|}{\tt{c}_1}&\multicolumn{2}{|c|}{\tt{c}_2}\\
\hline
\gammaB{red}{R}&\gammaB{blue}{B} & \gammac{red}{R} & \gammac{blue}{B} &\gammaC{red}{R} &\gammaC{blue}{B}\\
\hline
\multicolumn{2}{|c|}{z_i}& \multicolumn{2}{|c|}{z_i}&\multicolumn{2}{|c|}{1}\\
\hline
\end{array}\]
\caption{Colored Boltzmann weights for two colors $c_i$ and $c_j$, portrayed as red and blue. We assume that {\red{red}} $>$ {\blue{blue}}.
If the configuration is not in the table, the weight is zero.
The weights are not quite symmetric in the colors, since in the $\tt{a}_2$
patterns, the smaller of the two involved colors (blue) is not allowed on the right edge
and the larger color is not allowed on the bottom edge. With our boundary conditions, the patterns with four
edges all red or blue could be omitted, but this would change the R-matrix
in Figure~\ref{coloredrmatrix}; see Remark~\ref{superrem}. This would not
affect the results of this paper, but we prefer these weights for consistency
with the uncolored case.}
\label{coloredweights}
\end{figure}

For example, let $r=3$. We will denote the three colors
$c_1$, $c_2$ and $c_3$ as $R$ (red), $B$ (blue) and $G$ (green)
in the figures. Take $w=s_1s_2$. Then $\mathbf{c}=(R,B,G)$
and $w\mathbf{c}=(G,R,B)$. With $\lambda=(2,1,0)$ the
system $\mathfrak{S}_{\mathbf{z},\lambda,w}$ has two states,
which are illustrated in Figure~\ref{tostates}.

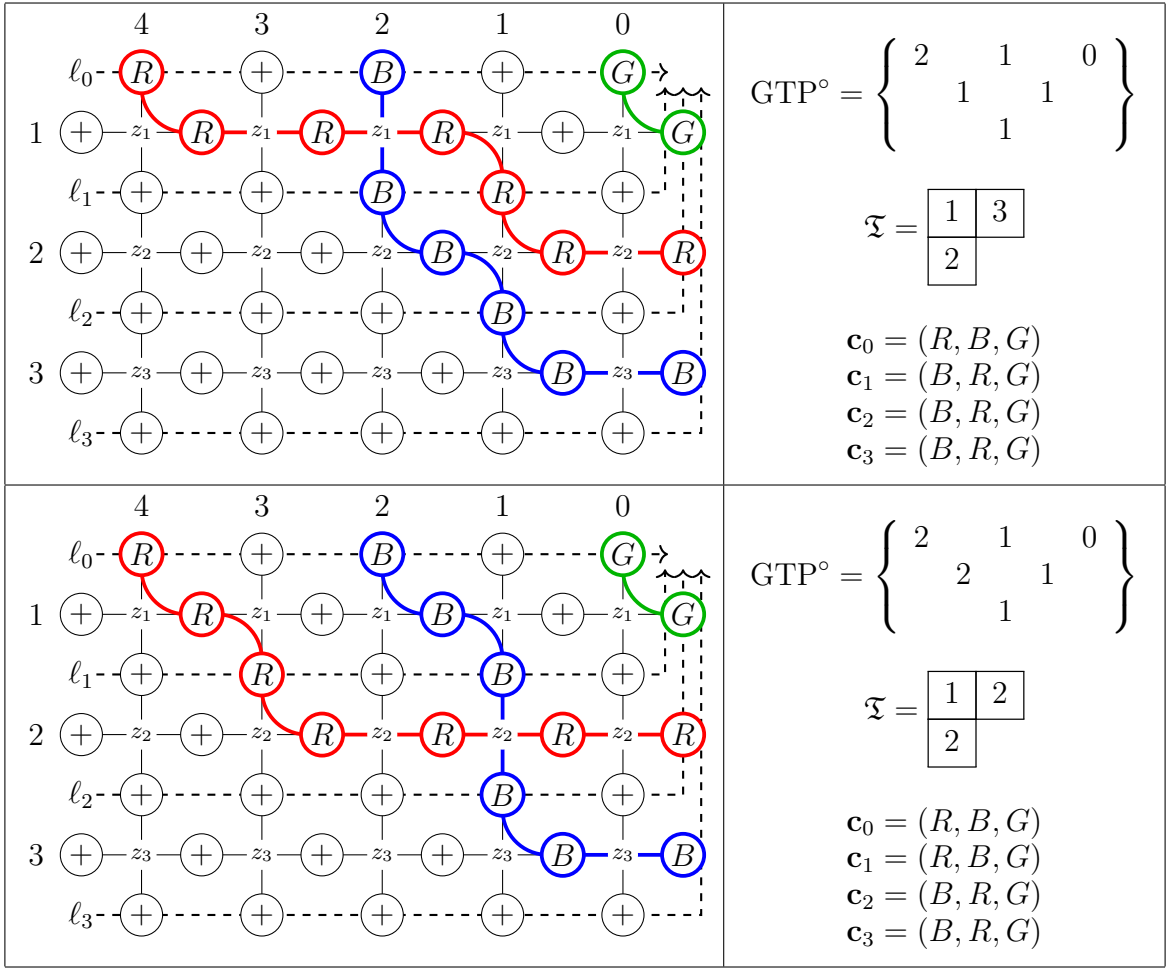
\begin{figure}
\[
\begin{array}{|c|c|}
  \hline
  \begin{tikzpicture}[scale=0.8]
    \draw (1,0)--(1,6);
    \draw (3,0)--(3,6);
    \draw (5,0)--(5,6);
    \draw (7,0)--(7,6);
    \draw (9,0)--(9,6);
    \draw (0,1)--(10,1);
    \draw (0,3)--(10,3);
    \draw (0,5)--(10,5);
    \draw [dashed,thick,->] (.25,6) -- (9.75,6);
    \node at (0,6) {$\ell_0$};
    \draw [dashed,thick,->] (.25,4) -- (9.7,4) -- (9.7,5.75);
    \node at (0,4) {$\ell_1$};
    \draw [dashed,thick,->] (.25,2) -- (10,2) -- (10,5.75);
    \node at (0,2) {$\ell_2$};
    \draw [dashed,thick,->] (.25,0) -- (10.3,0) -- (10.3,5.75);
    \node at (0,0) {$\ell_3$};
    
    \draw [line width=0.5mm,red] (1.00,5.75) to [out=-90,in=180] (1.75,5.00);
    \draw [line width=0.5mm,red] (2.35,5.00)--(3.75,5.00);
    \draw [line width=0.5mm,blue] (5,4)--(5,6);
    \draw [line width=0.5mm,red] (4.35,5.00)--(5.75,5.00);
    \draw [line width=0.5mm,red] (6.25,5.00) to [out=0,in=90] (7.00,4.25);
    \draw [line width=0.5mm,darkgreen] (9.00,5.75) to [out=-90,in=180] (9.75,5.00);
    \draw [line width=0.5mm,blue] (5.00,3.75) to [out=-90,in=180] (5.75,3.00);
    \draw [line width=0.5mm,red] (7.00,3.75) to [out=-90,in=180] (7.75,3.00);
    \draw [line width=0.5mm,blue] (6.25,3.00) to [out=0,in=90] (7.00,2.25);
    \draw [line width=0.5mm,red] (8.35,3.00)--(9.75,3.00);
    \draw [line width=0.5mm,blue] (7.00,1.75) to [out=-90,in=180] (7.75,1.00);
    \draw [line width=0.5mm,blue] (8.35,1.00)--(9.75,1.00);
    \draw[line width=0.5mm,red,fill=white] (1,6) circle (.35);
    \draw[fill=white] (3,6) circle (.35);
    \draw[line width=0.5mm,blue,fill=white] (5,6) circle (.35);
    \draw[fill=white] (7,6) circle (.35);
    \draw[line width=0.5mm,darkgreen,fill=white] (9,6) circle (.35);
    \draw[fill=white] (1,4) circle (.35);
    \draw[fill=white] (3,4) circle (.35);
    \draw[line width=0.5mm,blue,fill=white] (5,4) circle (.35);
    \draw[line width=0.5mm,red,fill=white] (7,4) circle (.35);
    \draw[fill=white] (9,4) circle (.35);
    \draw[fill=white] (1,2) circle (.35);
    \draw[fill=white] (3,2) circle (.35);
    \draw[fill=white] (5,2) circle (.35);
    \draw[line width=0.5mm,blue,fill=white] (7,2) circle (.35);
    \draw[fill=white] (9,2) circle (.35);
    \draw[fill=white] (1,0) circle (.35);
    \draw[fill=white] (3,0) circle (.35);
    \draw[fill=white] (5,0) circle (.35);
    \draw[fill=white] (7,0) circle (.35);
    \draw[fill=white] (9,0) circle (.35);
    \draw[fill=white] (0,5) circle (.35);
    \draw[line width=0.5mm,red,fill=white] (2,5) circle (.35);
    \draw[line width=0.5mm,red,fill=white] (4,5) circle (.35);
    \draw[line width=0.5mm,red,fill=white] (6,5) circle (.35);
    \draw[fill=white] (8,5) circle (.35);
    \draw[line width=0.5mm,darkgreen,fill=white] (10,5) circle (.35);
    \draw[fill=white] (0,3) circle (.35);
    \draw[fill=white] (2,3) circle (.35);
    \draw[fill=white] (4,3) circle (.35);
    \draw[line width=0.5mm,blue,fill=white] (6,3) circle (.35);
    \draw[line width=0.5mm,red,fill=white] (8,3) circle (.35);
    \draw[line width=0.5mm,red,fill=white] (10,3) circle (.35);
    \draw[fill=white] (0,1) circle (.35);
    \draw[fill=white] (2,1) circle (.35);
    \draw[fill=white] (4,1) circle (.35);
    \draw[fill=white] (6,1) circle (.35);
    \draw[line width=0.5mm,blue,fill=white] (8,1) circle (.35);
    \draw[line width=0.5mm,blue,fill=white] (10,1) circle (.35);
    \node at (1,6) {$R$};
    \node at (3,6) {$+$};
    \node at (5,6) {$B$};
    \node at (7,6) {$+$};
    \node at (9,6) {$G$};
    \node at (1,4) {$+$};
    \node at (3,4) {$+$};
    \node at (5,4) {$B$};
    \node at (7,4) {$R$};
    \node at (9,4) {$+$};
    \node at (1,2) {$+$};
    \node at (3,2) {$+$};
    \node at (5,2) {$+$};
    \node at (7,2) {$B$};
    \node at (9,2) {$+$};
    \node at (1,0) {$+$};
    \node at (3,0) {$+$};
    \node at (5,0) {$+$};
    \node at (7,0) {$+$};
    \node at (9,0) {$+$};
    \node at (0,5) {$+$};
    \node at (2,5) {$R$};
    \node at (4,5) {$R$};
    \node at (6,5) {$R$};
    \node at (8,5) {$+$};
    \node at (10,5) {$G$};
    \node at (0,3) {$+$};
    \node at (2,3) {$+$};
    \node at (4,3) {$+$};
    \node at (6,3) {$B$};
    \node at (8,3) {$R$};
    \node at (10,3) {$R$};
    \node at (0,1) {$+$};
    \node at (2,1) {$+$};
    \node at (4,1) {$+$};
    \node at (6,1) {$+$};
    \node at (8,1) {$B$};
    \node at (10,1) {$B$};
    \node at (1.00,6.8) {$ 4$};
    \node at (3.00,6.8) {$ 3$};
    \node at (5.00,6.8) {$ 2$};
    \node at (7.00,6.8) {$ 1$};
    \node at (9.00,6.8) {$ 0$};
    \node at (-.75,5) {$ 1$};
    \node at (-.75,3) {$ 2$};
    \node at (-.75,1) {$ 3$};
    \node at (0,-.50) {$\quad$};
    \path[fill=white] (1,1) circle (.25);
    \node at (1,1) {\scriptsize\scriptsize$z_3$};
    \path[fill=white] (3,1) circle (.25);
    \node at (3,1) {\scriptsize$z_3$};
    \path[fill=white] (5,1) circle (.25);
    \node at (5,1) {\scriptsize$z_3$};
    \path[fill=white] (7,1) circle (.25);
    \node at (7,1) {\scriptsize$z_3$};
    \path[fill=white] (9,1) circle (.25);
    \node at (9,1) {\scriptsize$z_3$};

    \path[fill=white] (1,3) circle (.25);
    \node at (1,3) {\scriptsize$z_2$};
    \path[fill=white] (3,3) circle (.25);
    \node at (3,3) {\scriptsize$z_2$};
    \path[fill=white] (5,3) circle (.25);
    \node at (5,3) {\scriptsize$z_2$};
    \path[fill=white] (7,3) circle (.25);
    \node at (7,3) {\scriptsize$z_2$};
    \path[fill=white] (9,3) circle (.25);
    \node at (9,3) {\scriptsize$z_2$};

    \path[fill=white] (1,5) circle (.25);
    \node at (1,5) {\scriptsize$z_1$};
    \path[fill=white] (3,5) circle (.25);
    \node at (3,5) {\scriptsize$z_1$};
    \path[fill=white] (5,5) circle (.25);
    \node at (5,5) {\scriptsize$z_1$};
    \path[fill=white] (7,5) circle (.25);
    \node at (7,5) {\scriptsize$z_1$};
    \path[fill=white] (9,5) circle (.25);
    \node at (9,5) {\scriptsize$z_1$};
  \end{tikzpicture} &
  \raisebox{80pt}{$\begin{array}{c}\GTP^\circ=\left\{\begin{array}{cccccc}2&&1&&0\\&1&&1\\&&1\end{array}\right\}\\\qquad\\
   \mathfrak{T}=\begin{ytableau}1&3\\2\end{ytableau}\\\\
   \mathbf{c}_0=(R,B,G)\\
   \mathbf{c}_1=(B,R,G)\\
   \mathbf{c}_2=(B,R,G)\\
   \mathbf{c}_3=(B,R,G)\end{array}$}\\
  \hline
  \begin{tikzpicture}[scale=0.8]
    \draw (1,0)--(1,6);
    \draw (3,0)--(3,6);
    \draw (5,0)--(5,6);
    \draw (7,0)--(7,6);
    \draw (9,0)--(9,6);
    \draw (0,1)--(10,1);
    \draw (0,3)--(10,3);
    \draw (0,5)--(10,5);
    \draw [dashed,thick,->] (.25,6) -- (9.75,6);
    \node at (0,6) {$\ell_0$};
    \draw [dashed,thick,->] (.25,4) -- (9.7,4) -- (9.7,5.75);
    \node at (0,4) {$\ell_1$};
    \draw [dashed,thick,->] (.25,2) -- (10,2) -- (10,5.75);
    \node at (0,2) {$\ell_2$};
    \draw [dashed,thick,->] (.25,0) -- (10.3,0) -- (10.3,5.75);
    \node at (0,0) {$\ell_3$};
    \draw [line width=0.5mm,red] (1.00,5.75) to [out=-90,in=180] (1.75,5.00);
    \draw [line width=0.5mm,red] (2.25,5.00) to [out=0,in=90] (3.00,4.25);
    \draw [line width=0.5mm,blue] (5.00,5.75) to [out=-90,in=180] (5.75,5.00);
    \draw [line width=0.5mm,blue] (6.25,5.00) to [out=0,in=90] (7.00,4.25);
    \draw [line width=0.5mm,darkgreen] (9.00,5.75) to [out=-90,in=180] (9.75,5.00);
    \draw [line width=0.5mm,red] (3.00,3.75) to [out=-90,in=180] (3.75,3.00);
    \draw [line width=0.5mm,red] (4.35,3.00)--(5.75,3.00);
    \draw [line width=0.5mm,blue] (7,2)--(7,4);
    \draw [line width=0.5mm,red] (6.35,3.00)--(7.75,3.00);
    \draw [line width=0.5mm,red] (8.35,3.00)--(9.75,3.00);
    \draw [line width=0.5mm,blue] (7.00,1.75) to [out=-90,in=180] (7.75,1.00);
    \draw [line width=0.5mm,blue] (8.35,1.00)--(9.75,1.00);
    \draw[line width=0.5mm,red,fill=white] (1,6) circle (.35);
    \draw[fill=white] (3,6) circle (.35);
    \draw[line width=0.5mm,blue,fill=white] (5,6) circle (.35);
    \draw[fill=white] (7,6) circle (.35);
    \draw[line width=0.5mm,darkgreen,fill=white] (9,6) circle (.35);
    \draw[fill=white] (1,4) circle (.35);
    \draw[line width=0.5mm,red,fill=white] (3,4) circle (.35);
    \draw[fill=white] (5,4) circle (.35);
    \draw[line width=0.5mm,blue,fill=white] (7,4) circle (.35);
    \draw[fill=white] (9,4) circle (.35);
    \draw[fill=white] (1,2) circle (.35);
    \draw[fill=white] (3,2) circle (.35);
    \draw[fill=white] (5,2) circle (.35);
    \draw[line width=0.5mm,blue,fill=white] (7,2) circle (.35);
    \draw[fill=white] (9,2) circle (.35);
    \draw[fill=white] (1,0) circle (.35);
    \draw[fill=white] (3,0) circle (.35);
    \draw[fill=white] (5,0) circle (.35);
    \draw[fill=white] (7,0) circle (.35);
    \draw[fill=white] (9,0) circle (.35);
    \draw[fill=white] (0,5) circle (.35);
    \draw[line width=0.5mm,red,fill=white] (2,5) circle (.35);
    \draw[fill=white] (4,5) circle (.35);
    \draw[line width=0.5mm,blue,fill=white] (6,5) circle (.35);
    \draw[fill=white] (8,5) circle (.35);
    \draw[line width=0.5mm,darkgreen,fill=white] (10,5) circle (.35);
    \draw[fill=white] (0,3) circle (.35);
    \draw[fill=white] (2,3) circle (.35);
    \draw[line width=0.5mm,red,fill=white] (4,3) circle (.35);
    \draw[line width=0.5mm,red,fill=white] (6,3) circle (.35);
    \draw[line width=0.5mm,red,fill=white] (8,3) circle (.35);
    \draw[line width=0.5mm,red,fill=white] (10,3) circle (.35);
    \draw[fill=white] (0,1) circle (.35);
    \draw[fill=white] (2,1) circle (.35);
    \draw[fill=white] (4,1) circle (.35);
    \draw[fill=white] (6,1) circle (.35);
    \draw[line width=0.5mm,blue,fill=white] (8,1) circle (.35);
    \draw[line width=0.5mm,blue,fill=white] (10,1) circle (.35);
    \node at (1,6) {$R$};
    \node at (3,6) {$+$};
    \node at (5,6) {$B$};
    \node at (7,6) {$+$};
    \node at (9,6) {$G$};
    \node at (1,4) {$+$};
    \node at (3,4) {$R$};
    \node at (5,4) {$+$};
    \node at (7,4) {$B$};
    \node at (9,4) {$+$};
    \node at (1,2) {$+$};
    \node at (3,2) {$+$};
    \node at (5,2) {$+$};
    \node at (7,2) {$B$};
    \node at (9,2) {$+$};
    \node at (1,0) {$+$};
    \node at (3,0) {$+$};
    \node at (5,0) {$+$};
    \node at (7,0) {$+$};
    \node at (9,0) {$+$};
    \node at (0,5) {$+$};
    \node at (2,5) {$R$};
    \node at (4,5) {$+$};
    \node at (6,5) {$B$};
    \node at (8,5) {$+$};
    \node at (10,5) {$G$};
    \node at (0,3) {$+$};
    \node at (2,3) {$+$};
    \node at (4,3) {$R$};
    \node at (6,3) {$R$};
    \node at (8,3) {$R$};
    \node at (10,3) {$R$};
    \node at (0,1) {$+$};
    \node at (2,1) {$+$};
    \node at (4,1) {$+$};
    \node at (6,1) {$+$};
    \node at (8,1) {$B$};
    \node at (10,1) {$B$};
    \node at (1.00,6.8) {$4$};
    \node at (3.00,6.8) {$3$};
    \node at (5.00,6.8) {$2$};
    \node at (7.00,6.8) {$1$};
    \node at (9.00,6.8) {$0$};
    \node at (-.75,5) {$ 1$};
    \node at (-.75,3) {$ 2$};
    \node at (-.75,1) {$ 3$};
    \path[fill=white] (1,1) circle (.25);
    \node at (1,1) {\scriptsize$z_3$};
    \path[fill=white] (3,1) circle (.25);
    \node at (3,1) {\scriptsize$z_3$};
    \path[fill=white] (5,1) circle (.25);
    \node at (5,1) {\scriptsize$z_3$};
    \path[fill=white] (7,1) circle (.25);
    \node at (7,1) {\scriptsize$z_3$};
    \path[fill=white] (9,1) circle (.25);
    \node at (9,1) {\scriptsize$z_3$};

    \path[fill=white] (1,3) circle (.25);
    \node at (1,3) {\scriptsize$z_2$};
    \path[fill=white] (3,3) circle (.25);
    \node at (3,3) {\scriptsize$z_2$};
    \path[fill=white] (5,3) circle (.25);
    \node at (5,3) {\scriptsize$z_2$};
    \path[fill=white] (7,3) circle (.25);
    \node at (7,3) {\scriptsize$z_2$};
    \path[fill=white] (9,3) circle (.25);
    \node at (9,3) {\scriptsize$z_2$};

    \path[fill=white] (1,5) circle (.25);
    \node at (1,5) {\scriptsize$z_1$};
    \path[fill=white] (3,5) circle (.25);
    \node at (3,5) {\scriptsize$z_1$};
    \path[fill=white] (5,5) circle (.25);
    \node at (5,5) {\scriptsize$z_1$};
    \path[fill=white] (7,5) circle (.25);
    \node at (7,5) {\scriptsize$z_1$};
    \path[fill=white] (9,5) circle (.25);
    \node at (9,5) {\scriptsize$z_1$};

    \node at (0,-.50) {$\quad$};
  \end{tikzpicture}&
  \raisebox{80pt}{$\begin{array}{c}\GTP^\circ=\left\{\begin{array}{cccccc}2&&1&&0\\&2&&1\\&&1\end{array}\right\}\\\qquad\\
   \mathfrak{T}=\begin{ytableau}1&2\\2\end{ytableau}\\\\
   \mathbf{c}_0=(R,B,G)\\
   \mathbf{c}_1=(R,B,G)\\
   \mathbf{c}_2=(B,R,G)\\
   \mathbf{c}_3=(B,R,G)\end{array}$}\\
  \hline
\end{array}
\]
\caption{The two states of the system $\mathfrak{S}_{\mathbf{z},(2,1,0),s_1s_2}$
where $\mathbf{c}=(R,B,G)$ (red, blue, green) and $w\mathbf{c} = s_1s_2\mathbf{c}=(G,R,B)$.
The dashed lines $\ell_i$, and the intermediate flags $\mathbf{c}_i$ will
be used in the proof of Theorem~\ref{crystalmt}. Each intermediate
flag $\mathbf{c}_i$ is the sequence of colors through the line $\ell_i$,
and is obtained from the previous $\mathbf{c}_{i-1}$ by interchanging
some colors on the vertical edges that intersect it. Because $\ell_{r-1}$
only intersects one vertical edge, no interchanges are possible at the
last step, meaning that $\mathbf{c}_{r-1}=\mathbf{c}_r$. 
Note that, while the flag $w\mathbf{c} = s_1s_2\mathbf{c}=(G,R,B)$ denoting the right boundary condition is read from the top down, the last line $\ell_3$ intersects the same edges from the bottom
up.
Thus, $\mathbf{c}_3=w_0s_1s_2\mathbf{c} = (B,R,G)$. }
\label{tostates}
\end{figure}

\begin{proposition} \label{statedpf} 
For any dominant weight $\lambda$, $\mathfrak{S}_{\mathbf{z},\lambda} = \bigsqcup_{w \in W} \mathfrak{S}_{\mathbf{z},\lambda,w}$
(disjoint union) where a colored spin $c_i$ is mapped to spin $-$, and hence
  \[Z(\mathfrak{S}_{\mathbf{z},\lambda})=\sum_{w\in W}Z(\mathfrak{S}_{\mathbf{z},\lambda,w}).\]
\end{proposition}

\begin{proof}
  We may begin with a state of the uncolored system and assign colors to the
  edges with $-$ spins. Along the top row, assign color $c_i$ to
  the $-$ spin in column $\lambda_i+r-i$ as directed for colored ice states.
  We will argue that there is a unique way of coloring the remaining
  $-$~spins that is consistent with the configurations in Figure~\ref{coloredweights}.

  The boundary spins on the left
  edge are all $+$, so they do not need colors assigned.
  After this, we proceed inductively, rightwards and downwards row by row, adding color to the $-$~spins of the state using the weights from Figure~\ref{coloredweights}.
  The key observation is that at a vertex labeled as follows:
  \[\gammaice{a}{b}{c}{d} \]
  the colored spins $a$ and $b$ and the spins $\pm$
  of $c$ and $d$ determine a unique color at $c$ and $d$ with non-zero weight
  according to Figure~\ref{coloredweights}. Indeed, colored spin is conserved at a
  vertex, meaning that the total incoming (top and left) colored spins counted
  with multiplicity equals the total outgoing (bottom and right) colored
  spins. Moreover for the $\tt{a}_2$ configurations if $a$ and $b$ are of
  different colors, then $d$ will be the smaller of the two colors.
  We see that the assignment of colors is completely deterministic, and the
  colored state falls into a unique one of the ensembles
  $\mathfrak{S}_{\mathbf{z},\lambda,w}$.

  Now mapping colored spins $c_i$ to spin $-$, the colored Boltzmann weights of
  Figure~\ref{coloredweights} map to the uncolored Boltzmann weights of
  Figures~\ref{uncoloredbw}, thus proving both statements.
\end{proof}

There is again a Yang-Baxter equation.

\begin{theorem} \label{coloredybethm}
  Using the Boltzmann weights in Figure~\ref{coloredweights} for the regular vertices and the R-matrix
  in Figure~\ref{coloredrmatrix} for the rotated vertices, let $a,b,c,d,e,f$ be colored spins.
  Then the partition functions of the (now colored) systems depicted in \eqref{eqn:ybe} are equal.
\end{theorem}

\begin{proof}
  In order for either side of (\ref{eqn:ybe}) to be nonzero, each
  color that appears on a boundary edge $a,b,c,d,e,f$ must
  appear an even number of times (and therefore at least twice), since otherwise according to Figures~\ref{coloredweights} and~\ref{coloredrmatrix},
  the Boltzmann weight of the state is zero.
  Therefore at most 3 colors can appear among
  $a,b,c,d,e,f$ and the interior edges cannot involve any further
  colors. Thus there are only a fixed finite number ($4^6=4096$) of cases to be
  considered (independent of the number of colors $r$), and this can easily be checked using a computer.
  (To check this we used the Sage mathematical software.)
\end{proof}

\begin{figure}[h]
\[\begin{array}{|c|c|c|c|c|c|}
  \hline
\begin{tikzpicture}[scale=0.7]
\draw (0,0) to [out = 0, in = 180] (2,2);
\draw (0,2) to [out = 0, in = 180] (2,0);
\draw[fill=white] (0,0) circle (.35);
\draw[fill=white] (0,2) circle (.35);
\draw[fill=white] (2,0) circle (.35);
\draw[fill=white] (2,2) circle (.35);
\node at (0,0) {$+$};
\node at (0,2) {$+$};
\node at (2,2) {$+$};
\node at (2,0) {$+$};
\path[fill=white] (1,1) circle (.3);
\node at (1,1) {$R_{z_i,z_j}$};
\end{tikzpicture}&
\begin{tikzpicture}[scale=0.7]
\draw (0,0) to [out = 0, in = 180] (2,2);
\draw (0,2) to [out = 0, in = 180] (2,0);
\draw[fill=white] (0,0) circle (.35);
\draw[line width=0.5mm, blue, fill=white] (0,2) circle (.35);
\draw[line width=0.5mm, blue, fill=white] (2,2) circle (.35);
\draw[fill=white] (2,0) circle (.35);
\node at (0,0) {$+$};
\node at (0,2) {$B$};
\node at (2,2) {$B$};
\node at (2,0) {$+$};
\path[fill=white] (1,1) circle (.3);
\node at (1,1) {$R_{z_i,z_j}$};
\end{tikzpicture}&
\begin{tikzpicture}[scale=0.7]
\draw (0,0) to [out = 0, in = 180] (2,2);
\draw (0,2) to [out = 0, in = 180] (2,0);
\draw[fill=white] (0,0) circle (.35);
\draw[line width=0.5mm, red, fill=white] (0,2) circle (.35);
\draw[line width=0.5mm, red, fill=white] (2,2) circle (.35);
\draw[fill=white] (2,0) circle (.35);
\node at (0,0) {$+$};
\node at (0,2) {$R$};
\node at (2,2) {$R$};
\node at (2,0) {$+$};
\path[fill=white] (1,1) circle (.3);
\node at (1,1) {$R_{z_i,z_j}$};
\end{tikzpicture}&
\begin{tikzpicture}[scale=0.7]
\draw (0,0) to [out = 0, in = 180] (2,2);
\draw (0,2) to [out = 0, in = 180] (2,0);
\draw[line width=0.5mm, blue, fill=white] (0,0) circle (.35);
\draw[fill=white] (0,2) circle (.35);
\draw[fill=white] (2,2) circle (.35);
\draw[line width=0.5mm, blue, fill=white] (2,0) circle (.35);
\node at (0,0) {$B$};
\node at (0,2) {$+$};
\node at (2,2) {$+$};
\node at (2,0) {$B$};
\path[fill=white] (1,1) circle (.3);
\node at (1,1) {$R_{z_i,z_j}$};
\end{tikzpicture}&
\begin{tikzpicture}[scale=0.7]
\draw (0,0) to [out = 0, in = 180] (2,2);
\draw (0,2) to [out = 0, in = 180] (2,0);
\draw[line width=0.5mm, red, fill=white] (0,0) circle (.35);
\draw[fill=white] (0,2) circle (.35);
\draw[fill=white] (2,2) circle (.35);
\draw[line width=0.5mm, red, fill=white] (2,0) circle (.35);
\node at (0,0) {$R$};
\node at (0,2) {$+$};
\node at (2,2) {$+$};
\node at (2,0) {$R$};
\path[fill=white] (1,1) circle (.3);
\node at (1,1) {$R_{z_i,z_j}$};
\end{tikzpicture}&
\begin{tikzpicture}[scale=0.7]
\draw (0,0) to [out = 0, in = 180] (2,2);
\draw (0,2) to [out = 0, in = 180] (2,0);
\draw[line width=0.5mm, blue, fill=white] (0,0) circle (.35);
\draw[fill=white] (0,2) circle (.35);
\draw[line width=0.5mm, blue, fill=white] (2,2) circle (.35);
\draw[fill=white] (2,0) circle (.35);
\node at (0,0) {$B$};
\node at (0,2) {$+$};
\node at (2,2) {$B$};
\node at (2,0) {$+$};
\path[fill=white] (1,1) circle (.3);
\node at (1,1) {$R_{z_i,z_j}$};
\end{tikzpicture}\\
   \hline
z_j & z_j & z_j & z_i & z_i & z_i - z_j\\
   \hline
   \hline
\begin{tikzpicture}[scale=0.7]
\draw (0,0) to [out = 0, in = 180] (2,2);
\draw (0,2) to [out = 0, in = 180] (2,0);
\draw[line width=0.5mm, red, fill=white] (0,0) circle (.35);
\draw[fill=white] (0,2) circle (.35);
\draw[line width=0.5mm, red, fill=white] (2,2) circle (.35);
\draw[fill=white] (2,0) circle (.35);
\node at (0,0) {$R$};
\node at (0,2) {$+$};
\node at (2,2) {$R$};
\node at (2,0) {$+$};
\path[fill=white] (1,1) circle (.3);
\node at (1,1) {$R_{z_i,z_j}$};
\end{tikzpicture}&
\begin{tikzpicture}[scale=0.7]
\draw (0,0) to [out = 0, in = 180] (2,2);
\draw (0,2) to [out = 0, in = 180] (2,0);
\draw[line width=0.5mm, blue, fill=white] (0,0) circle (.35);
\draw[line width=0.5mm, blue, fill=white] (0,2) circle (.35);
\draw[line width=0.5mm, blue, fill=white] (2,2) circle (.35);
\draw[line width=0.5mm, blue, fill=white] (2,0) circle (.35);
\node at (0,0) {$B$};
\node at (0,2) {$B$};
\node at (2,2) {$B$};
\node at (2,0) {$B$};
\path[fill=white] (1,1) circle (.3);
\node at (1,1) {$R_{z_i,z_j}$};
\end{tikzpicture}&
\begin{tikzpicture}[scale=0.7]
\draw (0,0) to [out = 0, in = 180] (2,2);
\draw (0,2) to [out = 0, in = 180] (2,0);
\draw[line width=0.5mm, blue, fill=white] (0,0) circle (.35);
\draw[line width=0.5mm, red, fill=white] (0,2) circle (.35);
\draw[line width=0.5mm, red, fill=white] (2,2) circle (.35);
\draw[line width=0.5mm, blue, fill=white] (2,0) circle (.35);
\node at (0,0) {$B$};
\node at (0,2) {$R$};
\node at (2,2) {$R$};
\node at (2,0) {$B$};
\path[fill=white] (1,1) circle (.3);
\node at (1,1) {$R_{z_i,z_j}$};
\end{tikzpicture}&
\begin{tikzpicture}[scale=0.7]
\draw (0,0) to [out = 0, in = 180] (2,2);
\draw (0,2) to [out = 0, in = 180] (2,0);
\draw[line width=0.5mm, red, fill=white] (0,0) circle (.35);
\draw[line width=0.5mm, blue, fill=white] (0,2) circle (.35);
\draw[line width=0.5mm, blue, fill=white] (2,2) circle (.35);
\draw[line width=0.5mm, red, fill=white] (2,0) circle (.35);
\node at (0,0) {$R$};
\node at (0,2) {$B$};
\node at (2,2) {$B$};
\node at (2,0) {$R$};
\path[fill=white] (1,1) circle (.3);
\node at (1,1) {$R_{z_i,z_j}$};
\end{tikzpicture}&
\begin{tikzpicture}[scale=0.7]
\draw (0,0) to [out = 0, in = 180] (2,2);
\draw (0,2) to [out = 0, in = 180] (2,0);
\draw[line width=0.5mm, red, fill=white] (0,0) circle (.35);
\draw[line width=0.5mm, blue, fill=white] (0,2) circle (.35);
\draw[line width=0.5mm, red, fill=white] (2,2) circle (.35);
\draw[line width=0.5mm, blue, fill=white] (2,0) circle (.35);
\node at (0,0) {$R$};
\node at (0,2) {$B$};
\node at (2,2) {$R$};
\node at (2,0) {$B$};
\path[fill=white] (1,1) circle (.3);
\node at (1,1) {$R_{z_i,z_j}$};
\end{tikzpicture}&
\begin{tikzpicture}[scale=0.7]
\draw (0,0) to [out = 0, in = 180] (2,2);
\draw (0,2) to [out = 0, in = 180] (2,0);
\draw[line width=0.5mm, red, fill=white] (0,0) circle (.35);
\draw[line width=0.5mm, red, fill=white] (0,2) circle (.35);
\draw[line width=0.5mm, red, fill=white] (2,2) circle (.35);
\draw[line width=0.5mm, red, fill=white] (2,0) circle (.35);
\node at (0,0) {$R$};
\node at (0,2) {$R$};
\node at (2,2) {$R$};
\node at (2,0) {$R$};
\path[fill=white] (1,1) circle (.3);
\node at (1,1) {$R_{z_i,z_j}$};
\end{tikzpicture}\\
   \hline
z_i-z_j & z_i & z_i & z_j & z_i-z_j & z_i\\
   \hline
\end{array}
\]
\caption{The colored R-matrix.}
\label{coloredrmatrix}
\end{figure}
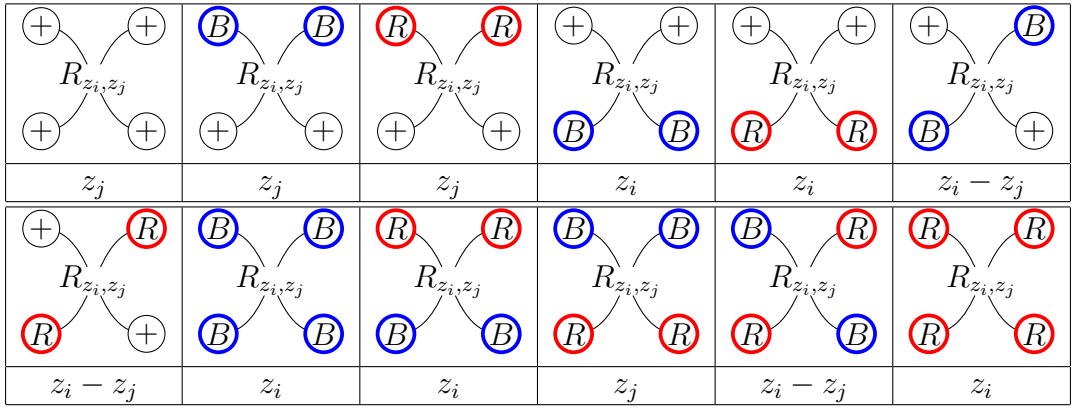

\begin{figure}[h]
\vspace{5mm}
\begin{tikzpicture}[scale=0.75, font=\small]
    \draw (1,0)--(1,6);
    \draw (3,0)--(3,6);
    \draw (5,0)--(5,6);
    \draw (7,0)--(7,6);
    \draw (9,0)--(9,6);
    \draw (11,0)--(11,6);
    \draw (13,0)--(13,6);
    \draw (15,0)--(15,6);
    \draw (0,1)--(16,1);
    \draw (0,3)--(16,3);
    \draw (0,5)--(16,5);
    \draw [line width=0.5mm,red] (1.00,5.75) to [out=-90,in=180] (1.75,5.00);
    \draw [line width=0.5mm,red] (2.35,5.00)--(3.75,5.00);
    \draw [line width=0.5mm,red] (4.35,5.00)--(5.75,5.00);
    \draw [line width=0.5mm,red] (6.35,5.00)--(7.75,5.00);
    \draw [line width=0.5mm,blue] (9,4)--(9,6);
    \draw [line width=0.5mm,red] (8.35,5.00)--(9.75,5.00);
    \draw [line width=0.5mm,red] (10.35,5.00)--(11.75,5.00);
    \draw [line width=0.5mm,red] (12.35,5.00)--(13.75,5.00);
    \draw [line width=0.5mm,darkgreen] (15,4)--(15,6);
    \draw [line width=0.5mm,red] (14.35,5.00)--(15.75,5.00);
    \draw [line width=0.5mm,blue] (9.00,3.75) to [out=-90,in=180] (9.75,3.00);
    \draw [line width=0.5mm,blue] (10.35,3.00)--(11.75,3.00);
    \draw [line width=0.5mm,blue] (12.35,3.00)--(13.75,3.00);
    \draw [line width=0.5mm,darkgreen] (15,2)--(15,4);
    \draw [line width=0.5mm,blue] (14.35,3.00)--(15.75,3.00);
    \draw [line width=0.5mm,darkgreen] (15.00,1.75) to [out=-90,in=180] (15.75,1.00);
    \draw[line width=0.5mm,red,fill=white] (1,6) circle (.35);
    \draw[fill=white] (3,6) circle (.35);
    \draw[fill=white] (5,6) circle (.35);
    \draw[fill=white] (7,6) circle (.35);
    \draw[line width=0.5mm,blue,fill=white] (9,6) circle (.35);
    \draw[fill=white] (11,6) circle (.35);
    \draw[fill=white] (13,6) circle (.35);
    \draw[line width=0.5mm,darkgreen,fill=white] (15,6) circle (.35);
    \draw[fill=white] (1,4) circle (.35);
    \draw[fill=white] (3,4) circle (.35);
    \draw[fill=white] (5,4) circle (.35);
    \draw[fill=white] (7,4) circle (.35);
    \draw[line width=0.5mm,blue,fill=white] (9,4) circle (.35);
    \draw[fill=white] (11,4) circle (.35);
    \draw[fill=white] (13,4) circle (.35);
    \draw[line width=0.5mm,darkgreen,fill=white] (15,4) circle (.35);
    \draw[fill=white] (1,2) circle (.35);
    \draw[fill=white] (3,2) circle (.35);
    \draw[fill=white] (5,2) circle (.35);
    \draw[fill=white] (7,2) circle (.35);
    \draw[fill=white] (9,2) circle (.35);
    \draw[fill=white] (11,2) circle (.35);
    \draw[fill=white] (13,2) circle (.35);
    \draw[line width=0.5mm,darkgreen,fill=white] (15,2) circle (.35);
    \draw[fill=white] (1,0) circle (.35);
    \draw[fill=white] (3,0) circle (.35);
    \draw[fill=white] (5,0) circle (.35);
    \draw[fill=white] (7,0) circle (.35);
    \draw[fill=white] (9,0) circle (.35);
    \draw[fill=white] (11,0) circle (.35);
    \draw[fill=white] (13,0) circle (.35);
    \draw[fill=white] (15,0) circle (.35);
    \draw[fill=white] (0,5) circle (.35);
    \draw[line width=0.5mm,red,fill=white] (2,5) circle (.35);
    \draw[line width=0.5mm,red,fill=white] (4,5) circle (.35);
    \draw[line width=0.5mm,red,fill=white] (6,5) circle (.35);
    \draw[line width=0.5mm,red,fill=white] (8,5) circle (.35);
    \draw[line width=0.5mm,red,fill=white] (10,5) circle (.35);
    \draw[line width=0.5mm,red,fill=white] (12,5) circle (.35);
    \draw[line width=0.5mm,red,fill=white] (14,5) circle (.35);
    \draw[line width=0.5mm,red,fill=white] (16,5) circle (.35);
    \draw[fill=white] (0,3) circle (.35);
    \draw[fill=white] (2,3) circle (.35);
    \draw[fill=white] (4,3) circle (.35);
    \draw[fill=white] (6,3) circle (.35);
    \draw[fill=white] (8,3) circle (.35);
    \draw[line width=0.5mm,blue,fill=white] (10,3) circle (.35);
    \draw[line width=0.5mm,blue,fill=white] (12,3) circle (.35);
    \draw[line width=0.5mm,blue,fill=white] (14,3) circle (.35);
    \draw[line width=0.5mm,blue,fill=white] (16,3) circle (.35);
    \draw[fill=white] (0,1) circle (.35);
    \draw[fill=white] (2,1) circle (.35);
    \draw[fill=white] (4,1) circle (.35);
    \draw[fill=white] (6,1) circle (.35);
    \draw[fill=white] (8,1) circle (.35);
    \draw[fill=white] (10,1) circle (.35);
    \draw[fill=white] (12,1) circle (.35);
    \draw[fill=white] (14,1) circle (.35);
    \draw[line width=0.5mm,darkgreen,fill=white] (16,1) circle (.35);
    \node at (1,6) {$R$};
    \node at (3,6) {$+$};
    \node at (5,6) {$+$};
    \node at (7,6) {$+$};
    \node at (9,6) {$B$};
    \node at (11,6) {$+$};
    \node at (13,6) {$+$};
    \node at (15,6) {$G$};
    \node at (1,4) {$+$};
    \node at (3,4) {$+$};
    \node at (5,4) {$+$};
    \node at (7,4) {$+$};
    \node at (9,4) {$B$};
    \node at (11,4) {$+$};
    \node at (13,4) {$+$};
    \node at (15,4) {$G$};
    \node at (1,2) {$+$};
    \node at (3,2) {$+$};
    \node at (5,2) {$+$};
    \node at (7,2) {$+$};
    \node at (9,2) {$+$};
    \node at (11,2) {$+$};
    \node at (13,2) {$+$};
    \node at (15,2) {$G$};
    \node at (1,0) {$+$};
    \node at (3,0) {$+$};
    \node at (5,0) {$+$};
    \node at (7,0) {$+$};
    \node at (9,0) {$+$};
    \node at (11,0) {$+$};
    \node at (13,0) {$+$};
    \node at (15,0) {$+$};
    \node at (0,5) {$+$};
    \node at (2,5) {$R$};
    \node at (4,5) {$R$};
    \node at (6,5) {$R$};
    \node at (8,5) {$R$};
    \node at (10,5) {$R$};
    \node at (12,5) {$R$};
    \node at (14,5) {$R$};
    \node at (16,5) {$R$};
    \node at (0,3) {$+$};
    \node at (2,3) {$+$};
    \node at (4,3) {$+$};
    \node at (6,3) {$+$};
    \node at (8,3) {$+$};
    \node at (10,3) {$B$};
    \node at (12,3) {$B$};
    \node at (14,3) {$B$};
    \node at (16,3) {$B$};
    \node at (0,1) {$+$};
    \node at (2,1) {$+$};
    \node at (4,1) {$+$};
    \node at (6,1) {$+$};
    \node at (8,1) {$+$};
    \node at (10,1) {$+$};
    \node at (12,1) {$+$};
    \node at (14,1) {$+$};
    \node at (16,1) {$G$};
    \node at (1,6.8) {$7$};
    \node at (3,6.8) {$6$};
    \node at (5,6.8) {$5$};
    \node at (7,6.8) {$4$};
    \node at (9,6.8) {$3$};
    \node at (11,6.8) {$2$};
    \node at (13,6.8) {$1$};
    \node at (15,6.8) {$0$};
    \node at (-.75,5) {$ 1$};
    \node at (-.75,3) {$ 2$};
    \node at (-.75,1) {$ 3$};
    \path[fill=white] (1,1) circle (.25);
    \node at (1,1) {\scriptsize$z_3$};
    \path[fill=white] (3,1) circle (.25);
    \node at (3,1) {\scriptsize$z_3$};
    \path[fill=white] (5,1) circle (.25);
    \node at (5,1) {\scriptsize$z_3$};
    \path[fill=white] (7,1) circle (.25);
    \node at (7,1) {\scriptsize$z_3$};
    \path[fill=white] (9,1) circle (.25);
    \node at (9,1) {\scriptsize$z_3$};
    \path[fill=white] (11,1) circle (.25);
    \node at (11,1) {\scriptsize$z_3$};
    \path[fill=white] (13,1) circle (.25);
    \node at (13,1) {\scriptsize$z_3$};
    \path[fill=white] (15,1) circle (.25);
    \node at (15,1) {\scriptsize$z_3$};

    \path[fill=white] (1,3) circle (.25);
    \node at (1,3) {\scriptsize$z_2$};
    \path[fill=white] (3,3) circle (.25);
    \node at (3,3) {\scriptsize$z_2$};
    \path[fill=white] (5,3) circle (.25);
    \node at (5,3) {\scriptsize$z_2$};
    \path[fill=white] (7,3) circle (.25);
    \node at (7,3) {\scriptsize$z_2$};
    \path[fill=white] (9,3) circle (.25);
    \node at (9,3) {\scriptsize$z_2$};
    \path[fill=white] (11,3) circle (.25);
    \node at (11,3) {\scriptsize$z_2$};
    \path[fill=white] (13,3) circle (.25);
    \node at (13,3) {\scriptsize$z_2$};
    \path[fill=white] (15,3) circle (.25);
    \node at (15,3) {\scriptsize$z_2$};

    \path[fill=white] (1,5) circle (.25);
    \node at (1,5) {\scriptsize$z_1$};
    \path[fill=white] (3,5) circle (.25);
    \node at (3,5) {\scriptsize$z_1$};
    \path[fill=white] (5,5) circle (.25);
    \node at (5,5) {\scriptsize$z_1$};
    \path[fill=white] (7,5) circle (.25);
    \node at (7,5) {\scriptsize$z_1$};
    \path[fill=white] (9,5) circle (.25);
    \node at (9,5) {\scriptsize$z_1$};
    \path[fill=white] (11,5) circle (.25);
    \node at (11,5) {\scriptsize$z_1$};
    \path[fill=white] (13,5) circle (.25);
    \node at (13,5) {\scriptsize$z_1$};
    \path[fill=white] (15,5) circle (.25);
    \node at (15,5) {\scriptsize$z_1$};

  \end{tikzpicture}
\vglue3pt
\[
\GTP(\mathfrak{s})=\left\{\begin{array}{ccccc}7&&3&&0\\&3&&0\\&&0\end{array}\right\},\quad
\GTP^\circ(\mathfrak{s})=\left\{\begin{array}{ccccc}5&&2&&0\\&2&&0\\&&0\end{array}\right\},\]
\vglue5pt
\[\mathfrak{T}(\mathfrak{s})=\begin{ytableau}2&2&3&3&3\\3&3\end{ytableau}\;,\qquad
\bzl_{(1,2,1)}(\mathfrak{T})=\left[\vcenter{\xymatrix@-1.5pc{*+[o][F-]{0}&*+[o][F-]{0}\\&*+[o][F-]{0}}}\right]\;.
\]
  \caption{The ground state. In this unique state with maximal number of
    crossings of colored lines, we have
    $\beta(\mathfrak{s})=\mathbf{z}^{\lambda+\rho}$,
    $\wt(\mathfrak{T}(\mathfrak{s}))=\mathbf{z}^{w_0(\lambda+\rho)}$.}
  \label{groundstate}
\end{figure}

\begin{remark}
\label{superrem}
It may be checked that the colored R-matrix (with $r$ colors) in
Figure~\ref{coloredrmatrix} is the limit as $q\to \infty$ of the R-matrix of a
Drinfeld twist of $U_q\big(\widehat{\mathfrak{sl}}(r|1)\big)$. It is
also possible to vary the Boltzmann weights as follows: in
Figure~\ref{coloredweights}, omit the $\tt{a}_2$ patterns in
which all four edges have the same color; and in Figure~\ref{coloredrmatrix},
change the Boltzmann weights of the patterns in which all four
edges have the same color from $z_i$ to $z_j$. These changes do
not affect any of the arguments in this paper since the changed patterns
do not appear in any of the states of the systems we consider,
but they change the underlying quantum group to a Drinfeld twist
of $U_q(\widehat{\mathfrak{sl}}_{r+1})$.
\end{remark}

Our next result shows that the colored partition function with $r$ colors 
and $r$ rows is a polynomial Demazure atom for
$\mathrm{GL}(r)$ up to a factor of $\mathbf{z}^{\rho}$.

\begin{theorem}
  \label{zdematoms}
  For every $w\in W$ we have
  \[Z(\mathfrak{S}_{\mathbf{z},\lambda,w})=\mathbf{z}^\rho\partial^\circ_w\mathbf{z}^\lambda.\]
\end{theorem}

\begin{proof}
  The proof is by induction with respect to Bruhat order. If $w=1_W$, it is easy to see that there is a unique state in
  $\mathfrak{S}_{\mathbf{z},\lambda,1_W}$ and its
  Boltzmann weight is $\mathbf{z}^{\rho+\lambda}$ (see Figure~\ref{groundstate}).
  Thus it suffices to show that for each $s_i$ and $w$ with $s_i w > w$,
  \begin{equation}
    \label{lemrecurse}
    \mathbf{z}^{- \rho} Z (\mathfrak{S}_{\mathbf{z}, \lambda, s_i w}) 
    = \partial^\circ_i \bigl(\mathbf{z}^{- \rho} Z
    (\mathfrak{S}_{\mathbf{z}, \lambda, w})\bigr).
  \end{equation}

  Let $w\mathbf{c}=\mathbf{d}= (d_1, \cdots, d_r)$. Since $s_i w >
  w$, we have $d_i > d_{i + 1}$. Consider the partition function
  of the system in Figure~\ref{modpf} (top). This is a system like the one portrayed in Figure~\ref{tostates} but with an attached rotated vertex $z_{i + 1}, z_i$ on the left.   We only exhibit two of the rows of the system because this is where the interesting changes occur.  
Also note that the parameters of the two rows are flipped, so now the top row has parameter $z_{i+1}$ and the bottom row has parameter $z_i$. 
  
  Consulting Figure~\ref{coloredrmatrix}, the rotated vertex (or the R-matrix) has only one possible admissible
  configuration (with all $+$ spins). This means the partition function of the top system in Figure~\ref{modpf} will be equal to the Boltzmann weight of
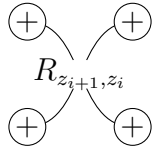
\[\begin{tikzpicture}[scale=0.7]
\draw (0,0) to [out = 0, in = 180] (2,2);
\draw (0,2) to [out = 0, in = 180] (2,0);
\draw[fill=white] (0,0) circle (.35);
\draw[fill=white] (0,2) circle (.35);
\draw[fill=white] (2,0) circle (.35);
\draw[fill=white] (2,2) circle (.35);
\node at (0,0) {$+$};
\node at (0,2) {$+$};
\node at (2,2) {$+$};
\node at (2,0) {$+$};
\path[fill=white] (1,1) circle (.3);
\node at (1,1) {$R_{z_{i+1},z_i}$};
\end{tikzpicture}
\]
times the partition function of the system with the rotated vertex removed. This is then $z_i Z (\mathfrak{S}_{s_i \mathbf{z}, \lambda, w})$.
 Note that $z_i$ and $z_j$ in Figure~\ref{coloredrmatrix}
  become here $z_{i+1}$ and $z_i$, respectively. 
  We are using red and blue for the colors $d_i$ and $d_{i + 1}$, respectively.

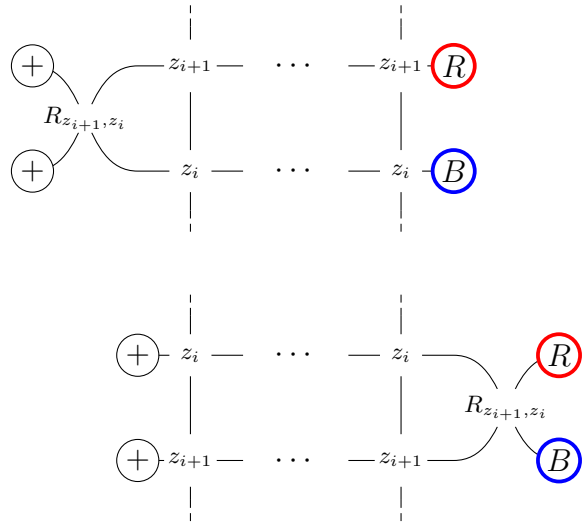
\begin{figure}[h]
\begin{tikzpicture}[scale=0.7]
\begin{scope}[shift={(-1,0)}]
  \draw (0,1) to [out = 0, in = 180] (2,3) to (4,3);
  \draw (0,3) to [out = 0, in = 180] (2,1) to (4,1);
  \draw (3,0.25) to (3,3.75);
  \draw (7,0.25) to (7,3.75);
  \draw (6,1) to (8,1);
  \draw (6,3) to (8,3);
  \draw[fill=white] (0,1) circle (.4);
  \draw[fill=white] (0,3) circle (.4);
  \draw[line width=0.5mm,red,fill=white] (8,3) circle (.4);
  \draw[line width=0.5mm,blue,fill=white] (8,1) circle (.4);
  \node at (0,1) {$+$};
  \node at (0,3) {$+$};
  \node at (5,3) {$\cdots$};
  \node at (5,1) {$\cdots$};
  \draw[densely dashed] (3,3.75) to (3,4.25);
  \draw[densely dashed] (3,0.25) to (3,-0.25);
  \draw[densely dashed] (7,3.75) to (7,4.25);
  \draw[densely dashed] (7,0.25) to (7,-0.25);
  \node at (8,1) {$B$};
  \node at (8,3) {$R$};
\path[fill=white] (3,3) circle (.5);
\node at (3,3) {\scriptsize$z_{i+1}$};
\path[fill=white] (3,1) circle (.4);
\node at (3,1) {\scriptsize$z_i$};
\path[fill=white] (7,3) circle (.5);
\node at (7,3) {\scriptsize$z_{i+1}$};
\path[fill=white] (7,1) circle (.4);
\node at (7,1) {\scriptsize$z_i$};
\path[fill=white] (1,2) circle (.3);
\node at (1,2) {\scriptsize$R_{z_{i+1},z_i}$};
\end{scope}
\begin{scope}[shift={(1,-5.5)}]
  \draw (4,1) to (6,1) to [out = 0, in = 180] (8,3);
  \draw (4,3) to (6,3) to [out = 0, in = 180] (8,1);
  \draw[line width=0.5mm,red,fill=white] (8,3) circle (.4);
  \draw[line width=0.5mm,blue,fill=white] (8,1) circle (.4);
  \draw (0,1) to (2,1);
  \draw (0,3) to (2,3);
  \draw (5,0.25) to (5,3.75);
  \draw (1,0.25) to (1,3.75);
  \draw[fill=white] (0,1) circle (.4);
  \draw[fill=white] (0,3) circle (.4);

  \node at (3,1) {$\cdots$};
  \node at (3,3) {$\cdots$};

  \draw[densely dashed] (1,3.75) to (1,4.25);
  \draw[densely dashed] (1,0.25) to (1,-0.25);
  \draw[densely dashed] (5,3.75) to (5,4.25);
  \draw[densely dashed] (5,0.25) to (5,-0.25);
  \path[fill=white] (1,3) circle (.4);
  \node at (1,3) {\scriptsize$z_{i}$};
  \path[fill=white] (1,1) circle (.5);
  \node at (1,1) {\scriptsize$z_{i+1}$};

  \path[fill=white] (5,3) circle (.4);
  \node (a) at (5,3) {\scriptsize$z_{i}$};
  \path[fill=white] (5,1) circle (.5);
  \node at (5,1) {\scriptsize$z_{i+1}$};

  \path[fill=white] (7,2) circle (.4);
  \node at (7,2) {\scriptsize$R_{z_{i+1},z_i}$};
  \node at (8,1) {$B$};
  \node at (8,3) {$R$};
  \node at (0,1) {$+$};
  \node at (0,3) {$+$};
\end{scope}
\end{tikzpicture}

\caption{Top: the system $\mathfrak{S}_{s_i\mathbf{z},\lambda,w}$ with
  the R-matrix attached. Bottom: after using the Yang-Baxter equation.}
\label{modpf}
\end{figure}

After repeated use of the Yang-Baxter equation (Figure~\ref{eqn:ybe}), we move
the rotated vertex to the right, switch the parameters of the two rows and
obtain a system with the same partition function by
Theorem~\ref{coloredybethm}. This is the system on the bottom of
Figure~\ref{modpf}. This method of Baxter is sometimes called the
``train argument.''
  
  Now looking at the possible weights from Figure~\ref{coloredrmatrix}, the R-matrix
  has two admisible configurations (third and fifth on the second row) and so the equality of partition functions from Figure~\ref{modpf} becomes the identity
  \[ z_i Z (\mathfrak{S}_{s_i \mathbf{z}, \lambda, w}) =
  z_{i + 1} Z (\mathfrak{S}_{\mathbf{z}, \lambda, w})
  + (z_{i+1} - z_{i})\,Z (\mathfrak{S}_{\mathbf{z}, \lambda, s_i w}) . \]
  Since $\mathbf{z}^{\alpha_i}=z_i/z_{i+1}$, the above identity may be rewritten as
  \begin{equation} \label{asdbaraction} 
  Z (\mathfrak{S}_{\mathbf{z}, \lambda, s_iw}) = - (1
     -\mathbf{z}^{\alpha_i})^{- 1} (Z (\mathfrak{S}_{\mathbf{z}, \lambda, w})
     -\mathbf{z}^{\alpha_i} Z (\mathfrak{S}_{s_i\mathbf{z}, \lambda, w})) . \end{equation}
  The right-hand side can be interpreted as the operator
  $-(1 -\mathbf{z}^{\alpha_i})^{-1} (1 - \mathbf{z}^{\alpha_i}s_i)$ applied to
  $Z (\mathfrak{S}_{\mathbf{z}, \lambda, w})$. Note that
  \[\partial^\circ_i=-(1 -\mathbf{z}^{\alpha_i})^{-1} (1 - s_i),\quad \textrm{and hence} \quad
  \mathbf{z}^\rho\partial^\circ_i\mathbf{z}^{-\rho}=
  -(1 -\mathbf{z}^{\alpha_i})^{-1} (1 - \mathbf{z}^{\alpha_i}s_i).\]
  Using this, (\ref{lemrecurse}) follows from (\ref{asdbaraction}).
\end{proof}

\begin{remark} It was recently found by Brubaker, Bump and Friedberg that
a variation of the Boltzmann weights produces the
Demazure character $\mathbf{z}^\rho\partial_{w}\mathbf{z}^\lambda$ instead
of the Demazure atom $\mathbf{z}^\rho\partial^\circ_{w}\mathbf{z}^\lambda$ in
Theorem~\ref{zdematoms}. The modification is to interchange red and blue
in the third case of Figure~\ref{coloredweights}.
We hope to discuss this in a subsequent paper.
\end{remark}

\section{\label{dcandci}Demazure crystals and atoms}

A refined Demazure character formula in the context of crystals was
obtained by Littelmann~\cite{LittelmannYT} and
Kashiwara~\cite{KashiwaraDemazure}. We begin this section by reviewing this refinement and then
proceed to identify Demazure atoms with subsets of crystal and characterize the vertices belonging to
this crystal. 

Let us fix a finite Cartan type
with weight lattice $\Lambda$; when we return to the colored ice we
will take this to be the $\GL(r)$ Cartan type. Let $\lambda$ be
a dominant weight, which we assume to be a partition.
Then there is a unique irreducible representation
$\pi_\lambda$ of highest weight $\lambda$, and a corresponding
normal crystal $\mathcal{B}_\lambda$ whose character is the
same as that of $\pi_\lambda$.

Recall that crystals come equipped with Kashiwara maps
$e_i,f_i:\mathcal{B}_\lambda\to\mathcal{B}_\lambda\cup\{0\}$ and 
$\varphi_i,\varepsilon_i:\mathcal{B}_\lambda\to\mathbb{Z}$ (see \cite{KashiwaraNakashima}).
For a crystal $\mathcal{B}$ an element
$v$ is called a \textit{highest weight element} if $e_i(v)=0$ for
all $i$; similarly it is \textit{lowest weight} if all $f_i(v)=0$.
The crystal $\mathcal{B}_\lambda$ has unique highest and lowest
weight elements $v_\lambda$ and $v_{w_0\lambda}$, respectively;
with weights $\wt(v_\lambda)=\lambda$ and $\wt(v_{w_0\lambda})=w_0\lambda$.

With $\mathcal{B}=\mathcal{B}_\lambda$ let
$\mathbb{Z} [\mathcal{B}]$ be the free abelian group on $\mathcal{B}$. We
define a map $\partial_i : \mathcal{B} \longrightarrow \mathbb{Z}
[\mathcal{B}]$ in terms of the Kashiwara operators $e_i$ and $f_i$ by
\[ \partial_i v = \left\{ \begin{array}{ll}
     v + f_i v + \ldots + f_i^k v & \text{if $k \geqslant 0,$}\\
     0 & \text{if $k = - 1,$}\\
     - (e_i v + \ldots + e_i^{- k - 1} v) & \text{if $k < -1$},
   \end{array} \right. \]
where $k = \langle \wt (v), \alpha_i^{\vee} \rangle$. This lifts the
Demazure operator $\partial_i$ to the crystal; indeed, composing with the familiar weight
map on the crystal (described in Section~\ref{sec:kansas}) produces the Demazure operators of (\ref{isobaricddefined}), and
so we will use the same notation for the operator in both contexts.

By an {\textit{$i$-root string}}
we mean an equivalence class of elements of $\mathcal{B}$ under the
equivalence relation that $x \equiv y$ if $x = e_i^r y$ or $x = f_i^r y$ for
some $r$. An $i$-root string $S$ has a unique highest weight element $u_S$
characterized by $e_i (u_S) = 0$. We may now state the \textit{refined
Demazure character formula} of Littelmann and Kashiwara.

\begin{theorem}[Littelmann, Kashiwara]
  \label{demazurecrystals}
  Let $\mathcal{B}=\mathcal{B}_\lambda$.
  \begin{enumerate}[label={(\roman*)}, leftmargin=*]
    \item \label{itm:crystals} There exist subsets $\mathcal{B}(w)$ of $\mathcal{B}$
  indexed by $w \in W$ such that $\mathcal{B}(1) = \{ v_{\lambda} \}$,
  $\mathcal{B}(w_0)=\mathcal{B}$ and if $s_i w > w$ then
  \[\mathcal{B}(s_i w) =
 \left\{ x \in \mathcal{B} \mid \text{$e_i^r x \in \mathcal{B}(w)$ for some
   $r$} \right\}\;.\]
 \item If $S$ is an $i$-root string then $\mathcal{B}(w) \cap S$ is one of the three
  possibilities: $\varnothing$, $S$ or $\{ u_S \}$. 
  \par\smallskip\noindent
\item \label{itm:character} We have
  \[ \sum_{x \in \mathcal{B}(w)} \mathbf{z}^{\wt(x)} = \partial_w\mathbf{z}^\lambda\;.\]
  \end{enumerate}
\end{theorem}

\noindent
See~\cite{KashiwaraDemazure} or \cite{BumpSchilling} Chapter~13 for proof.

\medskip

Demazure characters and atoms were defined in Section~\ref{sec:do} as functions
on the complex torus $T$. The preceding theorem allows us to lift Demazure characters to the crystal $\mathcal{B}=\mathcal{B}_\lambda$; as in
the theorem, we will denote these (lifted) Demazure characters by $\mathcal{B}(w)$ for $w \in W$.
Let $\mathcal{B}^\circ(w)$ ($w\in W$) be a family of disjoint subsets of
$\mathcal{B}$.  We call these a family of \textit{crystal Demazure atoms}
if
\begin{equation}
 \label{eqatomic}\mathcal{B}(w)=\bigcup_{y\leqslant w}\mathcal{B}^\circ(y).
\end{equation}

\begin{lemma}
\label{lem:uniqueatomic}
If a family of disjoint subsets
$\mathcal{B}^\circ(w)$ satisfying \eqref{eqatomic} exists it is unique.
\end{lemma}

\begin{proof}
Let us identify a subset $S$ of $\mathcal{B}$ with the element
$\sum_{v\in S} v$ of the free abelian group $\mathbb{Z}[\mathcal{B}]$.
Then we may rewrite (\ref{eqatomic}) as
\[\mathcal{B}(w)=\sum_{y\leqslant w}\mathcal{B}^\circ(y).\]
By M\"{o}bius inversion with respect to the Bruhat
order (\cite{Verma,Stembridge}) this is equivalent to
\[\mathcal{B}^\circ(w)=\sum_{y\leqslant w}(-1)^{\ell(w)-\ell(y)}\mathcal{B}(y).\]
This characterization of $\mathcal{B}^\circ(w)$ as an element
of $\mathbb{Z}[\mathcal{B}]$ proves the uniqueness.
\end{proof}

As explained in the Introduction, in type~A such a decomposition of the set of
tableaux in any $\mathcal{B}_\lambda$ is given by the theory of
Lascoux-Sch\"utzenberger keys. We will give another algorithm to compute, for
any $v \in \mathcal{B}$, the element~$w \in W$ such that $v \in
\mathcal{B}^\circ(w)$ and show that the resulting subsets satisfy
(\ref{eqatomic}), making them a family of crystal Demazure atoms.
This algorithm makes use of the \textit{string} or \textit{BZL} patterns for
vertices in a crystal, which we now describe. These patterns were introduced
in~\cite{BerensteinZelevinskyDuke} for type~$A$, and more generally
in~\cite{LittelmannCones}. See also~\cite{BumpSchilling} Chapter~11
and~\cite{wmd5book} Chapters~2 and~5.

Let $\mathbf{i}=(i_1,\cdots,i_N)$ be a reduced word
for $w_0=s_{i_1}\cdots s_{i_N}$.
Given any $v\in\mathcal{B}_\lambda$, let $b_1 := b_1(v)$ be the largest
nonnegative integer such that $f_{i_1}^{b_1}v\neq 0$.
Then let $b_2$ be the largest integer such that
$f_{i_2}^{b_2}f_{i_1}^{b_1}v\neq 0$. Continuing, we
find that $f_{i_N}^{b_N}\cdots f_{i_2}^{b_2}f_{i_1}^{b_1}v=v_{w_0\lambda}$.
We will denote the resulting vector of lengths in root strings by
\begin{equation}
  \label{plstring}
  \bzl^{(f)}_{\mathbf{i}}(v):=(b_1,\cdots,b_N).
\end{equation}
Dually, let $c_1,\cdots,c_N$ be the
maximum values such that $e_{i_k}^{c_k}\cdots e_{i_2}^{c_2}e_{i_1}^{c_1}v\neq 0$
for $k=1,2,\cdots,N$. Then 
$e_{i_N}^{c_N}\cdots e_{i_2}^{c_2}e_{i_1}^{c_1}v=v_\lambda$ and we define
\begin{equation}
  \label{plstringe}
  \bzl^{(e)}_{\mathbf{i}}(v):=(c_1,\cdots,c_N).
\end{equation}

The map $\alpha\mapsto -w_0\alpha$ permutes the positive roots,
and in particular the simple roots. Thus there is a
bijection $i\mapsto i'$ of the set $I$ of indices such that
$\alpha_{i'}=-w_0\alpha_i$ and $w_0s_iw_0^{-1}=s_{i'}$.
In the $\GL(r)$ case $I=\{1,\cdots,r-1\}$
and $i'=r-i$. The crystal also has a map $v\mapsto v'$, the
\textit{Sch\"utzenberger} or \textit{Lusztig} involution, such that if
$v\in\mathcal{B}$ then
\begin{equation}
  \label{schutzef}
  f_i(v')=(e_{i'}(v))',\qquad e_i(v')=(f_{i'}(v))' .
\end{equation}
It follows from (\ref{schutzef}) that
if $\mathbf{i}'=(i_1',\cdots,i_N')$ then
\begin{equation}
  \label{schutzbzl}
  \bzl^{(e)}_{\mathbf{i}'}(v)=\bzl^{(f)}_{\mathbf{i}}(v').
\end{equation}

Littelmann~\cite{LittelmannCones} observed that for certain ``good'' choices of long word
$\mathbf{i}$ the set of possible string patterns can be easily characterized.
For $\GL(r)$, we take
\begin{equation}
  \label{littstr}
  \mathbf{i}=(1,2,1,3,2,1,4,3,2,1,\cdots,r,r-1,\cdots,3,2,1).
\end{equation}
Thus
\begin{equation}
  \label{litstrp}
  \mathbf{i}'=(r-1,r-2,r-1,r-3,r-2,r-1,\cdots,r-3,r-2,r-1) .
\end{equation}

Following~\cite{LittelmannCones} we arrange the
string pattern $\bzl^{(e)}_{\mathbf{i}'}(v)=(b_1,b_2,\cdots)$
in an array
\begin{equation}
  \label{bzlpat}
  \bzl^{(e)}_{\mathbf{i}'}(v)=
\left[\begin{array}{cccc}
\ddots&&\vdots&\vdots\\
&b_4&b_5&b_6\\
&&b_2&b_3\\
&&&b_1\end{array}\right]
\end{equation}
in which the $b_i$ satisfy the Littelmann cone
inequalities 
\begin{equation}
 \label{litcone}
  b_1\geqslant 0,\qquad b_2\geqslant b_3\geqslant 0,\qquad b_4\geqslant
  b_5\geqslant b_6\geqslant 0\,,\qquad\cdots\;.
\end{equation}

Following~\cite{wmd5book} we decorate the string pattern (\ref{bzlpat})
by circling certain $b_i$ according to these cone inequalities.

\begin{circling}
  \label{circlingrule}
  Let $\mathbf{b}=(b_1,b_2,\cdots,b_N)$ where $N=r(r-1)/2$ be a
  sequence of nonnegative integers satisfying \eqref{litcone}.
  We arrange the sequence in an array \eqref{bzlpat} and
  decorate it by circling an entry $b_i$ if it is minimal in the cone. 
  Explicitly, if $i$ is a triangular number, so that
  $b_i$ is at the right end of its row, the condition for circling it
  is that $b_i=0$; otherwise, the condition for circling is
  that $b_i=b_{i+1}$. 
\end{circling}

Let $(i_1,i_2,i_3,i_4,i_5,i_6,\cdots)$ be the sequence $(1,2,1,3,2,1,\cdots)$ of \eqref{littstr}.
We transfer the circles from the string pattern to
the following array made with the simple reflections:
\begin{equation}
\label{circledreflections}
\left[\begin{array}{cccc}
\ddots&&\vdots&\vdots\\
&s_{i_6}&s_{i_5}&s_{i_4}\\
&&s_{i_3}&s_{i_2}\\
&&&s_{i_1}\end{array}\right] =
\left[\begin{array}{cccc}
\ddots&&\vdots&\vdots\\
&s_1&s_2&s_3\\
&&s_1&s_2\\
&&&s_1\end{array}\right]\;.
\end{equation}

\begin{remark}
  Note that the horizontal orders of the entries in (\ref{bzlpat})
  and (\ref{circledreflections}) are different.
\end{remark}

If $v\in\mathcal{B}_\lambda$, let $(s_{j_1},\cdots,s_{j_k})$ be the subsequence of
$(s_{i_1},s_{i_2},s_{i_3},\cdots)=(s_1,s_2,s_1,s_3,s_2,s_1,\cdots)$
consisting of the circled reflections in (\ref{circledreflections}) derived
from the string pattern $\bzl^{(e)}_{\mathbf{i}'}(v)$. Here $\mathbf{i}'$
is the specific sequence in (\ref{litstrp}). With the nondescending product $\Pind$ defined in (\ref{pinddef}), define $\omega : \mathcal{B}_\lambda \to W$ by
\begin{equation}
  \label{ompdef}\omega(v) := \Pind(s_{j_1},\cdots,s_{j_k}).
\end{equation}

For example, suppose that the string pattern is:
\begin{equation}
  \label{samplebzl}
\left[
\vcenter{
\xymatrix@-1.5pc{
*+[o][F-]{1}&1\\
&*+[o][F-]{0}}}\right]\;.
\end{equation}
The circling rule tells us to 
circle $b_1$ and $b_2$ since $b_1=0$ and $b_2=b_3$.
Thus we circle these entries:
\[\left[
\vcenter{
\xymatrix@-1.5pc{
*+[o][F-]{s_1}&s_2\\
&*+[o][F-]{s_1}}}\right]\;
\]
and $\omega(v)=\Pind(s_1,s_1)= \left\{ S_1^2 \right\} = s_1$ in this case, using the notation of Remark~\ref{heckecompute}.

We may now state one of our main results. Let $W_\lambda$ be the
stabilizer of $\lambda$ in $W$. Note that if $w,w'\in W$ lie in
the same coset of $W/W_\lambda$ then $\mathcal{B}_\lambda(w)=\mathcal{B}_\lambda(w')$.
We will say that $w\in W$ is $\lambda$-maximal if it is
the longest element of its subset.

\begin{theorem}
  \label{daform}
  Let $\mathcal{B}=\mathcal{B}_\lambda$.
  There exist a family of subsets $\mathcal{B}^\circ(w)$
  of $\mathcal{B}$ indexed by $w\in W$ such that
  $\mathcal{B}^\circ(w)=\mathcal{B}^\circ(w')$
  if and only if $w,w'$ lie in the same coset of
  $W/W_\lambda$; otherwise they are disjoint, and such
  that the decomposition \eqref{eqatomic} is satisfied. If $w$ is the
  longest element of this coset, then
  \begin{equation}
    \label{bcircdef}
    \mathcal{B}^\circ(w)=\{v\in\mathcal{B} \mid w_0\omega(v)=w\}.
  \end{equation}
  If $w$ is not the longest element of its coset then
  the equation $w_0\omega(v)=w$ has no solutions.
\end{theorem}

This is a refinement of results of Lascoux and
Sch\"utzenberger~\cite{LascouxSchutzenbergerKeys},
and is one of the main points of the paper.
Equation (\ref{bcircdef}), together with the definition and properties
of $\omega$, leads to the the algorithmic characterization of the crystal
Demazure atom in Subsection~\ref{subsec:alg}. The proof of Theorem~\ref{daform}
will be given later, in Section~\ref{sec:dafpro}.

\section{\label{bijsec}A bijection between colored states and Demazure atoms}

We return now to colored ice models. Recall from Proposition~\ref{statedpf} that the
admissible states of colored ice $\mathfrak{S}_{\mathbf{z}, \lambda, w}$ with
$w \in W$ partition the set of admissible states of uncolored ice in the
system $\mathfrak{S}_{\mathbf{z}, \lambda}$. The map from any
$\mathfrak{S}_{\mathbf{z}, \lambda, w}$ to
$\mathfrak{S}_{\mathbf{z}, \lambda}$ is simply given by ignoring the colors
(i.e., replacing each colored edge by a $-$ spin).

In Section~\ref{sec:kansas} we defined a map
$\mathfrak{s}\to\mathfrak{T}(\mathfrak{s})$ from
$\mathfrak{S}_{\mathbf{z},\lambda}$ to $\mathcal{B}_\lambda$.  We are
interested in knowing the image of $\mathfrak{S}_{\mathbf{z}, \lambda, w}$
under this map. Let $v\to v'$ be the Sch\"utzenberger (Lusztig)
involution of $\mathcal{B}_\lambda$.

\begin{theorem}
  \label{crystalmt}
  If $w \in W$ and $\mathfrak{s}\in\mathfrak{S}_{\mathbf{z},\lambda}$, then
  $\mathfrak{s}\in\mathfrak{S}_{\mathbf{z},\lambda,w}$ if and only if
  $w_0\omega\big(\mathfrak{T}(\mathfrak{s})'\big)=w$.
\end{theorem}

Thus if we accept Theorem~\ref{daform}, whose proof will be given later,
comparing Theorem~\ref{crystalmt} with (\ref{bcircdef}) shows
that the map $\mathfrak{s}\rightarrow\mathfrak{T}(\mathfrak{s})'$
sends the ensemble $\mathfrak{S}_{\mathbf{z},\lambda,w}$ to the Demazure
atom $\mathcal{B}_\lambda^\circ(w)$. Ultimately the proof of
Theorem~\ref{daform} in Section~\ref{sec:dafpro}
will rely on this Theorem~\ref{crystalmt}.

Before we prove Theorem~\ref{crystalmt} we give an example.
In Figure~\ref{adjointcrystal},
we have labeled the elements of the $\GL(3)$ crystal $\mathcal{B}_\lambda$
($\lambda=(2,1,0)$) by a flag indicating the colors along the right edge of the corresponding
state. These colors are read off from top to bottom on the horizontal
edges at the right boundary of the grid.
In the decomposition of Proposition~\ref{statedpf}, the flag
is a permutation $w\mathbf{c}$ of the colors of the standard
flag, which we are taking to be $\mathbf{c}=(R,B,G)$. For
example, to compute the flags for the elements
\begin{equation}
  \label{tostatesbis}
\ytableausetup{nosmalltableaux}\begin{ytableau}1&3\\2\end{ytableau} \quad \text{and} \quad
  \begin{ytableau}1&2\\2\end{ytableau}
\end{equation}
we construct the corresponding states as in Figure~\ref{tostates} and then read
off the colors from the right edge, which are $(G,R,B)$ for both states.
In Figure~\ref{adjointcrystal} these colors are
represented as a flag. The flag allows us to read off the unique $y\in W$
such that the corresponding state $\mathfrak{s}$ is in
$\mathfrak{S}_{\mathbf{z},\lambda,y}$. For example in
the two states in (\ref{tostatesbis}), we have the flag $(G,R,B)=s_1s_2(R,B,G)$ and
so $y=s_1s_2$.

\noindent
\begin{figure}[htb]
\[\ytableausetup{smalltableaux}
\begin{tikzpicture}[flagmatrix/.style={xshift=.9cm, matrix of nodes, nodes={outer sep=0pt, inner sep=0pt, minimum size=3.5mm, font=\scriptsize\bfseries, text=white}, row sep={3.5mm,between origins}}]
\node at (0,0) {$\begin{ytableau}2&3\\3\end{ytableau}$};
\node at (-2,2) {$\begin{ytableau}1&3\\3\end{ytableau}$};
\node at (2,2) {$\begin{ytableau}2&2\\3\end{ytableau}$};
\node at (-2,4) {$\begin{ytableau}1&3\\2\end{ytableau}$};
\node at (2,4) {$\begin{ytableau}1&2\\3\end{ytableau}$};
\node at (-2,6) {$\begin{ytableau}1&2\\2\end{ytableau}$};
\node at (2,6) {$\begin{ytableau}1&1\\3\end{ytableau}$};
\node at (0,8) {$\begin{ytableau}1&1\\2\end{ytableau}$};
\draw[->] (-1.6,1.6) -- (-.5,.5);
\draw[->] (1.5,1.5) -- (.5,.5);
\draw[->] (-2,3.4) -- (-2,2.6);
\draw[->] (2,3.4) -- (2,2.6);
\draw[->] (-2,5.4) -- (-2,4.6);
\draw[->] (2,5.4) -- (2,4.6);
\draw[->] (0.4,7.6) -- (1.5,6.5);
\draw[->] (-.5,7.5) -- (-1.5,6.5);

\matrix at (0,0) [flagmatrix] {|[fill=red]| R \\ |[fill=blue]| B \\ |[fill=green]| G \\};
\matrix at (-2,2) [flagmatrix] {|[fill=red]| R \\ |[fill=green]| G \\ |[fill=blue]| B \\};
\matrix at (-2,4) [flagmatrix] {|[fill=green]| G \\ |[fill=red]| R \\ |[fill=blue]| B \\};
\matrix at (-2,6) [flagmatrix] {|[fill=green]| G \\ |[fill=red]| R \\ |[fill=blue]| B \\};
\matrix at (2,2) [flagmatrix] {|[fill=blue]| B \\ |[fill=red]| R \\ |[fill=green]| G \\};
\matrix at (2,4) [flagmatrix] {|[fill=blue]| B \\ |[fill=green]| G \\ |[fill=red]| R \\};
\matrix at (2,6) [flagmatrix] {|[fill=blue]| B \\ |[fill=green]| G \\ |[fill=red]| R \\};
\matrix at (0,8) [flagmatrix] {|[fill=green]| G \\ |[fill=blue]| B \\ |[fill=red]| R \\};

\node at (-1.2,7.25) {$\scriptstyle 1$};
\node at (1.2,7.25) {$\scriptstyle 2$};
\node at (-2.2,5) {$\scriptstyle 2$};
\node at (2.2,5) {$\scriptstyle 1$};
\node at (-2.2,3) {$\scriptstyle 2$};
\node at (2.2,3) {$\scriptstyle 1$};
\node at (1.2,1) {$\scriptstyle 2$};
\node at (-1.2,1) {$\scriptstyle 1$};
\end{tikzpicture}\qquad\qquad
\begin{tikzpicture}
\node at (0,0) {$\begin{smallmatrix}0&0\\&0\end{smallmatrix}$};
\node at (-2,2) {$\begin{smallmatrix}0&0\\&1\end{smallmatrix}$};
\node at (-2,4) {$\begin{smallmatrix}1&1\\&0\end{smallmatrix}$};
\node at (-2,6) {$\begin{smallmatrix}2&1\\&0\end{smallmatrix}$};

\node at (2,2) {$\begin{smallmatrix}1&0\\&0\end{smallmatrix}$};
\node at (2,4) {$\begin{smallmatrix}1&0\\&1\end{smallmatrix}$};
\node at (2,6) {$\begin{smallmatrix}1&0\\&2\end{smallmatrix}$};
\node at (0,8) {$\begin{smallmatrix}1&2\\&1\end{smallmatrix}$};

\draw (0,0) circle (.4);
\draw (-2,2) circle (.4);
\draw (2,2) circle (.4);
\draw (-2,4) circle (.4);
\draw (2,4) circle (.4);
\draw (-2,6) circle (.4);
\draw (2,6) circle (.4);
\draw (0,8) circle (.4);

\draw[->] (-1.6,1.6) -- (-.5,.5);
\draw[->] (1.5,1.5) -- (.5,.5);
\draw[->] (-2,3.4) -- (-2,2.6);
\draw[->] (2,3.4) -- (2,2.6);
\draw[->] (-2,5.4) -- (-2,4.6);
\draw[->] (2,5.4) -- (2,4.6);
\draw[->] (0.4,7.6) -- (1.5,6.5);
\draw[->] (-.5,7.5) -- (-1.5,6.5);

\node at (-1.2,7.25) {$\scriptstyle 1$};
\node at (1.2,7.25) {$\scriptstyle 2$};
\node at (-2.2,5) {$\scriptstyle 2$};
\node at (2.2,5) {$\scriptstyle 1$};
\node at (-2.2,3) {$\scriptstyle 2$};
\node at (2.2,3) {$\scriptstyle 1$};
\node at (1.2,1) {$\scriptstyle 2$};
\node at (-1.2,1) {$\scriptstyle 1$};
\end{tikzpicture}\]
\caption{Left: The $\GL(3)$ crystal of highest weight $\lambda={(2,1,0)}$,
showing the ``flags'' that are the colors of the right edges of the
corresponding states. Right: the same crystal, showing the pattern
$\bzl^{(f)}_{\mathbf{i}}$ that controls both the crossings of colored lines in the
state, and which also carry information about the Demazure crystals.}
\label{adjointcrystal}
\end{figure}
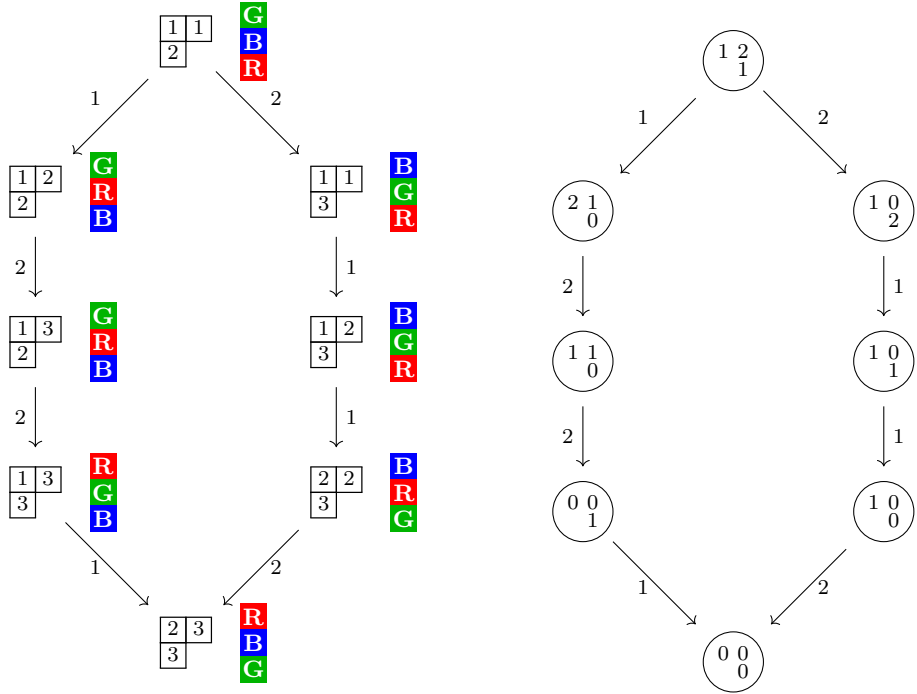

Now let us also verify Theorem~\ref{daform} and Theorem~\ref{crystalmt}
for the patterns in Figure~\ref{tostates}.
Both are in the system $\mathfrak{S}_{\mathbf{z},(2,1,0),s_1s_2}$.
Their string patterns
$\bzl^{(e)}_{\mathbf{i}'}(\mathfrak{T}')=\bzl^{(f)}_{\mathbf{i}}(\mathfrak{T})$
are shown in Table~\ref{tab:string-patterns}.

\begin{table}[h]
\ytableausetup{nosmalltableaux}
  \centering
  \caption{String patterns for the examples shown in Figure~\ref{tostates} with tableau $\mathfrak{T}$ and its Sch\"utzenberger involution $\mathfrak{T}'$.}
  \label{tab:string-patterns}
  \vspace{-1em}
\[\def\arraystretch{2.2}
\begin{array}{c@{\hskip 2em}c@{\hskip 2em}l}
  \toprule
  \mathfrak{T} & \mathfrak{T}' & \bzl^{(f)}_{(1,2,1)}(\mathfrak{T})=\bzl^{(e)}_{(2,1,2)}(\mathfrak{T}') \\
  \midrule 
  \ytableaushort{12,2} & \ytableaushort{22,3} & \left[\vcenter{\xymatrix@-1.5pc{2&1\\&*+[o][F-]{0}}}\right] \vspace{1em}\\
  \ytableaushort{13,2} & \ytableaushort{12,3} & \left[\vcenter{\xymatrix@-1.5pc{*+[o][F-]{1}&1\\&*+[o][F-]{0}}}\right] \vspace{1em}\\
  \bottomrule
\end{array}
\]
\end{table}

We have $\omega(\mathfrak{T}')=s_1$ in both cases; indeed for the first row in Table~\ref{tab:string-patterns}, $\omega(\mathfrak{T}')=\Pind(s_1)=s_1$
and in the second row $\omega(\mathfrak{T}')=\Pind(s_1,s_1)=s_1$, and in both cases
$w_0\omega(\mathfrak{T'})=s_1s_2$. Moreover the two patterns $\mathfrak{T}'$
comprise the Demazure atom $\mathcal{B}^\circ(s_1s_2)$ since they are
the two patterns in $\mathcal{B}(s_1s_2)$ that are not already in
$\mathcal{B}(s_2)$.
Thus we have confirmed both Theorem~\ref{daform} and Theorem~\ref{crystalmt} for one
particular Demazure atom.

\begin{proof}[Proof of Theorem~\ref{crystalmt}]
  First we will show that the circled locations in $\GTP^\circ(\mathfrak{s})$
  correspond to $\texttt{a}_2$ vertices in the state $\mathfrak{s}$
  (by the labeling in Figure~\ref{coloredweights}), which are places where the
  colored lines may cross.

  Let $\mathfrak{s}$ be a state of $\mathfrak{S}_{\mathbf{z},\lambda,y}$.
  Let $\GTP^\circ(\mathfrak{s})$ and $\mathfrak{T}\in\mathcal{B}_\lambda$
  be the corresponding Gelfand-Tsetlin pattern and tableau as described in Section~\ref{sec:kansas} (using the embedding of $\mathfrak{S}_{\mathbf{z},\lambda,y}$ into $\mathfrak{S}_{\mathbf{z},\lambda}$). We take
  $v=\mathfrak{T}'$ in (\ref{bzlpat}) so we are using
  $\bzl_{\mathbf{i}'}^{(e)}(\mathfrak{T}')=\bzl_{\mathbf{i}}^{(f)}(\mathfrak{T})$ represented as a vector $(b_1, b_2, \cdots)$.
  Let us consider how the circles may be read off from the
  Gelfand-Tsetlin pattern with entries $a_{i,j}$ as in (\ref{reducedgtp}). According to Proposition~2.2 of~\cite{wmd5book}, 
  \begin{equation}
\label{bzlfromgtp}
\begin{array}{l}
  b_1 = a_{r,r}-a_{r-1,r}\\
  b_2 = (a_{r-1,r-1}+a_{r-1,r})-(a_{r-2,r-1}+a_{r-2,r}),\\
  b_3 = a_{r-1,r}-a_{r-2,r},\\
  b_4 = (a_{r-2,r-2}+a_{r-2,r-1}+a_{r-2,r})-(a_{r-3,r-2}+a_{r-3,r-1}+a_{r-3,r}),\\
  b_5 = (a_{r-2,r-1}+a_{r-2,r})-(a_{r-3,r-1}+a_{r-3,r}),\\
  b_6 = a_{r-2,r}-a_{r-3,r},\\
  \quad\vdots\end{array}
\end{equation}
  These imply that the circled locations depend on equalities between entries in
  $\GTP(\mathfrak{s})$ or, equivalently, $\GTP^\circ(\mathfrak{s})$. 
For example $b_2$ is circled if and only if
  $a_{r-1,r-1}=a_{r-2,r-1}$. With $A_{i,j}$ the entries in
  $\GTP(\mathfrak{s})$, so that $A_{i,j}=a_{i,j}+r-j$, this
  is equivalent to $A_{r-1,r-1}=A_{r-2,r-1}$, and similarly
  if any $b_k$ is circled then we have $A_{i,j}=A_{i-1,j}$
  for the appropriate $i,j$. Now recall that in the
  bijection $\mathfrak{T} \leftrightarrow \mathfrak{s}$, $A_{i,j}$ is the
  number of a column where a vertical edge has a colored spin.
  Therefore from the admissible colored ice configurations of
  Figure~\ref{coloredweights}, the circled entries in (\ref{bzlpat})
  correspond to vertices of type $\tt{a}_2$ in the state of ice
  $\mathfrak{s}$. These are locations where two colored lines may cross.  From
  Figure~\ref{coloredweights} the lines will cross if and only if the left
  edge color is greater than the top edge color at the vertex, which is
  equivalent to the assumption that they have not crossed previously.
  
  We consider a sequence of lines $\ell_i$ through the grid, $i=0,\ldots,r$
  to be described as follows. The line $\ell_i$ begins to the left of the
  grid between the $i$-th and $(i+1)$-th row, or above the first row if $i=0$,
  or below the $r$-th row if $i=r$. It traverses the grid, then moves up to
  the northeast corner. See Figure~\ref{tostates} where these lines are drawn
  in two examples.

  Each $\ell_i$ intersects exactly $r$ colored lines, and we can read off
  the colors sequentially; let $\mathbf{c}^i$ be the corresponding sequence
  of colors. Thus $\mathbf{c}^0=\mathbf{c}$, while $\mathbf{c}^r=w_0y\mathbf{c}$,
  where $y$ is the Weyl group element we wish to compute. The $w_0$ in this
  last identity is included because the line $\ell_r$ visits the horizontal
  edges on the right edge from bottom to top, whereas in describing the flag
  $y\mathbf{c}$, the reading is from top to bottom. (See Figure~\ref{tostates}.)

  As we have already noted, the circled entries in (\ref{bzlpat})
  correspond to $\tt{a}_2$ patterns in the state. These are places
  where two colored lines may cross. The crossings interchange
  colors and each corresponds to a simple reflection that is circled in
  (\ref{circledreflections}). So if $i>0$            
  we may try to compute $\mathbf{c}_{i}$ from
  $\mathbf{c}_{i-1}$ by applying the circled reflections in the $i$-th row of
  (\ref{circledreflections}). Remembering from the proof of
  Proposition~\ref{statedpf} that the colors in the $i$-th row
  are assigned from right to left, this means (subject to a
  caveat that we will explain below) that
  \[\mathbf{c}_i=(s_{r-i})_i\cdots(s_3)_i(s_2)_i(s_1)_i\mathbf{c}_{i-1},\]
  where $(s_j)_i$ denotes $s_j$ if $s_j$ is circled in the $i$-th row
  of (\ref{circledreflections}), and $(s_j)_i=1$ if $s_j$ is not
  circled.

  If $i=r$, there is no $i$-th row to (\ref{circledreflections}), and
  correspondingly $\mathbf{c}_r=\mathbf{c}_{r-1}$. This is as it should be
  since at this stage there is only a single colored vertical edge that
  intersects the line $\ell_{r-1}$, and no interchanges are possible. (See
  Figure~\ref{tostates}.)
    
  We mentioned that there is a caveat in the above explanation.
  This is because from Figure~\ref{coloredweights} we see that in
  an $\tt{a}_2$ vertex, if the color $c$ is left of the vertex and $d$ is
  above, the colored lines will cross if $c>d$ but not otherwise. In particular,
  two colored lines can only cross once. More precisely, if two colored
  lines meet more than once (at $\tt{a}_2$ vertices) they will cross the
  first time they meet, and never again. For this reason, the permutation
  that turns $\mathbf{c}^0=\mathbf{c}$ into $\mathbf{c}^r=w_0y\mathbf{c}$ is
  the nondescending product $\Pi_{\operatorname{nd}}(s_{i_1},\cdots,s_{i_k})$ where
  $(s_{i_1},\cdots,s_{i_k})$ is the subsequence of circled simple reflections
  in \eqref{circledreflections}. Note that according to the definition of $\Pi_{\operatorname{nd}}(s_{i_1},\cdots,s_{i_k})$ in equation~\eqref{pinddef}, the circled simple reflections corresponding to the $\tt{a}_2$ patterns where there is a crossing play a role in recursively defining $\Pi_{\operatorname{nd}}(s_{i_1},\cdots,s_{i_k})$, while the circled simple reflections corresponding to the $\tt{a}_2$ patterns where there is no crossing do not affect the product. 
  Therefore
  $y=w_0\Pi_{\operatorname{nd}}(s_{i_1},\cdots,s_{i_k})=w_0\omega(\mathfrak{T}')$.
  
  This shows that $\mathfrak{s} \in \mathfrak{S}_{\mathbf{z}, \lambda, y}$
  implies $y = w_0 \omega(\mathfrak{T}(\mathfrak{s})')$. By
  Proposition~\ref{statedpf}, if
  $\mathfrak{s} \notin \mathfrak{S}_{\mathbf{z},\lambda, y}$ then
  $\mathfrak{s} \in \mathfrak{S}_{\mathbf{z},\lambda, y'}$ with
  $y \neq y' \in W$, 
  which we have shown implies that $y \neq y' = w_0 \omega(\mathfrak{T}')$.
\end{proof}

\begin{proposition}
  \label{adjacentcols}Suppose that a part of $\lambda$ is repeated, so that
  $\lambda_i = \lambda_{i + 1} = \ldots = \lambda_j = c$. Then each pair of
  colored lines through the top boundary edges in columns $c + r - i, \cdots,
  c + r - j$ must cross. Thus if $\mathfrak{S}_{\mathbf{z}, \lambda, w}$ is
  nonempty, then $w$ is the shortest Weyl group element in its coset in
  $W / W_{\lambda}$.
\end{proposition}

\begin{proof}
  We are only considering states in which there are no $\tt{b}_1$ patterns
  since these have weight $0$ in Figure~\ref{coloredweights}.  We leave it to
  the reader to convince themselves that because of this, colored lines
  that start in adjacent columns, or more generally in columns not separated by a
  $+$ spin on the top boundary edge must cross. Because we read the colors on
  the top boundary vertical edges from left to right and on the right
  horizontal boundary edges from top to bottom, this means that the colors are
  in the same order. Hence if $\mathfrak{S}_{\mathbf{z}, \lambda, w}$ is
  nonempty, $w$ does not change the order of colors corresponding to equal
  parts in the partition $\lambda$. This is the same as saying that it is the
  shortest element of its coset in $W / W_{\lambda}$.
\end{proof}

\begin{corollary}
  \label{adjacentcor}
  If $v \in \mathcal{B}_{\lambda}$ then $\omega (v)$ is the longest element of
  its coset in $W / W_{\lambda}$.
\end{corollary}

\begin{proof}
  Let $\mathfrak{s}$ be the state such that $\mathfrak{T} (\mathfrak{s})' =
  v$. Then $\mathfrak{s} \in \mathfrak{S}_{\mathbf{z}, \lambda, w}$ with $w = w_0 \omega(v)$ by
  Proposition~\ref{statedpf} and Theorem~\ref{crystalmt}. Thus, according to Proposition~\ref{adjacentcols}, $w_0 \omega (v)$ is the shortest element in its
  coset and therefore $\omega (v)$ is the longest element of its coset.
\end{proof}

\section{\label{sec:binf}Demazure atoms in $\mathcal{B}_\infty$}

Littelmann~\cite{LittelmannYT} proved the refined Demazure character formula
Theorem~\ref{demazurecrystals} using tableaux methods in many
cases. Kashiwara~{\cite{KashiwaraDemazure}} used two innovations in proving
it completely for symmetrizable Kac-Moody Lie algebras.

The first innovation in~\cite{KashiwaraDemazure} is to
prove the formula indirectly by working not with $\mathcal{B}_\lambda$
but with the infinite crystal $\mathcal{B}_\infty$ that is the
crystal base of a Verma module. Thus Theorem~\ref{demazurecrystals}
is true for $\mathcal{B}_\infty$ as well as $\mathcal{B}_\lambda$ 
meaning that we also have Demazure crystals $\mathcal{B}_\infty(w)$ for $\mathcal{B}_\infty$.
One may embed $\mathcal{B}_{\lambda}$ into $\mathcal{B}_\infty$,
and the preimage of the Demazure crystal $\mathcal{B}_\infty(w)$
in $\mathcal{B}_\infty$ is the Demazure crystal $\mathcal{B}_\lambda(w)$.
In {\cite{KashiwaraDemazure,BumpSchilling,JosephConsequences}} proofs of the
refined Demazure character formula proceed by proving a version on
$\mathcal{B}_\infty$ first.

The second innovation in~\cite{KashiwaraDemazure} is to
make use of an involution $\star$ which, as we will explain,
interchanges two natural parametrizations of the crystal
by elements of a convex cone in~$\mathbb{Z}^N$.

We will use both of these ideas from~\cite{KashiwaraDemazure}, namely to lift
the problem to $\mathcal{B}_\infty$ crystal and to exploit the properties of
the $\star$-involution, in proving Theorem~\ref{daform}. Two references
adopting a point of view similar to Kashiwara's are Bump and Schilling~\cite{BumpSchilling}
and Joseph~\cite{JosephConsequences}. Both these references treat the
Demazure character formula in the context of $\mathcal{B}_\infty$
and the $\star$-involution.

The notion of crystal Demazure atoms can be adapted to $\mathcal{B}_\infty$; we define
these to be subsets $\mathcal{B}^\circ(w)$ that are disjoint and
satisfy (\ref{eqatomic}).  By Lemma~\ref{lem:uniqueatomic} these conditions
determine the atoms, and at least for type~A, the existence of a family of
crystal Demazure atoms for $\mathcal{B}_\infty$ will be proved in Corollary~\ref{cor:binfda} in the next
section.

The characterizations of $\mathcal{B}(w)$ and $\mathcal{B}^\circ(w)$
in terms of the function $\omega$ translates readily to $\mathcal{B}_\infty$.
The $\star$-involution of $\mathcal{B}_\infty$ is not a crystal graph
automorphism, but it has other important properties.  In particular, it maps
the Demazure crystal $\mathcal{B} (w)$ into $\mathcal{B}(w^{- 1})$. So using
the $\star$-involution we are able to reformulate Theorem~\ref{daform}, or more precisely the
corresponding identity for $\mathcal{B}_\infty(w)$, as the identity
\begin{equation}
\label{daformstar}
\mathcal{B}_\infty(w^{-1})=\{v\in\mathcal{B}\,|\,w_0\omega(v^\star)\leqslant w\}.
\end{equation}
The definition of $\omega$ for $\mathcal{B}_\lambda$ was given in
terms of Gelfand-Tsetlin patterns, but it may be restated in
terms of string data (\ref{bzlpat}). As we will explain later in this section,
the $\star$-involution transforms the string data into other natural data. (See (\ref{stringbinf}).)
In Lemma~\ref{daggerlem} below we have an explicit formula for $\omega(v^\star)$ in
terms of this data. Thus (\ref{daformstar}) becomes amenable to proof. The
main details are in the proof of Lemma~\ref{omstarrec}, which contains
partial information about how $\omega(v^\star)$ changes when $f_k$
is applied to $v$. The proof of this Lemma is technical, but the starting
point is the formula (\ref{fmaxtens}) for $f_k(v)$ in terms of data that we
have in hand due to Lemma~\ref{lem:techlem}. Once Lemma~\ref{omstarrec} is
proved, we conclude this section with Lemma~\ref{lem:B-inclusion} which makes progress towards showing (\ref{daformstar}) by proving the inclusion of the left-hand side in the right-hand side.

Then, using the information that we have obtained from
the Yang-Baxter equation in Theorem~\ref{zdematoms}, we can leverage this
inclusion to prove Theorem~\ref{daform} in Section~\ref{sec:dafpro}. Note that this is a statement
about $\mathcal{B}_\lambda$, not $\mathcal{B}_\infty$. Equation
(\ref{daformstar}) is equivalent to Theorem~\ref{thm:binfdc}, which
is proved after Theorem~\ref{daform} by going back to $\mathcal{B_\infty}$.
Theorem~\ref{thm:binfdc} would of course imply Theorem~\ref{daform}, but we prove
Theorem~\ref{daform} first where we can apply Theorem~\ref{zdematoms}. Thus we go
back and forth between $\mathcal{B}_\infty$ and $\mathcal{B}_\lambda$ in order
to prove everything. Finally in Corollary~\ref{cor:binfda} we obtain
a characterization of crystal Demazure atoms in $\mathcal{B}_\infty$.
\medskip

In {\cite{KashiwaraDemazure,BumpSchilling}}, the construction of
$\mathcal{B}_\infty$ depends on the choice of a reduced decomposition of the
long Weyl group element $w_0 = s_{i_1} \cdots s_{i_N}$. A main feature of the
theory is that the crystal is independent of this choice of decomposition; to
change to another reduced decomposition one may apply piecewise linear maps to
all data.  On the other hand, Littelmann~{\cite{LittelmannCones}} showed that
one particular choice of reduced word is especially nice, and it is this
Littelmann word that is important for us. Given this choice, elements of the
crystal are parametrized by data from which we can read off the Demazure
atoms.

We recall Kashiwara's definition of $\mathcal{B}_\infty$ for an arbitrary
Cartan type before specializing to the $\GL (r)$ (Cartan type~$A_{r - 1}$)
crystal. (For further details and proofs see \cite{KashiwaraDemazure} and Chapter~12 of
{\cite{BumpSchilling}}.)

If $i \in I$, the index set for the simple reflections, let
$\mathcal{B}_i$ be the elementary crystal defined in \cite{KashiwaraDemazure}
Example~1.2.6 or {\cite{BumpSchilling}} Section~12.1.
This crystal has one element $u_i (a)$ of weight $a \alpha_i$ for every
$a \in \mathbb{Z}$ on which the crystal operators $e_i$ and $f_i$ act as $e_i(u_i(a)) = u_i(a+1)$ and $f_i(u_i(a)) = u_i(a-1)$.
Let $\mathbf{i}= (i_1, \cdots, i_N)$ be a sequence of
indices such that $w_0 = s_{i_N} \cdots s_{i_1}$ is a reduced expression of
the long Weyl group element and let
\[\mathcal{B}_{\mathbf{i}}=
\mathcal{B}_{i_1}\otimes\cdots\otimes\mathcal{B}_{i_N}\;.\]

\begin{remark}
We recall that there is a difference between notation for tensor product of
crystals between~\cite{KashiwaraDemazure} and {\cite{BumpSchilling}}. We
will follow the second reference, so to read Kashiwara or Joseph, reverse the
order of tensor products, interpreting $x \otimes y$ as $y \otimes x$.
\end{remark}

Let
$u_0 = u_{i_1} (0) \otimes \cdots \otimes u_{i_N}
(0)\in\mathcal{B}_{\mathbf{i}}$, and let $\mathfrak{C}_{\mathbf{i}}$ be the
subset of $\mathbb{Z}^N$ consisting of all elements $\mathbf{a}= (a_1, \cdots,
a_N)$ such that
\begin{equation}
  \label{kashiwaradata}
  u_{\mathbf{i}} (\mathbf{a})
  = u (\mathbf{a}) = u_{i_1} (- a_1) \otimes \cdots \otimes
  u_{i_N} (- a_N)
\end{equation}
can be obtained from $u_0$ by applying some succession of
crystal operators $f_i$. Then $\mathfrak{C}_{\mathbf{i}}$ is the set
of integer points in a convex polyhedral cone in $\mathbb{R}^N$.
We regard $\mathfrak{C}_{\mathbf{i}}$, embedded via
the map $\mathbf{a} \mapsto u (\mathbf{a})$ to be a subcrystal of
$\mathcal{B}_{\mathbf{i}}$; this requires redefining $e_i (v) = 0$ if
$\varepsilon_i (v) = 0$. With this exception, the Kashiwara operators $e_i$,
$f_i$, $\varepsilon_i$ and $\varphi_i$ are the same as for the ambient crystal
$\mathcal{B}_{\mathbf{i}}$. If $\mathbf{j}$ is another reduced expresion for
$w_0$ then there is a piecewise-linear bijection $\mathfrak{C}_{\mathbf{i}}
\longrightarrow \mathfrak{C}_{\mathbf{j}}$ that is an isomorphism of crystals;
in this sense the crystal $\mathfrak{C}_{\mathbf{i}}$ is independent of the
choice of word $\mathbf{i}$. The crystal $\mathcal{B}_\infty$ is
defined to be this crystal.

In $\mathcal{B}_\infty$ the element $u_0$ is the unique highest weight
element, and the unique element of weight $0$. If $x\in\mathcal{B}_\infty$
then, as with the finite crystals $\mathcal{B}_\lambda$,
the integer $\varepsilon_i(x)$ is nonnegative and equals the number of times
$e_i$ may be applied to $x$, i.e. $\varepsilon_i(x)=\max\{k|e_i^k(x)=0\}$.  On
the other hand $f_i(x)$ is never $0$, so $\varphi_i(x)$ has no such
interpretation. It still has meaning and the identity
$\varphi_i(x)-\varepsilon_i(x)=\langle\wt(x),\alpha_i^\vee\rangle$ holds.

Because $f_i(x)$ is never $0$, the string patterns
$\bzl_{\mathbf{i}}^{(f)}(v)$ cannot be defined for $\mathcal{B}_\infty$ since
the sequence $f_i^kv$ never terminates. However $\bzl_{\mathbf{i}}^{(e)}(v)$
can be defined by (\ref{plstringe}).
Interestingly, for each reduced word $\mathbf{i}$
representing $w_0$, the set $\{\bzl_{\mathbf{i}}^{(e)}(v) \mid
v\in\mathcal{B}_\infty\}$ coincides with the cone
$\mathfrak{C}_{\mathbf{i}}$. However the data $\mathbf{a}$ such that
(\ref{kashiwaradata}) holds is \textit{not} the string data. Rather, there is
a weight-preserving bijection $\star:\mathcal{B}_\infty\to\mathcal{B}_\infty$
of order two such that
\begin{equation}
\label{stringbinf}
\mathbf{a}=\bzl^{(e)}_{\mathbf{i}}(v^\star), \qquad
v=u_{\mathbf{i}}(\mathbf{a})\;.
\end{equation}
This is true for any reduced word $\mathbf{i}$, and $\star$ is independent of
$\mathbf{i}$. This is
Kashiwara's \textit{$\star$-involution}. See~\cite{KashiwaraDemazure},
\cite{BumpSchilling} and~\cite{JosephConsequences}. Equation
(\ref{stringbinf}) is Theorem~14.16 in~\cite{BumpSchilling}, or
see the proof of Proposition~3.2.3 in~\cite{JosephConsequences}.

Let $\lambda$ be a dominant weight. There is a crystal $\mathcal{T}_{\lambda}$
with a single element $t_{\lambda}$ of weight $\lambda$; then
$\mathcal{T}_{\lambda} \otimes \mathcal{B}_\infty$ is a crystal identical to
$\mathcal{B}_{\infty}$ except that the weights of its elements are all
shifted by $\lambda$. Thus its highest weight element is $t_{\lambda} \otimes
u_0$ with weight $\lambda$. If $\mathcal{B}_{\lambda}$ is the crystal with highest weight $\lambda$,
then $\mathcal{B}_{\lambda}$ may be embedded in $\mathcal{T}_{\lambda} \otimes
\mathcal{B}_\infty$ by mapping the highest weight vector $v_{\lambda}$ to
$t_{\lambda} \otimes u_0$. Let $\psi_\lambda:\mathcal{B}_\lambda\to\mathcal{B}_\infty$
be the map such that $v\mapsto t_\lambda\otimes\psi_\lambda(v)$ is this
embedding of crystals.

Demazure crystals are defined for $\mathcal{B}=\mathcal{B}_\infty$ as follows.
If $w=1$ then $\mathcal{B}(w)=\{u_0\}$. Then recursively: if $s_i$ is a simple
reflection such that $s_iw>w$ we define $\mathcal{B}(s_iw)$ to be the set of
all $v\in\mathcal{B}$ such that $e_i^kv\in\mathcal{B}(w)$ for some
$k\geqslant0$. Theorem~\ref{demazurecrystals}~\ref{itm:crystals} remains valid
for $\mathcal{B}_\infty$. The theory of Demazure crystals for
$\mathcal{B}_\infty$ is related to the theory for $\mathcal{B}_\lambda$
by the fact that $\mathcal{B}_\lambda(w)$ is the preimage of the corresponding
$\mathcal{B}_\infty$ Demazure crystal under the embedding of
$\mathcal{B}_\lambda$ into $\mathcal{T}_\lambda\otimes\mathcal{B}_\infty$.
See~\cite{KashiwaraDemazure} and~\cite{BumpSchilling} Chapters~12 and~13.

Now we specialize to $\GL(r)$ crystals; the Cartan type is $A_{r-1}$.
If we use either the Littelmann word (\ref{littstr}) or $\mathbf{i}'$
in (\ref{litstrp}) then the cone $\mathfrak{C}_{\mathbf{i}}$ is characterized by
the inequalities (\ref{litcone}). See~\cite{LittelmannCones} Theorem~5.1
or~\cite{wmd5book}, Proposition~2.2. Now $\psi_\lambda$ is a crystal
morphism, so if $v\in\mathcal{B}_\lambda$ then
\[\bzl_{\mathbf{i}'}^{(e)}(\psi_\lambda(v))=
\bzl_{\mathbf{i}'}^{(e)}(v).\]
Thus we may define $\omega:\mathcal{B}_\infty\to W$
by (\ref{ompdef}) and if $v\in\mathcal{B}_\lambda$ then
$\omega(\psi_\lambda(v))=\omega(v)$. Then we may define
$\mathcal{B}^\circ(w)$ by \eqref{bcircdef} also for $\mathcal{B}=\mathcal{B}_\infty$
and $\mathcal{B}_\lambda^\circ(w)$ is the preimage of
$\mathcal{B}_\infty^\circ(w)$ under the map $\psi_\lambda$.

Let $\mathbf{i}$ be as in \eqref{litstrp} so that $\mathbf{i}'=(1,2,1,3,2,1,\cdots)$. Let
\begin{equation}
\label{vddefa}
v = D_1 \otimes \cdots \otimes D_{r -
1} \in \mathcal{B}_\infty \subset \mathcal{B}_\mathbf{i},\qquad
D_i \in \mathcal{B}_{r - i} \otimes \cdots \otimes \mathcal{B}_{r - 1}.
\end{equation}
Specifically we may write
\begin{equation}
\label{vddefb}
D_i = D_i (v) = u_{r - i} (- d_{i, r - i}) \otimes \cdots \otimes u_{r - 1}
   (- d_{i, r - 1}) = \bigotimes_{j = r - i}^{r - 1} u_j (- d_{i, j}) .
\end{equation}
Remembering (\ref{stringbinf}), for $v$ to be in $\mathcal{B}_\infty$ the
entries $d_{i j} = d_{ij} (v)$ must lie in the Littelmann cone
(\ref{litcone}), which in our present notation is determined by the
inequalities
\[ d_{i, j} \geqslant d_{i, j + 1}, \qquad (r - i \leqslant j \leqslant i) .
\]
Let $c_{i, j} = c_{i, j} (v) = d_{i, j} - d_{i, j + 1} \geqslant 0$.

\begin{remark}
  \label{remconvention}Initially $d_{i, j}$ is defined if $r - i \leqslant j
  \leqslant r - 1$ but we extend this to $j = r$ with the convention that
  $d_{i, r} = 0$. Hence by this convention $c_{i, r - 1} = d_{i, r - 1}$. \
  This convention will prevent certain cases having to be treated separately.
\end{remark}

By \cite{BumpSchilling} Lemma~2.33 the function $\varphi_k$ (part of
the data defining a crystal) is given by
\begin{equation}
  \label{phimaxtens} \varphi_k (v) = \max_i (\Phi_{i, k} (v))
\end{equation}
where
\[ \Phi_{i, k} = \Phi_{i, k} (v) = \varphi_k (D_i) + \sum_{\ell < i} \langle
   \wt (D_{\ell}), \alpha_k^{\vee} \rangle \; . \]
\begin{lemma}
  \label{lem:techlem}
  Assume that $r - k \leqslant i \leqslant r - 1$. Then
  \begin{equation}
    \label{phievv} \varphi_k (D_i) = \left\{\begin{array}{ll}c_{i, k - 1} &
    \text{if $k>r-i$;}\\-d_{i,k}&\text{if $k=r-i$.}\end{array}\right.
  \end{equation}
  and
  \begin{equation}
    \label{majdiff} \Phi_{i, k} - \Phi_{i + 1, k} = c_{i, k} - c_{i + 1, k -
    1} .
  \end{equation}
\end{lemma}

\begin{proof}
  First assume that $r-k+1\leqslant i\leqslant r-1$. Then using
  Lemma~2.33 of~\cite{BumpSchilling} again,
   $\varphi_k (D_i)$ is the maximum over $r - i \leqslant j \leqslant r - 1$ of
  \[ \varphi_k (u_j (- d_{i, j})) + \Bigl\langle \sum_{r - i \leqslant \ell <
     j} - d_{i, \ell} \alpha_{\ell}, \, \alpha_k^{\vee} \Bigr\rangle . \]
  By the definition of the elementary crystal ({\cite{BumpSchilling}}
  Section~12.1) we have $\varphi_k (u_j (- d_{i, j})) = - \infty$ unless $j =
  k$, so
  \[ \varphi_k (D_i) = \varphi_k (u_k (- d_{i, k})) + \Bigl\langle \sum_{r - i
     \leqslant \ell < k} - d_{i, \ell} \alpha_{\ell}, \, \alpha_k^{\vee}
     \Bigr\rangle = - d_{i, k} + d_{i, k - 1} = c_{i, k - 1} \]
  proving (\ref{phievv}). Here we have used the fact that $\varphi_k(u_k(-a)) = -a$, as well as $\langle
  \alpha_{\ell}, \alpha_k^{\vee} \rangle = - 1$ if $l = k \pm 1$ and $2$ if $l
  = k$, and $0$ otherwise. Now
  \[ \Phi_{i, k} - \Phi_{i + 1, k} = \varphi_k (D_i) - \varphi_k (D_{i + 1}) -
     \langle \wt (D_i), \alpha_k^{\vee} \rangle \]
  and with $r - k + 1 \leqslant i \leqslant r - 1$ we have (using
  Remark~\ref{remconvention} if $k = r - 1$)
  \[ \langle \wt (D_i), \alpha_k^{\vee} \rangle = d_{i, k - 1} - 2 d_{i,
     k} + d_{i, k + 1} = c_{i, k - 1} - c_{i, k} . \]
  Combining this with (\ref{phievv}) we obtain (\ref{majdiff}).
  The case $k=r-i$ is similar, except that $d_{i,k-1}$ is
  replaced by zero where it appears.
\end{proof}

We now wish to use some nondescending products.
We will use the notation of Remark~\ref{heckecompute}. Let
\begin{equation}
  \label{omegaomg}
   \Omega_i (D_i) = \prod_{\substack{
     1 \leqslant j \leqslant i\\
     c_{i, r - 1 + j - i} = 0}} S_j\;.
\end{equation}
Define $\omega^\dagger : \mathcal{B}_\infty \to W$ by 
\begin{equation}
  \label{dagdef}
  \omega^\dagger (v) = \{\Omega_{r - 1} (D_{r - 1}) \cdots \Omega_1 (D_1)\}\;.
\end{equation}
From Remark~\ref{heckecompute} the brackets $\{\cdot\}$ here mean that
the product is taken in the degenerate Hecke algebra, then the
resulting basis vector is replaced by the corresponding Weyl
group element.

\begin{lemma} We have
  \label{daggerlem}
  \[\omega^\dagger(v)=\omega(v^\star)^{-1}.\]
\end{lemma}

\begin{proof}
  By (\ref{stringbinf}), the string pattern
  $\bzl_{\mathbf{i}'}^{(e)}(v^\star)$ is the sequence $(b_1,b_2,\ldots)$ such that
  \[v=u_{i_1'}(-b_1)\otimes u_{i_2'}(-b_2)\otimes\cdots\;\, .\]
  Put these into an array as in (\ref{bzlpat}) and circle entries as in
  Circling Rule~\ref{circlingrule}. Thus the $b_k$ are the $d_{i,j}$ in
  the order determined by (\ref{vddefa}) and (\ref{vddefb}). Since
  $c_{i,j}=d_{i,j}-d_{i,j+1}$ (with the caveat in Remark~\ref{remconvention})
  we see that if $b_k$ equals $d_{i,j}$, it is circled if and only if $c_{i,j}=0$.  
  Recall that
  \[\omega(v^\star)=\Pind(s_{j_1},\cdots,s_{j_k})=\{S_{j_1}\cdots S_{j_k}\}\]
  where $s_{j_1},s_{j_2},\cdots$ are the circled elements. Now
  $S_{j_1},S_{j_2},\cdots$ are exactly the entries that appear in the product
  (\ref{dagdef}), but they appear in reverse order; so what we
  get is $\omega(v^\star)^{-1}$.
\end{proof}

\begin{lemma}
  \label{omstarrec}
  We have either $\omega^\dagger (v) = \omega^\dagger (f_k v)$ or
  $\omega^\dagger (v) = s_k \omega^\dagger (f_k v)$.
\end{lemma}

\begin{proof}
  Let $p$ be the first value of $i$ where $\Phi_{i, k} (v)$ attains its maximum.
  By \cite{BumpSchilling} Lemma 2.33
  \begin{equation}
   \label{fmaxtens}
   f_k (v) = D_1 \otimes \cdots \otimes f_k (D_p) \otimes \cdots \otimes
     D_{r - 1} .
  \end{equation}
  Furthermore, by applying the same Lemma to $f_k(D_p)$ and using the fact
  that $\varphi_k (u_j (- d_{i, j})) = - \infty$ unless $j = k$ we have
  \begin{equation*}
    f_k(D_p) = u_{r - p} (- d_{p, r - i}) \otimes \cdots \otimes u_k(-d_{p,k} - 1) \otimes \cdots \otimes u_{r - 1}(- d_{p, r - 1})
  \end{equation*}
  meaning that $f_k$ acting on $v$ has the effect that $d_{p, k}$ is replaced by $d_{p, k} + 1$.
  
  We factor $\Omega_p (D_p) = \Omega_p' (D_p) \Omega''_p (D_p)$ where
  \[ \Omega_p' (D_p) = \prod_{\substack{1 \leqslant j \leqslant k + p - r\\
       c_{p, r - 1 + j - p} = 0}} S_j, \qquad \Omega''_p(D_p) = \prod_{\substack{
       k + p - r + 1 \leqslant j \leqslant p\\
       c_{p, r - 1 + j - p} = 0}} S_j\;.\]
  We will prove that
  \begin{equation}
    \label{omegacommute} \Omega_p' (D_p) \Omega''_p (D_p) = \Omega''_p (D_p)
    \Omega_p' (D_p), \qquad \Omega_p' (f_k D_p) \Omega''_p (f_k D_p) =
    \Omega''_p (f_k D_p) \Omega_p' (f_k D_p) .
  \end{equation}
  Indeed, every $S_j$ above with $1 \leqslant j \leqslant k + p - r$ commutes with
  every $S_{j'}$ with $k + p - r + 1 \leqslant j' \leqslant p$ with one
  possible exception: $S_{k + p - r}$ does not commute with $S_{k + p - r +
  1}$. These factors are both present if both $c_{p, k - 1} = c_{p, k} = 0$.
  Now since $i = p$ is the first value that maximizes $\Phi_{i, k}$ we have
  \begin{equation}
    \label{zpkid} 0 < \Phi_{p, k} - \Phi_{p - 1, k} = c_{p, k - 1} - c_{p - 1,
    k}
  \end{equation}
  by $\left( \ref{majdiff} \right)$. Now $c_{p - 1, k} \geqslant 0$ and so
  $c_{p, k - 1} > 0$. Hence $\Omega_p' (D_p)$ does not involve $S_{k + p -
  r}$, proving the first identity in (\ref{omegacommute}). On the other hand
  $d_{p, k} (f_k v) = d_{p, k} (v) + 1$ while $d_{p, j} (f_k v) = d_{p, j}
  (v)$ for all $j \neq k$. Therefore $c_{p, k} (f_k v) = c_{p, k} (v) + 1 > 0$
  and so $S_{k + p - r - 1}$ does not appear in $\Omega''_p (f_k D_p)$,
  proving the second identity in (\ref{omegacommute}).
  
  Now using (\ref{omegacommute}) we may rearrange the products and write
  $\omega^\dagger (v) = \{\omega_1^\dagger (v) \omega_2^\dagger (v) \omega_3^\dagger (v)\}$ where
  \[ \omega_1^\dagger (v) = \Omega_{r - 1} (D_{r - 1}) \cdots \Omega_{p + 1} (D_{p +
     1}) \Omega_p'' (D_p (v)), \]
  \[ \omega_2^\dagger (v) = \Omega_p' (D_p (v)) \Omega_{p - 1} (D_{p - 1}) \cdots
     \Omega_{r - k} (D_{r - k}), \]
  \[ \omega_3^\dagger (v) = \Omega_{r - k - 1} (D_{r - k - 1}) \cdots \Omega_1 (D_1),
  \]
  and similarly for $f_k v$. Here all factors $\Omega_i (D_i)$ with $i \neq p$
  are the same for $v$ and $f_k v$ so we omit the $v$ from the notation except
  when $i = p$. Then we trivially have that $\omega_3^\dagger(f_kv) = \omega_3^\dagger(v)$ and will show that
  \begin{equation}
    \label{firstomid} \omega_1^\dagger (f_kv) = \omega_1^\dagger (v) \quad \text{or} \quad S_k
    \omega_1^\dagger (f_k v)
  \end{equation}
  and
  \begin{equation}
    \label{secondomid} \omega_2^\dagger (f_k v) = \omega_2^\dagger (v) .
  \end{equation}
  The lemma will follow upon demonstrating these two identities.
  
  Let us prove (\ref{firstomid}). Since $c_{p, k} (f_k v) = c_{p, k} (v) + 1 > 0$ as shown above, we have that $\Omega_p'' (D_p (f_k v)) =
  \Omega''_p (f_k D_p (v)) = \Omega_p'' (D_p (v))$ unless $c_{p, k} = 0$. If
  this is true we are done, so we assume that $c_{p, k} = 0$. Then
  \[ \Omega_p'' (D_p) = S_{k + p - r + 1} \Omega_p'' (f_k D_p) . \]
  Thus what we must show is that either
  \begin{equation}
    \label{firstindc} \begin{array}{lll}
      \Omega_{r - 1} (D_{r - 1}) \cdots \Omega_{p + 1} (D_{p + 1}) S_{k + p -
      r + 1} & = & S_k \Omega_{r - 1} (D_{r - 1}) \cdots \Omega_{p + 1} (D_{p
      + 1})\\
      & \text{or} & \Omega_{r - 1} (D_{r - 1}) \cdots \Omega_{p + 1} (D_{p +
      1}) .
    \end{array}
  \end{equation}
  We will prove this, obtaining a series of inequalities along the way. First
  consider $\Omega_{p + 1} (D_{p + 1}) S_{k + p - r + 1}$. Let us argue
  that $\Omega_{p + 1} (D_{p + 1})$ involves $S_{k + p - r + 1}$. Indeed, its
  presence is conditioned on $c_{p + 1, k - 1} = 0$. Now since the first value
  where $\Phi_{i, k}$ attains its maximum is at $i = p$, we have $0 \leqslant
  \Phi_{p, k} - \Phi_{p + 1, k} = c_{p, k} - c_{p + 1, k - 1}$. Therefore
  $c_{p + 1, k - 1} \leqslant c_{p, k} = 0$, so $c_{p + 1, k - 1} = 0$. Thus
  $\Omega_{p + 1} (D_{p + 1})$ involves $S_{k + p - r + 1}$ and $c_{p + 1, k -
  1} = c_{p, k} = 0$. Now unless $c_{p + 1, k} = 0$, the product $\Omega_{p +
  1} (D_{p + 1})$ does not involve $S_{k + p - r + 2}$ and so $\Omega_{p + 1}
  (D_{p + 1}) = \cdots S_{k + p - r + 1} \cdots$, where the second ellipsis
  represents factors that all commute with $S_{k + p - r + 1}$. Therefore
  since $S_{k + p - r + 1}^2 = S_{k + p - r + 1}$ we have $\Omega_{p + 1}
  (D_{p + 1}) S_{k + p - r + 1} = \Omega_{p + 1} (D_{p + 1})$, and
  (\ref{firstindc}) is proved. This means that we may assume that $c_{p + 1,
  k} = 0$ and so $\Omega_{p + 1} (D_{p + 1}) = \cdots S_{k + p - r + 1} S_{k +
  p - r + 2} \cdots$ where again the second ellipsis represents factors that
  all commute with $S_{k + p - r + 1}$. Now we use the braid relation and
  write
  \[ \Omega_{p + 1} (D_{p + 1}) S_{k + p - r + 1} = \cdots S_{k + p - r + 1}
     S_{k + p - r + 2} \cdots S_{k + p - r + 1} = \cdots S_{k + p - r + 2}
     S_{k + p - r + 1} S_{k + p - r + 2} \cdots . \]
  The first ellipsis represents factors that commute with $S_{k + p - r + 2}$
  and so we obtain
  \[ \Omega_{p + 1} (D_{p + 1}) S_{k + p - r + 1} = S_{k + p - r + 2}
     \Omega_{p + 1} (D_{p + 1}) . \]
  We wish to repeat the process so we consider now $\Omega_{p + 2} (D_{p + 2})
  S_{k + p - r + 2}$. To continue, we need to know that $c_{p + 2, k - 1} =
  0$. Because the first value where $\Phi_{i, k}$ attains its maximum is at $i
  = p$, we have $0 \leqslant \Phi_{p, k} - \Phi_{p + 2, k} = c_{p, k} - c_{p +
  1, k - 1} + c_{p + 1, k} - c_{p + 2, k - 1}$. Since we already have $c_{p,
  k} = c_{p + 1, k - 1} = c_{p + 1, k} = 0$ we have $c_{p + 2, k - 1}
  \leqslant c_{p + 1, k} = 0$ so $c_{p + 2, k - 1} = 0$ as required. Now the
  same argument as before produces either $\Omega_{p + 2} (D_{p + 2}) S_{k + p
  - r + 2} = \Omega_{p + 2} (D_{p + 2})$, in which case we are done, or
  \[ \Omega_{p + 2} (D_{p + 2}) S_{k + p - r + 2} = S_{k + p - r + 2}
     \Omega_{p + 1} (D_{p + 1}) \]
  and the further equality $c_{p + 2, k} = 0$. Repeating this argument gives a
  succession of identities which together imply (\ref{firstindc}) and
  (\ref{firstomid}).
  
  Now let us prove (\ref{secondomid}). We recall that 
  $D_p (f_k v) = f_k D_p (v)$ differs from $D_p (v)$ in replacing $d_{p, k}$ by
  $d_{p, k} + 1$. This can change only the last factor in $\Omega'_p (D_p)$,
  and this only if $d_{p, k} = d_{p, k - 1} - 1$. Therefore we may assume that
  $c_{p, k - 1} = 1$ and $\Omega'_p (D_p (f_k v)) = \Omega' (D_p (v)) S_{k + p
    - r}$.  Therefore what we must prove is that
  \begin{equation}
  \label{secondpmid}
   S_{k + p - r} \Omega_{p - 1} (D_{p - 1}) \cdots \Omega_{r - k} (D_{r -
     k}) = \Omega_{p - 1} (D_{p - 1}) \cdots \Omega_{r - k} (D_{r - k}) .
  \end{equation}
  Thus consider $S_{k + p - r} \Omega_{p - 1} (D_{p - 1})$. We have $c_{p -
  1, k} < c_{p, k - 1} = 1$ by (\ref{zpkid}), and so $c_{p - 1, k} = 0$. This
  means that $\Omega_{p - 1} (D_{p - 1})$ has $S_{k + p - r}$ as a factor, and
  unless it also has $S_{k + p - r - 1}$ as a factor, we obtain $S_{k + p - r}
  \Omega_{p - 1} (D_{p - 1}) = \Omega_{p - 1} (D_{p - 1})$, which implies
  (\ref{secondomid}). Therefore we may assume that $\Omega_{p - 1} (D_{p -
  1})$ has $S_{k + p - r - 1}$ as a factor, which means that $c_{p - 1, k - 1}
  = 0$, which we now assume. Now we use $S_{k + p - r} \Omega_{p - 1} (D_{p -
  1}) = S_{k + p - r} \cdots S_{k + p - r - 1} S_{k + p - r} \cdots$ where the
  first ellipsis represents factors that commute with $S_{k + p - r}$ and the
  second ellipsis represents factors that commute with $S_{k + p - r - 1}$.
  Using the braid relation we obtain
  \[ S_{k + p - r} \Omega_{p - 1} (D_{p - 1}) = \Omega_{p - 1} (D_{p - 1})
     S_{k + p - r - 1} . \]
  We repeat the process. The next step is to prove that either
  \[ S_{k + p - r - 1} \Omega_{p - 2} (D_{p - 2}) = \Omega_{p - 2} (D_{p - 2})
     \qquad \text{or\qquad$\Omega_{p - 2} (D_{p - 2})$} S_{k + p - r - 2} . \]
  If $S_{k + p - r - 1} \Omega_{p - 2} (D_{p - 2}) = \Omega_{p - 2} (D_{p -
  2})$ then (\ref{secondomid}) follows and we may stop; otherwise we will
  prove the second identity together with the equation $c_{p-2,k}=c_{p-2,k-1}=0$
  that will be needed for subsequent steps. Since $i = p$ is the first value
  to maximize $\Phi_{i, k}$ we have, using (\ref{majdiff})
  \[ 0 < \Phi_{p, k} - \Phi_{p - 2, k} = \Phi_{p, k} - \Phi_{p-1, k} +
     \Phi_{p-1, k} - \Phi_{p - 2, k} = c_{p, k - 1} - c_{p - 1, k} + c_{p -
     1, k - 1} - c_{p - 2, k} . \]
  We already have $c_{p, k - 1} = 1$ while $c_{p - 1, k} = c_{p - 1, k - 1} =
  0$, so $c_{p - 2, k} = 0$. This means that $\Omega_{p - 2} (D_{p - 2})$ has
  a factor $S_{k + p - r - 1}$. Unless it also has a factor $S_{k + p - r -
  2}$ we have $S_{k + p - r - 1} \Omega_{p - 2} (D_{p - 2}) = \Omega_{p - 2}
  (D_{p - 2})$ and we are done. If it does have the factor $S_{k + p - r - 2}$
  then we have $c_{p - 2, k - 1} = 0$ and $S_{k + p - r - 1} \Omega_{p - 2}
  (D_{p - 2}) = \Omega_{p - 2} (D_{p - 2}) S_{k + p - r - 2}$ follows from the
  braid relation. Continuing this way, we obtain a sequence of
  identities $c_{p-a,k}=0$ and
  \[S_{k+p-r+1-a}\Omega_{p-a}(D_{p-a})=\Omega_{p-a}(D_{p-a})\qquad\text{or}\qquad \Omega_{p-a}(D_{p-a})S_{p+k-r-a}.\]
  If first alternative is true we may stop, since then
  (\ref{secondpmid}) is proved and we are done. Otherwise
  if the second equality is true we have also $c_{p-a,k-1}=0$, which
  is used to prove $c_{p-a-1,k}=0$ by an argument as above based
  on (\ref{majdiff}) and move to the next stage.
  Finally, with $c_{r-k,k}=0$, the last identity to be proved is 
  \[ S_1 \Omega_{r - k} (D_{r - k}) = \Omega_{r - k} (D_{r - k}), \]
  and this time there is no second alternative.
  This is true since then the first factor
  of $\Omega_{r-k}(D_{r-k})$ is $S_1$, and $S_1^2=S_1$. Now
  (\ref{secondpmid}) is proved, establishing~(\ref{secondomid}).
\end{proof}

\begin{lemma}
  \label{lem:B-inclusion}
  Let $w \in W$. Then
  \begin{equation}
    \label{daformstarineq}
    \mathcal{B}_{\infty} (w^{- 1}) \subseteq \{ v \in \mathcal{B}_{\infty} \mid w_0
    \omega (v^{\star}) \leqslant w \}
  \end{equation}
  and
  \begin{equation}
    \label{daformstarineqa}
    \mathcal{B}_{\infty} (w) \subseteq \{ v \in \mathcal{B}_{\infty} \mid w_0
  \omega (v) \leqslant w \} \, .
  \end{equation} 
\end{lemma}

\noindent
We will improve the inclusions in this Lemma later in Theorem~\ref{thm:binfdc} to equalities,
taking into account the additional information we have from Theorem~\ref{zdematoms}.

\begin{proof}[Proof of Lemma~\ref{lem:B-inclusion}]
  By \cite{KashiwaraDemazure} or \cite{BumpSchilling} Theorem~14.17,
  the $\star$-involution takes $\mathcal{B}_{\infty} (w^{- 1})$ to $\mathcal{B}_{\infty} (w)$.
  Thus (\ref{daformstarineq}) and (\ref{daformstarineqa}) are equivalent.
  Using Lemma~\ref{daggerlem} and the fact that the inverse map on $W$ preserves the Bruhat order,
  \eqref{daformstarineq} is also equivalent to
  \begin{equation}
    \label{eq:dagdemc} \mathcal{B}_{\infty} (w) \subseteq \{ v \in
    \mathcal{B}_{\infty} \mid \omega^{\dagger} (v) w_0 \leqslant w \} \, ,
  \end{equation}
  which we will now prove by induction on $\ell (w)$. If $w = 1$ then
  $\mathcal{B}_{\infty} (1) = \{ u_0 \}$, where $u_0$ is the highest weight
  vector in $\mathcal{B}_{\infty}$. For $v = u_0$ all the conditions $c_{i, r
  - 1 + j - 1} = 0$ are satsified in (\ref{omegaomg}) and it follows that
  $\omega^{\dagger} (u_0) = w_0$, so \eqref{eq:dagdemc} is satisfied in this
  case. Now assume that $\eqref{eq:dagdemc}$ is true for $w$; we
  show that if $s_i$ is a simple reflection and $s_i w > w$ then it is also true
  for $s_i w$. Now, by Theorem~\ref{demazurecrystals}~\ref{itm:crystals} for $\mathcal{B} = \mathcal{B}_\infty$, if $v \in
  \mathcal{B}_{\infty} (s_i w)$ then there is
  a $v_1 \in \mathcal{B}_{\infty} (w)$ such that $v$ and $v_1$ lie in the same
  root string. Note that Lemma~\ref{omstarrec} implies that if $v, v_1$ lie in
  the same $i$-string then either $\omega^\dagger (v_1) = \omega^\dagger (v)$ 
  or $\omega^\dagger (v_1) = s_i \omega^\dagger (v)$. Then $\omega^{\dagger} (v_1) w_0 \leqslant w$ by induction, and
  $\omega^{\dagger} (v) w_0 = \omega^{\dagger} (v_1) w_0$ or $s_i
  \omega^{\dagger} (v_1) w_0$; in either case $\omega^{\dagger} (v) w_0 \leqslant s_i w$.
\end{proof}

\section{\label{sec:dafpro}Proof of Theorem~\ref{daform}}
In this section we will prove Theorem~\ref{daform}
and its $\mathcal{B}_\infty$ analogue.

\begin{proof}[Proof of Theorem~\ref{daform}]
We consider the preimage in $\mathcal{B}_{\lambda}$ of both
sides of the identity in Lemma~\ref{lem:B-inclusion} under the map
$\psi_\lambda:\mathcal{B}_\lambda\to\mathcal{B}_\infty$ defined in
Section~\ref{sec:binf} and we obtain the inclusion of sets
\begin{equation}
  \label{oneinclusion} \mathcal{B}_{\lambda} (w) \subseteq \{ v \in
  \mathcal{B}_{\lambda}  \mid w_0 \omega (v) \leqslant w \} = \bigcup_{y \leqslant
  w} \{ v \in \mathcal{B}_{\lambda}  \mid w_0 \omega (v) = y \} .
\end{equation}
We claim that, in fact, these sets are equal, which would give us 
\eqref{eqatomic}. We caution the reader that the
Kashiwara involution $\star$ (which is not a crystal isomorphism)
does not preserve $\mathcal{B}_\lambda$ embedded in the crystal
via $\psi_\lambda$. What is true is that it maps $\mathcal{B}_\infty(w)$
into $\mathcal{B}_\infty(w^{-1})$, and the preimage of
$\mathcal{B}_\infty(w)$ in $\mathcal{B}_\lambda$ is
$\mathcal{B}_\lambda(w)$. That is all that is needed for (\ref{oneinclusion}).

Let $X$ and $Y$ be the two subsets of $\mathcal{B}_\lambda$ on the left- and
right-hand sides of (\ref{oneinclusion}). We have just shown that $X \subseteq
Y$. Now, on the one hand, we have from Theorem~\ref{demazurecrystals}~\ref{itm:character} that
\begin{equation}
  \label{eq:sumXweights}
  \sum_{\mathfrak{T} \in X} \mathbf{z}^{\wt(\mathfrak{T})} = \partial_w \mathbf{z}^\lambda \, .
\end{equation}

On the other hand, using the bijection between the crystal $\mathcal{B}_\lambda$ and the ensemble of states $\mathfrak{S}_{\mathbf{z}, \lambda}$ together with 
Theorem~\ref{crystalmt}, we have that
\begin{equation*}
  \sum_{\mathfrak{T} \in Y} \mathbf{z}^{\wt(\mathfrak{T})} := \sum_{y \leqslant w} \sum_{\substack{v \in \mathcal{B}_\lambda \\ w_0 \omega(v) = y}} \mathbf{z}^{\wt(v)} = \sum_{y \leqslant w} \sum_{s \in \mathfrak{S}_{\mathbf{z}, \lambda, y}} \mathbf{z}^{\wt(\mathfrak{T}(s)')} \, .
\end{equation*}

The Sch\"utzenberger involution satisfies the property that
$\wt(\mathfrak{T}') = w_0 \wt(\mathfrak{T})$. Using this, then \ref{itm:zweightid}
of Proposition~\ref{zschur} and then Theorem~\ref{zdematoms} we get that
\begin{equation*}
  \sum_{s \in \mathfrak{S}_{\mathbf{z}, \lambda, y}} \mathbf{z}^{\wt(\mathfrak{T}(s)')} = \sum_{s \in \mathfrak{S}_{\mathbf{z}, \lambda, y}} \mathbf{z}^{w_0 \wt(\mathfrak{T}(s))} = \mathbf{z}^{-\rho} Z(\mathfrak{S}_{\mathbf{z}, \lambda, y}) = \partial^\circ_y \mathbf{z}^\lambda \, . 
\end{equation*}
Finally by Theorem~\ref{thm:lskeys} and comparing with \eqref{eq:sumXweights}, it follows that
\begin{equation}
  \label{eq:keystep}
  \sum_{\mathfrak{T} \in Y} \mathbf{z}^{\wt(\mathfrak{T})} = \sum_{y \leqslant w}  \partial^\circ_y \mathbf{z}^\lambda = \partial_w \mathbf{z}^\lambda = \sum_{\mathfrak{T} \in X} \mathbf{z}^{\wt(\mathfrak{T})}. 
\end{equation}
Setting $\mathbf{z} = 1_T$ in the above equality shows that $X$ and $Y$ have the same
cardinality. Therefore $X = Y$.

The assertion that $w_0\omega(v)=w$ implies that $w$ is the longest element
of its coset in $W/W_\lambda$ is Corollary~\ref{adjacentcor}.
\end{proof}

Now that Theorem~\ref{daform} is proved, we have an analogous characterization of
Demazure crystals and Demazure atoms in $\mathcal{B}_\infty$.

\begin{theorem}
  \label{thm:binfdc}
  For any $w \in W$,
  \begin{equation}
    \label{binfdcid}
    \mathcal{B}_\infty(w)=\{v\in\mathcal{B}_\infty \mid \omega^\dagger(v)w_0\leqslant w\}=
    \{v\in\mathcal{B}_\infty \mid w_0\omega(v)\leqslant w\}\;.
  \end{equation}
  The map $\omega$ satisfies
  \begin{equation}
    \label{omegaid}
    w_0\omega(v)w_0=\omega^\dagger(v)=\omega(v^\star)^{-1}.
  \end{equation}
\end{theorem}

\begin{proof}
  The identities
  \[\mathcal{B}_\lambda(w)=\{v\in\mathcal{B}_\lambda \mid \omega^\dagger(v)w_0\leqslant w\}=
  \{v\in\mathcal{B}_\lambda \mid w_0\omega(v)\leqslant w\}\]
  have been proved for the finite crystal $\mathcal{B}_\lambda$,
  and since the images of $\psi_\lambda$ exhaust
  $\mathcal{B}_\infty$, (\ref{binfdcid}) follows. The identity
  (\ref{omegaid}) follows using Lemma~\ref{daggerlem}.
\end{proof}

Now the Demazure atoms in $\mathcal{B}_\infty$ may be defined as
\begin{equation}
\label{binfatoms}
\mathcal{B}^\circ_\infty(w)=\{v\in\mathcal{B}_\infty \mid \omega^\dagger(v)w_0=w\}=
    \{v\in\mathcal{B}_\infty \mid w_0\omega(v)=w\}\;.
\end{equation}

\begin{corollary}
  \label{cor:binfda}
  The subsets $\mathcal{B}^\circ_\infty(w)$ are a family of crystal Demazure atoms
  for $\mathcal{B}_\infty$.
\end{corollary}

\begin{proof}
These are obviously a family of disjoint subsets of $\mathcal{B}_\infty$
and by Theorem~\ref{thm:binfdc} they satisfy the characterizing
identity~(\ref{eqatomic}).
\end{proof}

\section{\label{sec:algo}Proof of the algorithms for computing Lascoux-Sch\"utzenberger keys}

We now prove the algorithms from Subsection~\ref{subsec:alg}. For the first algorithm, given any tableau $T \in \mathcal{B}_{\lambda}$, we compute $\omega(T')$ by means of the
definition (\ref{ompdef}). Thus we consider
$\bzl_{\mathbf{i}'}^{(e)}(T')=\bzl_{\mathbf{i}}^{(f)}(T)=(b_1,b_2,\cdots)$, 
where the Gelfand-Tsetlin pattern of $T$ is the array $(a_{ij})$ and
the $b_i$ are given by the formula (\ref{bzlfromgtp}). Then
\[\begin{array}{lcl}
b_1=0&\quad\iff\quad& a_{r,r}=a_{r-1,r},\\
b_2=b_3&\quad\iff\quad& a_{r-1,r-1}=a_{r-2,r-1},\\
b_3=0&\quad\iff\quad&a_{r-1,r}=a_{r-2,r},\\
&\vdots&\end{array}\]
and so forth. This means that the circled entries in (\ref{gamcirc}) are
the same as in (\ref{circledreflections}). Therefore the first algorithm
follows from Theorem~\ref{daform}.

We may now prove Algorithm~2. The idea is to deduce it from Algorithm~1 (which
is already proved) for the crystal $\mathcal{B}_{- w_0 \lambda}$. Now $- w_0
\lambda = (- \lambda_r, \cdots, - \lambda_1)$ is not a partition (since its
entries may be negative) but it is a dominant weight. Fortunately the facts
that we need, particularly the map to Gelfand-Tsetlin patterns and Algorithm~1,
may be extended to crystals $\mathcal{B}_{\lambda}$ where $\lambda$ is a
dominant weight by the following considerations.

If $\lambda = (\lambda_1, \cdots, \lambda_r)$ a dominant weight (that is,
$\lambda_1 \geqslant \cdots \geqslant \lambda_r$ but the entries may be
negative) then for sufficiently large $N$, $\lambda + N^r = (\lambda_1 + N,
\cdots, \lambda_r + N)$ is a partition and $\mathcal{B}_{\lambda + (N^r)}$ is
a crystal of tableaux. To put this into context, $\mathcal{B}_{\lambda}$ is
the crystal of the representation $\pi_{\lambda}$ of $\GL(r)$ with
highest weight $\lambda$, and $\mathcal{B}_{\lambda + (N^r)}$ is the crystal
of $\det^N \otimes \pi_{\lambda}$. The crystal graph of $\mathcal{B}_{\lambda
+ (N^r)}$ is isomorphic to that of $\mathcal{B}_{\lambda}$ and we may transfer
results such as Theorem~\ref{daform} from $\mathcal{B}_{\lambda + (N^r)}$ to
$\mathcal{B}_{\lambda}$.

In particular let $\mathfrak{P}_{\lambda}$ be the space of Gelfand-Tsetlin
patterns with top row $\lambda$. Let $\Gamma : \mathcal{B}_{\lambda}
\longrightarrow \mathfrak{P}_{\lambda}$ be the map defined in the introduction
for $\lambda$ a partition. If $\lambda$ is a dominant weight, then $\Gamma :
\mathcal{B}_{\lambda} \longrightarrow \mathfrak{P}_{\lambda}$ may be similarly
defined; for if $v \in \mathcal{B}_{\lambda}$ and $T$ is the corresponding
element of $\mathcal{B}_{\lambda + (N^r)}$, then $\Gamma (T)$ is defined and
we define $\Gamma (v)$ to be the Gelfand-Tsetlin pattern ottained from $\Gamma
(T)$ by subtracting $N$ from every element of $\Gamma (T)$. The map $\omega :
\mathcal{B}_{\lambda} \longrightarrow W$ is also defined and Algorithm~1 is
valid.

Now there are maps $\alpha_1, \alpha_2 : \mathfrak{P}_{\lambda}
\longrightarrow W$ corresponding to Algorithm~1 and Algorithm~2 of the
introduction. Thus if $a = (a_{i j})$ is a Gelfand-Tsetlin pattern, then for
each $(i, j)$ with $a_{i, j} = a_{i - 1, j}$ we circle the corresponding entry
in (\ref{gamcirc}) and $\alpha_1 (a)$ will be the nondecreasing product of the
circled reflection in order from bottom to top, right to left; and similarly
to compute $\alpha_2 (a)$ we circle the entries of (\ref{delcirc}) when $a_{i,
j} = a_{i - 1, j - 1}$ and take the nondecreasing product in order from bottom
to top, left to right.

There is an operation $- \text{rev}$ on Gelfand-Tsetlin patterns that maps
$\mathfrak{P}_{\lambda}$ to $\mathfrak{P}_{- w_0 \lambda}$ that negates the
entries in a pattern $a$ and mirror-reflects them from left to right, so if $r
= 3$
\[ - \text{rev} \left( \begin{array}{ccccc}
     \lambda_1 &  & \lambda_2 &  & \lambda_3\\
     & a &  & b & \\
     &  & c &  & 
   \end{array} \right) = \left( \begin{array}{ccccc}
     - \lambda_3 &  & - \lambda_2 &  & - \lambda_1\\
     & - b &  & - a & \\
     &  & - c &  & 
   \end{array} \right) . \]

As further discussed in \cite{wmd5book}, there is a map $\phi_{\lambda} : \mathcal{B}_{\lambda} \longrightarrow
\mathcal{B}_{- w_0 \lambda}$ that maps the highest weight element to the
highest weight element and has the effect that $\phi_{\lambda} (e_i v) =
e_{i'} \phi_{\lambda} (v)$, where we recall that $i' = r - i$.

\begin{proposition}
  For all $T\in\mathcal{B}_\lambda$
  \begin{equation}
    \label{omphic} \omega (\phi_{\lambda} (T)) = w_0 \omega (T) w_0^{- 1} .
  \end{equation}
\end{proposition}

\begin{proof}
  Note that $w \mapsto w_0 w w_0^{- 1}$ is the automorphism of $W$ that sends
  the simple reflection $s_i$ to $s_{i'}$. So by the definition of the
  Demazure crystals it is clear that $\phi_{\lambda} \mathcal{B}_{\lambda} (w)
  =\mathcal{B}_{- w_0 \lambda} (w_0 w w_0^{- 1}) .$ Hence $\phi_{\lambda}
  (\mathcal{B}_{\lambda}^{\circ} (w)) =\mathcal{B}_{- w_0 \lambda}^{\circ}
  (w_0 w w_0^{- 1})$. By Theorem~\ref{daform}, we may characterize $\omega
  (T)$ as the shortest Weyl group element such that $T \in
  \mathcal{B}_{\lambda}^{\circ} (w_0 \omega (T))$. Equation (\ref{omphic})
  follows.
\end{proof}

The map $\phi_{\lambda}$ intertwines the Sch{\"u}tzenberger-Lusztig
involutions $v \mapsto v'$ on $\mathcal{B}_{\lambda}$ and $\mathcal{B}_{- w_0
\lambda}$. We will denote $\phi'_{\lambda} (v) = \phi_{\lambda} (v') =
\phi_{\lambda} (v)'$. Let $\tau : W \longrightarrow W$ be conjugation by $w_0$.
We have a commutative diagram
\[
\begin{tikzcd}[column sep=huge]
  \mathcal{B}_\lambda \arrow{r}{\phi'_\lambda} \arrow{d}{\Gamma} & \mathcal{B}_{-w_0\lambda} \arrow{d}{\Gamma} \\
  \mathfrak{P}_\lambda \arrow{r}{-\text{rev}} \arrow{d}{\alpha_2} & \mathfrak{P}_{-w_0\lambda} \arrow{d}{\alpha_1} \\
  W \arrow{r}{\tau} & W
\end{tikzcd}
\]
Indeed, the top square commutes by (2.12) of {\cite{wmd5book}}, which is proved
there using the description of the Sch\"utzenberger involution on Gelfand-Tsetlin
patterns in~{\cite{BerensteinKirillov}}. The commutativity of the bottom square is clear from the
definitions of $\alpha_1$ and $\alpha_2$, bearing in mind that
$w_0 s_i w_0^{-1} = s_{i'}$ when circling \eqref{gamcirc} and \eqref{delcirc}.

We may now prove the second algorithm. If $T \in \mathcal{B}_{\lambda}$, the
commutative diagram shows that
\[ w_0 \alpha_2 (\Gamma (T)) w_0^{- 1} = \alpha_1 (\Gamma (\phi_{\lambda}
   (T)')) = \omega (\phi_{\lambda} (T)) = w_0 \omega (T) w_0^{- 1} \]
where the second step is by applying Algorithm~1 to $\phi_{\lambda} (T)' \in
\mathcal{B}_{- w_0 \lambda}$ and the last step is by (\ref{omphic}). Therefore
$\omega (T) = \alpha_2 (\Gamma (T))$, which is Algorithm~2.

\bibliographystyle{hplain} 
\bibliography{demice}

\end{document}